\documentclass[11pt, reqno]{amsart}

\usepackage{comment}
\usepackage[rgb,dvipsnames]{xcolor}
\usepackage{bm} 
\usepackage{stmaryrd}
\usepackage{amsfonts,amsmath,amsthm,amssymb}
\usepackage{latexsym}
\usepackage{amscd}
\usepackage{esint}
\usepackage{mathrsfs}
\usepackage{dsfont}
\usepackage{setspace}
\usepackage{enumitem}
\usepackage[margin=1.25in]{geometry}
\usepackage{hyperref}
\usepackage{enumitem}

\newtheorem{prop}{Proposition}[section]
\newtheorem{thm}[prop]{Theorem}
\newtheorem{lemm}[prop]{Lemma}
\newtheorem{coro}[prop]{Corollary}

\newtheorem*{claim*}{Claim}

\theoremstyle{definition}
\newtheorem{defi}[prop]{Definition}
\newtheorem{rmk}[prop]{Remark}

\newcommand{\CC}{\mathbb{C}}

\newcommand{\NN}{\mathbb{N}}

\newcommand{\RR}{\mathbb{R}}

\newcommand{\cN}{\mathcal N}

\newcommand{\cV}{\mathcal V}

\newcommand{\cX}{\mathcal X}
\newcommand{\cY}{\mathcal Y}

\newcommand{\sF}{\mathscr{F}}

\newcommand{\sU}{\mathscr{U}}

\def\bB{\mathbf{B}}
\def\be{\mathbf{e}}
\def\bP{\mathbf{P}}
\def\bT{\mathbf{T}}
\def\bM{\mathbf{M}}
\def\bL{\mathbf{L}}

\DeclareMathOperator{\tr}{tr}

\DeclareMathOperator{\im}{Im}

\DeclareMathOperator{\supp}{supp}

\DeclareMathOperator{\Div}{div}
\DeclareMathOperator{\loc}{loc}
\DeclareMathOperator{\osc}{osc}
\DeclareMathOperator{\re}{Re}

\DeclareMathOperator{\dist}{dist}

\DeclareMathOperator{\vol}{vol}

\DeclareMathOperator{\uni}{uni}
\DeclareMathOperator{\conv}{conv}
\DeclareMathOperator{\reg}{reg}
\DeclareMathOperator{\GL}{GL}
\DeclareMathOperator*{\esssup}{ess\,sup}

\newcommand{\ep}{\varepsilon}

\newcommand{\id}{\text{id}}

\newcommand{\bangle}[1]{\big\langle #1 \big\rangle}

\newcommand{\pa}[2]{\frac{\partial #1}{\partial #2}}

\newcommand{\paop}[1]{\pa{}{#1}}

\newcommand{\rom}[1]{\expandafter\romannumeral #1}
\newcommand{\Rom}[1]{\uppercase\expandafter{\romannumeral #1}}

\setlist[enumerate]{leftmargin = 2em}
\setlist[itemize]{leftmargin = 2em}
\allowdisplaybreaks

\numberwithin{equation}{section}

\title[Energy convexity and uniformity of $H$-surface flow]{Energy convexity and uniformity of the $H$-surface flow in $\mathbb{R}^{3}$ with Dirichlet boundary condition}
\author{Da Rong Cheng}
\address[D.~Cheng]{Department of Mathematics\\University of Miami\\1365 Memorial Drive\\
Coral Gables, FL 33146\\USA}
\email{darong.cheng@miami.edu}

\author{Longzhi Lin}
\address[L.~Lin]{Mathematics Department\\University of California, Santa Cruz\\1156 High Street\\
Santa Cruz, CA 95064\\USA}
\email{lzlin@ucsc.edu}

\author{Xin Zhou}
\address[X.~Zhou]{Mathematics Department\\Cornell University\\212 Garden Ave\\
Ithaca, NY 14853\\USA}
\email{xinzhou@cornell.edu}

\date{\today}

\begin{document}

\begin{abstract}    
We prove that the energy functional associated with surfaces of prescribed mean curvature in $\mathbb{R}^3$ exhibits a convexity property when restricted to maps from the unit $2$-disk having small Dirichlet energy and a fixed boundary value, provided that the $3$-form $H \cdot \vol_{\RR^3}$ has a primitive satisfying certain bounds.

Under milder assumptions on $H:\mathbb{R}^3 \to \mathbb{R}$, we show that an analogous convexity estimate holds along the heat flow of the functional (the $H$-surface flow) when the initial Dirichlet energy is sufficiently small. This is done first for classical solutions, and then extended by approximation to weak solutions using a quantitative uniqueness result adapted from previous work on the harmonic map heat flow. As a consequence of the convexity estimate, we show that the $H$-surface flow with small-energy initial map of class $C^0 \cap W^{1, 2}$ on the unit $2$-disk converges uniformly to a unique limit at infinite time, which solves the corresponding stationary problem (the $H$-surface system) with the same boundary value.
\end{abstract}

\maketitle 

\section{Introduction}
One of the central questions in geometric analysis is the study of the \emph{$H$-surface system}, namely
\begin{equation}\label{eq:H-system}
    \Delta u = 2H (u) \, u_{x^1} \times u_{x^2} \quad \text{on } \bB,
\end{equation}
where $\bB$ is the open unit disk in $\RR^2$, $H : \RR^3 \to \RR$ is a given function, and $H(u)$ stands for $H \circ u$. Also, the subscripts denote partial derivatives, and $\times$ is the cross product in $\RR^3$. Given any solution $u$ to~\eqref{eq:H-system} which is \emph{weakly conformal} in the sense that 
\[
|u_{x^1}| = |u_{x^2}|\quad\text{and}\quad u_{x^1}\cdot u_{x^2}=0,
\]
it is classically known that the mean curvature at $u(x)$ of its image, which we take to be the average, as opposed to the sum, of the principal curvatures, is given by $H(u(x))$, provided $\nabla u(x) \neq 0$. The so-called \emph{parametric} approach to the Plateau problem for disk-type surfaces with prescribed mean curvature $H$ then consists of solving~\eqref{eq:H-system} subject to the requirement that $u$ be weakly conformal, and that its restriction to $\partial \bB$ traces out the given boundary curve monotonically. Examples of important work along this line include the existence results of Heinz \cite{Heinz1954}, Wente \cite{Wente1969}, Hildebrandt \cite{Hildebrandt1970-2,Hildebrandt1969,Hildebrandt1970}, Gulliver--Spruck \cite{GulliverSpruck1971, GulliverSpruck1972}, and Steffen~\cite{Steffen1976-1,Steffen1976-2}, as well as the non-uniqueness results of Brezis--Coron~\cite{BrezisCoron1984}, Struwe~\cite{Struwe1986,Struwe1990-1}, and Bethuel--Rey~\cite{BethuelRey1994}. 

In tackling the Plateau problem, it is typical to begin by studying~\eqref{eq:H-system} coupled with the Dirichlet condition 
\begin{equation}\label{eq:Dirichlet-condition}
u = u_0 \quad \text{on }\partial \bB,
\end{equation}
where $u_0:\overline{\bB}\to \RR^3$ is a given function. 
From a variational perspective, the Dirichlet problem \eqref{eq:H-system}-\eqref{eq:Dirichlet-condition} arises as the Euler-Lagrange equation of the functional 
\begin{equation}\label{eq:energy-functional}
    E_H(u, u_0) := D(u) + V_H(u, u_0),
\end{equation}
subject to $u - u_0 = 0$ on $\partial \bB$, where 
\[
D(u) = \frac{1}{2}\int_{\bB} |\nabla u|^2 dx,
\]
is the \textit{Dirichlet energy}, and $V_H(u, u_0)$ is intuitively the volume, weighted by $2H(\cdot)$, of the oriented region enclosed by $u$ and $u_0$. We recall in Section~\ref{sec:prelim} how $V_H$ is defined, following the presentation by Duzaar--Steffen~\cite{DuzaarSteffen1999} and B\"ogelein--Duzaar--Scheven~\cite{BogeleinDuzaarScheven2013}. For now, we only note that the definition, which makes use of integral currents, goes back to the work of Steffen~\cite{Steffen1976-1,Steffen1976-2}, and underlies his variational approach towards~\eqref{eq:H-system} which unified many of the earlier existence results. An important property of $V_H$ is the isoperimetric inequality:
\[
|V_H(u, v)| \leq c\,\|H\|_{\infty; \RR^3} \big( D(u) + D(v) \big)^{\frac{3}{2}},
\]
whenever $u = v$ on $\partial \bB$, where $c$ is a universal constant. Also, if $H \in C^{1}(\RR^3; \RR)$, which we assume throughout much of this paper, and if $Q:\RR^3 \to \RR^3$ is a $C^{1}$-vector field satisfying 
\begin{equation}\label{eq:Q-H-relation}
    \Div Q(\cdot)= 2 H(\cdot),
\end{equation} 
or equivalently $d(\iota_{Q}\vol_{\RR^3}) = 2H \cdot \vol_{\RR^3}$, where $\iota$ is the interior product, then there holds
\[
V_H(u, v) = \int_{\bB}Q(u) \cdot u_{x^1}\times u_{x^2}\, dx - \int_{\bB}Q(v) \cdot v_{x^1}\times v_{x^2}\, dx.
\]

\vskip 2pt 

In this paper we are concerned with the Dirichlet problem \eqref{eq:H-system}-\eqref{eq:Dirichlet-condition}, as well as its parabolic counterpart (see~\eqref{eq:H-flow}). The main results are stated as Theorems~\ref{thm:uniqueness},~\ref{thm:H-convexity}, and~\ref{thm:uniformity} below. Given $u \in W^{1, 2}(\bB;\RR^3)$, we say it is a weak solution to~\eqref{eq:H-system} if it satisfies
\begin{equation}\label{eq:H-sys-weak}
\int_{\bB} \bangle{\nabla u, \nabla \zeta}\, dx = -2\int_{\bB}H(u)\, u_{x^1} \times u_{x^2} \cdot \zeta\, dx,
\end{equation}
for all $\zeta \in W^{1, 2}_0 \cap L^{\infty}(\bB; \RR^3)$. It is a theorem of Bethuel~\cite{Bethuel1992} that when $H \in C^1(\RR^3; \RR)$, weak solutions belong to $C^{2, \alpha}_{\loc}(\bB)$ for any $\alpha \in (0, 1)$. See also the alternative proof by Strzelecki~\cite{Strzelecki2003}. For regularity results under much weaker assumptions on $H(\cdot)$, we refer the reader to the works of Rivi\'ere~\cite{Riviere2007} and Schikorra~\cite{Schikorra2010}. Our first result is a uniqueness statement.

\begin{thm}\label{thm:uniqueness}
Suppose $H \in C^{1}(\RR^3; \RR)$ is such that 
\[
\|H\|_{\infty; \RR^3} \leq \Lambda_0, \quad \|\nabla H\|_{\infty; \RR^3} \leq\Lambda_1.
\]
There exists $\ep_{\uni} \in (0, 1)$, depending only on $\Lambda_0$ and $\Lambda_1$, such that if $u, v \in W^{1, 2}(\bB; \RR^3)$ are weak solutions of~\eqref{eq:H-system} satisfying $u - v \in W^{1, 2}_{0}(\bB; \RR^3)$, and if $D(u) + D(v) \leq \ep_{\uni}$, then $u = v$ on $\bB$.
\end{thm}
\begin{rmk}\label{rmk:existence-of-small-solution}
The existence of solutions with small energy is addressed by~\cite[Theorem 2.2]{Steffen1976-2}. Alternatively, Theorem~\ref{thm:uniformity} below implies that if $u_0 \in C^0(\overline{\bB}) \cap W^{1, 2}(\bB)$, and if $D(u_0)$ is sufficiently small depending on $\Lambda_0$ and $\Lambda_1$, then \eqref{eq:H-system}-\eqref{eq:Dirichlet-condition} has a solution in $C^0(\overline{\bB}) \cap W^{1, 2}(\bB) \cap C^{2}_{\loc}(\bB)$ with energy at most $2D(u_0)$. 
\end{rmk}

The proof of Theorem~\ref{thm:uniqueness} uses a well-known argument, the key step of which consists of applying in succession the Hardy inequality in $W^{1, 2}_0(\bB)$, and standard gradient estimates for small-energy solutions of~\eqref{eq:H-system}. The latter bounds $(1 - |x|)|\nabla u|$ and $(1- |x|)|\nabla v|$ in $L^{\infty}$, while the former bounds $(1 - |x|)^{-1} |u - v|$ in $L^2$. This combination of estimates appears already in previous works such as~\cite{LuWang2012}, \cite{Huang-Wang2016}, and \cite{Lin-Sun-Zhou2020} on harmonic maps and the associated heat flow, and is used repeatedly throughout this paper to control quantities of the form $\big(H(u)\, u_{x^1} \times u_{x^1} - H(v)\, v_{x^1}\times v_{x^2}\big)\cdot (u-v)$, among other things.
\vskip 1pt
The uniqueness of small-energy solutions to \eqref{eq:H-system}-\eqref{eq:Dirichlet-condition} in the case of constant $H(\cdot)$ was proved by Wente~\cite[Theorem 4.4]{Wente1969} as a consequence of a strict convexity property of $E_H$~\cite[Lemma 4.4]{Wente1969}. In addition, uniqueness results are also available for solutions with small oscillation~\cite{Jager-Kaul1979}, and for those that minimize the $E_{H}$-functional in a suitable sense~\cite[Remark 4]{BrezisCoron1984}; see also~\cite[Corollary IV.1.3]{Struwe1988}. In general, however, uniqueness fails for \eqref{eq:H-system}-\eqref{eq:Dirichlet-condition} without further restrictions, as shown by Struwe~\cite{Struwe85,Struwe1986} and Brezis--Coron~\cite{BrezisCoron1984} as a step in their resolution, independently, of the Rellich conjecture on the Plateau problem for constant mean curvature disks. Similar non-uniqueness results were obtained for instance in~\cite{Struwe1990-1,WangGuoFang1992,LamiDozo-Mariani1993b,Jakobowsky1994,BethuelRey1994} for certain classes of non-constant $H(\cdot)$ depending on the choice of $u_0$. In a different direction, we mention that there are examples of $H(\cdot)$, which can be made arbitrarily close to a constant, such that \eqref{eq:H-system}-\eqref{eq:Dirichlet-condition} admits no ``mountain pass'' solutions, so to speak, once $u_0$ is sufficiently small~\cite[Theorem 1.2]{CaldiroliMusina2006}.

\vskip 2pt
Our next result is a convexity estimate for the $E_H$-functional. It generalizes the result of Wente mentioned above, and is also analogous to the convexity estimate for the Dirichlet energy of mappings into closed manifolds established by Colding and Minicozzi~\cite[Theorem 3.1]{Colding-Minicozzi08b} as a central ingredient in their \emph{harmonic replacement} approach to the min-max construction of minimal $2$-spheres. (See~\cite{Zhou10,Zhou17b,Lin-Sun-Zhou2020,Laurain-Petrides2019} for examples of other existence results for minimal surfaces obtained via harmonic replacement. Convexity estimates for functionals other than the Dirichlet energy can be found for instance in~\cite{Berchenko-Kogan2018, LaurainLin2021, Jost-Zhu2023, LinZhu2025}.) 
\begin{thm}\label{thm:H-convexity}
Suppose $H \in C^{1}(\RR^3; \RR)$, and that $Q:\RR^3 \to \RR^3$ is a $C^{2}$-vector field related to $H$ by~\eqref{eq:Q-H-relation}, satisfying in addition that
\begin{equation}\label{eq:bounds-on-C2-vector-field}
\|DQ\|_{\infty; \RR^3} + \|D^{2}Q\|_{\infty; \RR^3} \leq \Lambda'.
\end{equation}
Also, let $\Lambda_0$ be an upper bound for $\|H\|_{\infty; \RR^3}$. Then there exists $\ep_{\conv} >0$, depending only on $\Lambda_0$ and $\Lambda'$, such that given a weak solution $u \in W^{1, 2} \cap L^{\infty}(\bB; \RR^3)$ of~\eqref{eq:H-system} satisfying $D(u) \leq \ep_{\conv}$, we have
\begin{equation}\label{eq:H-convexity}
 \frac{1}{4}\int_{\bB}|\nabla v - \nabla u|^2 \,dx \,\leq \,D(v) - D(u) + V_{H}(v, u),
\end{equation} 
 for all $v \in W^{1, 2} \cap L^{\infty}(\bB; \RR^3)$ such that $D(v) \leq \ep_{\conv}$ and $u - v \in W^{1, 2}_{0}(\bB; \RR^3)$.
\end{thm}
\begin{rmk}
Mainly because of the rather strong assumption~\eqref{eq:bounds-on-C2-vector-field}, we do not recover Theorem~\ref{thm:uniqueness} in general by taking both maps in Theorem~\ref{thm:H-convexity} to be weak solutions. 
\end{rmk}
Starting with the identity at the end of~\cite[page 2579]{Colding-Minicozzi08b}, the proof of Theorem~\ref{thm:H-convexity} boils down to showing that 
\begin{equation}\label{eq:volume-expansion}
V_H(v, u) - 2\int_{\bB}H(u)\, (v - u) \cdot u_{x^1} \times u_{x^2}\, dx
\end{equation}
is small compared to $\int_{\bB}|\nabla v - \nabla u|^2\, dx$. The vector field $Q$ in the assumption allows us to achieve this by rewriting the above difference as a sum of three integrals (see~\eqref{eq:Ef-difference} and the expressions that follow), the first two of which, thanks to the small-energy assumption on $u$, are amenable to the combination of estimates mentioned in the paragraph after Remark~\ref{rmk:existence-of-small-solution}, while the third has integrand $Q(v) \cdot (v - u)_{x^1} \times (v - u)_{x^2}$, which we treat by solving
\[
\Delta \psi = (v - u)_{x^1} \times (v - u)_{x^2}
\]
subject to homogeneous Neumann boundary condition, and using a generalization due to Da Lio--Palmurella--Rivi\'ere~\cite[Lemma A.6]{DaLioPalmurellaRiviere2020} of Wente's famous inequality~\cite[Lemma A.1]{BrezisCoron1984}. It is in this last step that we require the map $v$ to also have small energy. Also, the assumption that $v = u$ on $\partial \bB$ is important both to solving the above boundary value problem, and to the validity of a Wente-type estimate~\cite{DaLioPalmurella2017,Hirsch2019}. Throughout the proof of Theorem~\ref{thm:H-convexity}, we use in an essential way the bound~\eqref{eq:bounds-on-C2-vector-field}, which amounts to asking the $3$-form $2H \cdot \vol_{\RR^3}$ to have a primitive on $\RR^3$ satisfying certain estimates. As such, even when $H(\cdot)$ is constant, the proof of Theorem~\ref{thm:H-convexity} does not seem to generalize in any straightforward manner to other target manifolds, or to the case of partially free boundary conditions, for that matter. This is a problem we wish to address in a future work.

\bigskip
Next, we switch gears and discuss some relevant background for our third main result (Theorem~\ref{thm:uniformity}), which concerns the parabolic counterpart of~ \eqref{eq:H-system}-\eqref{eq:Dirichlet-condition}, namely the following Cauchy-Dirichlet problem:
\begin{equation}\label{eq:H-flow}
\left\{
\begin{aligned}
u_{t} - \Delta u &= -2 H(u)\, u_{x^1}\times u_{x^2}  &&\quad \text{on } \bB \times [0,T)\,,\\
u(\cdot, 0) &= u_0 &&\quad \text{on } \bB \,,\\
u(x, t) &= \chi(x) &&\quad \text{for } (x, t)\in \partial \bB \times [0,T) \,,
\end{aligned}
\right.
\end{equation}
where the initial data $u_0:\bB \to \RR^3$ and boundary data $\chi:\partial \bB \to \RR^3$ satisfy the usual compatibility condition $u_0|_{\partial \bB} = \chi$. The differential equation in~\eqref{eq:H-flow} is sometimes referred to as the \emph{$H$-surface flow}.

As a tool in his pioneering construction of free boundary constant mean curvature disks, Struwe~\cite{Struwe88} performed a rather comprehensive analysis of the system~\eqref{eq:H-flow} when $H(\cdot)$ is constant and the third equation is replaced by a geometric orthogonality condition. 
In \cite{Rey1991}, Rey studied~\eqref{eq:H-flow} assuming that $H(\cdot) \in C^{1}\cap L^\infty(\mathbb{R}^3, \mathbb{R})$, that $u_0 \in W^{1,2}\cap L^{\infty} (\bB, \mathbb{R}^3)$, and that $\chi \in C^{2,\gamma}(\partial \bB, \mathbb{R}^3)$.
Under a Hildebrandt-type condition:
\begin{equation}\label{HSmall}
    \|H\|_{\infty; \RR^3}\cdot \|u_0\|_{\infty; \bB} <1\,,
\end{equation} 
he obtained a solution in $C^{2+\gamma, 1+\frac{\gamma}{2}}(\overline \bB \times (0,\infty); \mathbb{R}^3)$ which converges, along a sequence $t_k \to \infty$, to a limiting map that satisfies~\eqref{eq:H-system}-\eqref{eq:Dirichlet-condition}.

Next, taking $H(\cdot)$ to be bounded and Lipschitz, and $\chi$ to be in $W^{\frac{3}{2},2}(\partial \bB;\RR^3)$, but without a smallness condition of the form \eqref{HSmall}, Chen and Levine \cite{ChenLevine2002} constructed a regular solution to \eqref{eq:H-flow} which exists up until the first occurrence of Dirichlet energy concentration, and, provided that $t\mapsto D(u(\cdot, t))$ is non-increasing, extends to a global weak solution with only finitely many singularities. However, unlike the situation with the harmonic map flow, it is not clear whether the Dirichlet energy is indeed monotone along the $H$-surface flow, making the study of the latter more subtle. We remark also that, similar to~\cite{Rey1991}, the construction in~\cite{ChenLevine2002} proceeds by first solving~\eqref{eq:H-flow} with $u_0$ replaced by a sequence of smooth approximations, and then establishing a priori bounds and uniqueness of solutions within the class
\begin{equation*}
V^{T}:= \big\{u \in C^0([0, T]; W^{1, 2}(\bB, \mathbb{R}^3)) \ |\ |\nabla^2 u|,\ |u_{t}| \in L^2(\bB \times (0, T))\big\},
\end{equation*}
in order to pass to the limit, and extend the resulting solution up to the maximal time of existence. Still assuming that $H(\cdot)$ is bounded and Lipschitz, but working with the wider class $W^{1,2}(\bB \times (0, T); \mathbb{R}^3)$, Wang~\cite{Wang1999} proved a partial regularity result for weak solutions to the $H$-surface flow. 

A general existence result for global weak solutions was more recently obtained by B\"{o}gelein, Duzaar, and Scheven~\cite{BogeleinDuzaarScheven2013}, assuming only that $H(\cdot)$ is bounded, continuous, and satisfies an isoperimetric condition in the style of~\cite{Steffen1976-1}. The latter also determines an energy threshold imposed on $u_0$, which otherwise is only required to lie in $W^{1, 2}(\bB)$. By a time-discretization scheme coupled with variational methods, they produced a weak solution to~\eqref{eq:H-flow} of class $C^0([0, \infty); L^2(\bB)) \cap L^{\infty}((0, \infty); W^{1, 2}(\bB))$ having time derivative in $L^2(\bB \times (0, \infty))$, and showed that it has a weak sequential limit at infinite time that solves~\eqref{eq:H-system}-\eqref{eq:Dirichlet-condition}. Under similar assumptions, the same authors proved in a follow-up work~\cite{BogeleinDuzaarScheven2015} that if in addition $u_0 \in W^{2, q}(\bB) + W^{1, 2}_{0}(\bB)$ for some $q > 2$, then the solution obtained in~\cite{BogeleinDuzaarScheven2013} has H\"older continuous gradient up to the boundary, away from a finite set of singular time slices which can be ruled out when the initial energy is small.

We now state our result on the $H$-surface flow,  which is analogous to the work of the second named author on the harmonic map flow with Dirichlet boundary condition and small initial energy~\cite{Lin2013}. Our notion of weak solution for~\eqref{eq:H-flow} is the same as the one in~\cite{BogeleinDuzaarScheven2013}. We recall the definition at the start of Section~\ref{sec:existence-uniformity}.
\begin{thm}\label{thm:uniformity}
Suppose $H \in C^{1}(\RR^3; \RR)$ is such that 
\[
\|H\|_{\infty; \RR^3} \leq \Lambda_0, \quad \|\nabla H\|_{\infty; \RR^3} \leq\Lambda_1.
\]
There is a constant $\ep  = \ep(\Lambda_0, \Lambda_1) \in (0, 1)$ with the property that for all $u_0 \in C^0(\overline{\bB}; \RR^3) \cap W^{1, 2}(\bB;\RR^3)$ satisfying $D(u_0) < \ep$, there exists a unique weak solution in $\cap_{T>0} W^{1,2}(\bB \times (0,T); \mathbb{R}^3)$ to the Cauchy-Dirichlet problem \eqref{eq:H-flow} subject to the requirement that 
\begin{equation}\label{eq:continued-small-energy}
D(u(\cdot, t)) \leq 2\ep,\quad \text{for a.e. }t > 0.
\end{equation}
This solution belongs to $C_{\loc}^{2 + \mu, 1 + \frac{\mu}{2}}(\bB \times (0, \infty); \RR^3)$ for all $\mu \in (0, 1)$, and extends continuously to $\overline{\bB} \times (0, \infty)$. Also, the map $t \mapsto u(\cdot, t)$ is continuous from $(0, \infty)$ to $W^{1, 2}(\bB; \RR^3)$, and we have for all $t \geq s \geq 1$ that
\begin{equation}\label{eq:convexity-for-statement}
\frac{1}{8}\int_{\bB}|\nabla u(\cdot, s) - \nabla u(\cdot, t)|^2\, dx \leq E_H(u(\cdot, s), u_0) - E_H (u(\cdot, t), u_0)\,.
\end{equation}
Finally, there exists a solution $v\in C^0(\overline{\bB};\RR^3) \cap W^{1, 2}(\bB;\RR^3) \cap C^{2}_{\loc}(\bB;\RR^3)$ to 
\begin{equation*}
\left\{
\begin{array}{ll}
\Delta v = 2H(v)\,v_{x^1} \times v_{x^2} & \text{ in }\bB,\\
v = u_0  & \text{ on }\partial \bB,
\end{array}
\right.
\end{equation*}
satisfying the energy bound $D(v) \leq 2D(u_0)$, such that as $t \to \infty$, the slices $u(\cdot, t)$ converge to $v$ in both the $W^{1,2}$ and $C^0$-norms on $\bB$, and also in $C^2$ on compact subsets of $\bB$.
\end{thm}

We note that the general existence result in~\cite{BogeleinDuzaarScheven2013} applies as soon as $\sqrt{D(u_0)} \cdot \|H\|_{\infty}$ is below an explicit constant (see specifically Theorem 9.1 therein). Nonetheless, we take a different route towards constructing a solution in order to have the convexity estimate~\eqref{eq:convexity-for-statement} built in, which plays an important role in the $C^0 \cap W^{1, 2}$ convergence at infinite time. We also emphasize that the latter occurs without having to pick a sequence of time slices. Still another aspect of Theorem~\ref{thm:uniformity} that we want to stress is the uniqueness of weak solutions subject to~\eqref{eq:continued-small-energy}, which is inspired by a similar result for the harmonic map flow proved by Wang~\cite[Theorem 1.1]{LuWang2012}. The proof of Theorem~\ref{thm:uniformity} has three main steps, each making heavy use of the $C^1$-assumption on $H(\cdot)$:

\vskip 1mm
\begin{enumerate}
\item[(i)] We first work with the case where $u_0 \in C^{2, \alpha}(\overline{\bB})$. The existence of a regular solution up to a maximal time is guaranteed by the results in~\cite{ChenLevine2002}. Then, using the isoperimetric inequality, and the monotonicity formula for $t \mapsto E_H(u(\cdot, t), u_0)$, we show, similar to~\cite[Remark 8.2]{BogeleinDuzaarScheven2013}, that $D(u(\cdot, t))$ and the integral of $|u_{t}|^2$ with respect to $(x, t)$ remain below twice the initial energy, provided the latter is sufficiently small. These bounds, along with standard a priori estimates for~\eqref{eq:H-flow}, and a generalization~\cite[Lemma 4.2]{WangGuoFang1992} of Wente's uniqueness result for~\eqref{eq:H-system} subject to constant boundary data~\cite{Wente1975}, allows us to rule out finite-time extinction of the flow, by a well-known contradiction argument that involves rescaling the solution.
\vskip 1mm
\item[(ii)] Using interior gradient estimates and the Hardy inequality in $W^{1,2}_{0}(\bB)$, as well as Wente-type estimates for equations involving determinant terms, up to decreasing the smallness threshold for the initial energy, we establish a convexity estimate for the $E_{H}$-functional along the flow (Proposition~\ref{prop:flow-convexity}). The statement is analogous to~\eqref{eq:H-convexity}. However, in dealing with the counterpart of~\eqref{eq:volume-expansion}, we express the enclosed volume using a homotopy formula, instead of a primitive of $2H \cdot \vol_{\RR^3}$, and take advantage of the fact that, in the flow setting, we have gradient estimates available for each time slice, whereas in Theorem~\ref{thm:H-convexity}, only one of the two maps are guaranteed to satisfy such estimates. This is the main reason that we are able to obtain the flow version of the convexity estimate under weaker assumptions on $H(\cdot)$ than in Theorem~\ref{thm:H-convexity}.
\vskip 1mm
\item[(iii)] Adapting an argument due to Wang~\cite[Lemma 3.2]{LuWang2012}, we prove a quantitative uniqueness result for weak solutions of~\eqref{eq:H-flow} under a small-energy condition of the form~\eqref{eq:continued-small-energy}, which allows us to construct a flow out of $C^0\cap W^{1, 2}$ initial data by approximation as in~\cite{ChenLevine2002}, with the class $V^{T}$ replaced by $W^{1, 2}(\bB\times (0, T); \RR^3)$. The convexity estimate survives this approximation process, which implies the asserted strong $W^{1, 2}$-convergence when combined with a lower bound on $E_H(u(\cdot, t), u_0)$ coming from the isoperimetric inequality. We then use standard interior estimates, and the integrability of $|u_t|^2$ on $\bB \times (0, \infty)$, to locally upgrade the convergence to $C^2$, and show that the limit satisfies the time-independent equation~\eqref{eq:H-system}. Finally, an adaptation of the work of Qing~\cite{Qing1993} on the boundary regularity of weakly harmonic maps gives uniform convergence on all of $\bB$, based on the previous two modes of convergence.
\end{enumerate}

The remainder of this paper is organized as follows. In Section~\ref{sec:prelim}, we recall the definition of the enclosed volume $V_H$ and some of its basic properties, following~\cite{DuzaarSteffen1999} and~\cite{BogeleinDuzaarScheven2013}. Section~\ref{sec:convexity} contains the proof of Theorems~\ref{thm:uniqueness} and~\ref{thm:H-convexity}. In Section~\ref{sec:regular-solutions}, we study the $H$-surface flow~\eqref{eq:H-flow} in the case where $u_0 \in C^{2, \alpha}(\overline{\bB};\RR^3)$, show that a solution exists in $C^{2 + \alpha, 1 + \frac{\alpha}{2}}(\overline{\bB} \times [0, \infty); \RR^3)$ provided $D(u_0)$ is small enough, and derive a number of other estimates along the flow. In Section~\ref{sec:existence-uniformity}, we turn to weak solutions and complete the proof of Theorem~\ref{thm:uniformity}. The appendices contain some standard analytical tools that we use throughout the paper.
\subsection*{Acknowledgement} 
The research of L. Lin is partially supported by a UCSC research grant. X. Zhou acknowledges the support by NSF grant DMS-1945178, DMS-2506717 and a grant from the Simons Foundation.

\section{Preliminaries on the enclosed volume}\label{sec:prelim}

Much of the material collected in this section can be found, in one form or another, in~\cite[Section 3]{BogeleinDuzaarScheven2013},~\cite[Section 3]{DuzaarSteffen1999}, or~\cite[Sections 2 and 3]{Steffen1976-1}. Also, a standard reference for geometric measure theory is~\cite{LeonGMT}.

Given $u \in W^{1, 2}(\bB; \RR^3)$, we define, for any smooth, compactly supported $2$-form $\omega$ on $\RR^3$, 
\begin{equation}\label{eq:Ju-definition}
J_{u} (\omega) := \int_{\bB} \omega_{u}(u_{x^1}, u_{x^2}) dx^1 \wedge dx^2,
\end{equation}
where by $\omega_{u}$ we mean $\omega \circ u$. The pointwise bound
\[
\big|\omega_{u}(u_{x^1}, u_{x^2})\big| \leq \|\omega\|_{\infty}\cdot  |u_{x^1} \wedge u_{x^2}|
\]
implies that $J_{u}$ is a $2$-current in the sense of being continuous as a linear functional on the space of compactly supported smooth $2$-forms, and that its mass satisfies 
\begin{equation}\label{eq:Ju-mass-bound}
\bM (J_{u}) \leq\int_{\bB} \left |u_{x^1} \wedge u_{x^2} \right|dx \leq D(u)\,.
\end{equation}
If $(u_n)$ is a sequence in $W^{1, 2}(\bB;\RR^3)$ such that $\|u_n - u\|_{1, 2; \bB} \to 0$, and that $u_n \to u$ almost everywhere on $\bB$, then for any $\omega$ as above we have
\begin{equation}\label{eq:continuity-of-pushforward-operation}
\begin{split}
\big|J_{u_n}(\omega) - J_u(\omega)\big| \leq\ & \int_{\bB} |\omega_{u_n} - \omega_{u}||\nabla u|^2 +\int_{\bB} |w_{u_n}|(|\nabla u| + |\nabla u_n|)|\nabla u - \nabla u_n|,
\end{split}
\end{equation}
and the integrals converge to $0$ as $n \to \infty$ by the dominated convergence theorem and, respectively, H\"older's inequality. 

\begin{lemm}\label{lemm:I-vu-construction}
Suppose $u, v \in L^{\infty} \cap W^{1, 2}(\bB; \RR^3)$ and that $v - u \in W^{1, 2}_0(\bB)$. Then there exists a unique integer multiplicity $3$-current $I_{v, u}$ on $\RR^3$ such that $\partial I_{v, u} = J_v - J_u$ and that $\bM(I_{v, u}) < \infty$. This $I_{v, u}$ has the following additional properties.
\vskip 1mm
\begin{enumerate}
\item[(a)] There holds the mass bound
\begin{equation}\label{eq:isoperimetric-ineq}
\bM(I_{v, u}) \leq c \cdot \big(D(u) + D(v)\big)^{\frac{3}{2}},
\end{equation}
where $c$ is a universal constant.
\vskip 1mm
\item[(b)] Letting $R_0 = \max\{\|u\|_{\infty}, \|v\|_{\infty}\}$ and $E_0 = \sqrt{D(u) + D(v)}$, we have
\begin{equation}\label{eq:I-vu-support-bound}
\supp(I_{v, u}) \subset \overline{B_{R_0 + cE_0}},
\end{equation}
where again $c$ is a universal constant.
\vskip 1mm
\item[(c)] For all smooth, compactly supported $3$-form $\alpha$, we have
\[
I_{v, u}(\alpha) = \int_{0}^{1}\int_{\bB}\alpha_{sv + (1-s)u}\big(v - u, sv_{x^1} + (1-s)u_{x^1}, sv_{x^2} + (1-s)u_{x^2}\big)\, dxds.
\]
\end{enumerate}
\end{lemm}
\begin{proof}
For the uniqueness assertion, suppose $S_1$ and $S_2$ are integer multiplicity $3$-currents such that $\partial S_i = J_{v} - J_{u}$ and $\bM(S_i) < \infty$ for $i = 1, 2$. Then by the constancy theorem~\cite[Theorem 26.27]{LeonGMT} there exists $a \in \RR$ such that 
\[
\big|a\int_{\RR^3}\varphi\big| = \big|S_1(\varphi) - S_2(\varphi)\big| \leq (\bM(S_1)+\bM(S_2)) \cdot \|\varphi\|_{\infty},
\]
for any compactly supported, smooth $3$-form $\varphi$ on $\RR^3$. Since both $S_1$ and $S_2$ have finite mass, we must have $a = 0$, and hence $S_1 = S_2$. 

For existence, we first show that $J_{v} - J_{u}$ is a compactly supported, integer multiplicity $2$-current with no boundary. To that end, let $(u_n)$ and $(w_n)$ be sequences in $C^{\infty}(\overline{\bB})$ and $C^{\infty}_{c}(\bB)$, respectively, such that 
\begin{equation}\label{eq:I-vu-a.e.-approximation}
|u_n - u| + |w_n - (v - u)| \to 0 \quad\text{a.e. on }\bB,
\end{equation}
and that
\begin{equation}\label{eq:I-vu-W12-approximation}
\lim_{n \to \infty}\big(\|u_n - u\|_{1, 2; \bB} + \|w_n - (v - u)\|_{1, 2;\bB}\big)  = 0.
\end{equation}
Under the assumption that $u$ and $v$ are essentially bounded, we can choose $(u_n)$ and $(w_n)$ to satisfy in addition that
\begin{equation}\label{eq:approximation-preserve-L-infty}
\|u_n\|_{\infty; \bB} \leq \|u\|_{\infty; \bB}, \quad \|w_n\|_{\infty; \bB} \leq \|v - u\|_{\infty; \bB};
\end{equation}
see for instance the proofs of~\cite[Chapter 5.3.3, Theorem 3]{Ev} and~\cite[Chapter 5.5, Theorem 2]{Ev}. Defining also
\[
v_n = u_n + w_n,
\]
we see from~\eqref{eq:continuity-of-pushforward-operation} that, in the sense of $2$-currents,
\begin{equation}\label{eq:approximation-of-Ju-in-Ivu-construction}
J_{u} = \lim_{n \to \infty} J_{u_n}, \quad J_{v} = \lim_{n \to \infty} J_{v_n}.
\end{equation}
Since $u_n$ and $v_n$ are smooth on $\overline{\bB}$ and agree on $\partial \bB$, it is standard that each $J_{v_n} - J_{u_n}$ is an integer multiplicity $2$-current, and satisfies $\partial(J_{v_n} - J_{u_n}) = 0$. Noting by~\eqref{eq:Ju-mass-bound} and~\eqref{eq:I-vu-W12-approximation} that 
\begin{equation}\label{eq:Tn-mass-bound}
\limsup_{n \to \infty}\bM(J_{v_n} - J_{u_n}) \leq D(u) + D(v),
\end{equation}
we may invoke the Federer--Fleming compactness theorem~\cite[Theorem 27.3]{LeonGMT}, which together with the convergence~\eqref{eq:approximation-of-Ju-in-Ivu-construction} shows that $J_{v} - J_{u}$ is an integer multiplicity $2$-current. That it has no boundary is clear since the same is true for each $J_{v_n} - J_{u_n}$; that it has compact support follows from~\eqref{eq:Ju-definition} and the assumption $u, v \in L^{\infty}(\bB; \RR^3)$. The isoperimetric theorem~\cite[Theorem 30.1]{LeonGMT} then yields an integer multiplicity $3$-current $S$ such that $\partial S = J_{v} - J_{u}$ and that
\[
\bM(S) \leq c \cdot \big( \bM(J_{v} - J_{u}) \big)^{\frac{3}{2}} \leq c \cdot \big( D(v) + D(u) \big)^{\frac{3}{2}},
\]
where $c$ is a universal constant, and the second inequality uses~\eqref{eq:Ju-mass-bound}. This proves the existence assertion along with property (a).

For property (b), with $R_0$ as in the statement, it follows from the definition~\eqref{eq:Ju-definition} that 
\[
\supp(J_{v} - J_{u}) \subset \overline{B_{R_0}}.
\]
Thus, in view of the proof of the isoperimetric theorem~\cite[Theorem 30.1]{LeonGMT}, as well as the statement of~\cite[Corollary 29.3]{LeonGMT}, we see that 
\[
\supp(I_{v, u}) \subset \big\{\,y \in \RR^3\ \big|\ |y| \leq R_0 + c\sqrt{\bM(J_v - J_u)}\,\big\},
\]
for some universal constant $c$, and we get~\eqref{eq:I-vu-support-bound} upon recalling~\eqref{eq:Ju-mass-bound}.

For property (c), let $(u_n)$, $(w_n)$, and $(v_n)$ be as above. For $(s, x) \in [0, 1] \times \bB$, we define
\[
F(s, x) = sv(x) + (1-s)u(x), \quad\quad F_n(s, x) = sv_n(x) + (1-s)u_n(x),
\]
so that
\begin{equation}\label{eq:Fn-boundary-condition}
F_n(s, x) = u_n(x) = v_n(x)\quad\text{on } [0, 1] \times \partial \bB,
\end{equation}
and that, by~\eqref{eq:I-vu-a.e.-approximation},
\begin{equation}\label{eq:Fn-F-a.e.}
F_n \to F\quad\text{a.e. on } [0, 1] \times \bB.
\end{equation}
Also, with $\nabla$ still standing for $(\paop{x^1}, \paop{x^2})$, we have
\begin{equation}\label{eq:Fn-F-W12}
\int_{0}^{1}\int_{\bB} |\nabla F_n - \nabla F|^2 \, dx ds \leq 2\int_{\bB}|\nabla u_n - \nabla u|^2 + |\nabla v_n - \nabla v|^2\, dx.
\end{equation}
For each $n \in\NN$, since $F_n$ is smooth on $[0, 1] \times \overline{\bB}$, the following assignment
\[
T_n: \alpha \longmapsto \int_{0}^{1}\int_{\bB} \alpha_{F_n}(F_{n, s}, F_{n, x^{1}}, F_{n, x^2})\, dx ds,
\]
defines an integer-multiplicity $3$-current, and satisfies by~\eqref{eq:Fn-boundary-condition} that $\partial T_n= J_{v_n} - J_{u_n}$. In particular, 
\[
\bM(\partial T_n) \leq D(v_n) + D(u_n).
\]
On the other hand, estimating the mass of $T_n$ itself using the standard fact that
\[
|\alpha_{F_n}(F_{n, s}, F_{n, x^{1}}, F_{n, x^2})| \leq |\alpha_{F_n}| \cdot \big| F_{n, s} \wedge F_{n, x^{1}} \wedge  F_{n, x^2}\big|,
\]
and using also~\eqref{eq:approximation-preserve-L-infty}, we obtain
\[
\bM(T_n) \leq 2\|u - v\|_{\infty; \bB} \cdot (D(u_n) + D(v_n)).
\]
Recalling~\eqref{eq:I-vu-W12-approximation}, we see that $(\bM(T_n) + \bM(\partial T_n))$ is a bounded sequence, so the compactness theorem~\cite[Theorem 27.3]{LeonGMT} yields a subsequence of $T_n$, which we do not relabel, that converges to an integer multiplicity $3$-current. Since, by the above estimate and~\eqref{eq:approximation-of-Ju-in-Ivu-construction}, the limiting current has finite mass and has boundary equal to $J_{v} - J_{u}$, we conclude that it coincides with $I_{v, u}$, so that
\begin{equation}\label{eq:Tn-to-I}
I_{v, u}(\alpha) = \lim_{n \to \infty}T_n(\alpha),
\end{equation}
for all smooth, compactly supported $3$-form $\alpha$ on $\RR^3$. On the other hand, given such an $\alpha$, we have on $[0, 1] \times \bB$ that
\begin{align*}
&\big| \alpha_{F_n}(F_{n, s}, F_{n, x^{1}}, F_{n, x^2}) - \alpha_{F}(F_{s}, F_{x^{1}}, F_{x^2}) \big|\\
\leq\ & \big| \alpha_{F_n}(v_n - u_n, F_{x^{1}}, F_{x^2}) - \alpha_{F}(v - u, F_{x^{1}}, F_{x^2}) \big|\\
& + \big| \alpha_{F_n}(v_n - u_n, F_{n, x^{1}}, F_{n, x^2}) - \alpha_{F_n}(v_n - u_n, F_{x^{1}}, F_{x^2}) \big|\\
=:\ & A_n(s, x) + B_n(s, x).
\end{align*}
Recalling the a.e. convergence~\eqref{eq:I-vu-a.e.-approximation} and~\eqref{eq:Fn-F-a.e.}, as well as the upper bound~\eqref{eq:approximation-preserve-L-infty}, we see from the dominated convergence theorem that
\[
\lim_{n \to \infty}\int_{0}^{1}\int_{\bB}A_n(s, x)\, dx ds = 0.
\]
That the same is true of $\int_{0}^{1}\int_{\bB} B_n$ follows from~\eqref{eq:Fn-F-W12} along with the estimate
\[
|B_n(s, x)| \leq \|\alpha\|_{\infty; \RR^3} \cdot \|v - u\|_{\infty; \bB} \cdot |\nabla F_n  - \nabla F| (|\nabla F_n| + |\nabla F|),
\]
and we have shown that
\[
\lim_{n \to \infty}T_n(\alpha) = \int_{0}^{1}\int_{\bB}\alpha_{F}(v - u, F_{x^1}, F_{x^2})\, dxds,
\]
which together with~\eqref{eq:Tn-to-I} finishes the proof.
\end{proof}
\begin{rmk}\label{rmk:current-extension}
Suppose $u, v \in L^{\infty} \cap W^{1, 2}(\bB; \RR^3)$ and that $v - u \in W^{1, 2}_{0}(\bB)$. Since the integral $3$-current $I_{v, u}$ has finite mass and compact support, we can extend it to act on $3$-forms that are merely continuous and have possibly non-compact support, by letting
\begin{equation}\label{eq:I-vu-extension}
I_{v, u}(\alpha): = \lim_{n \to \infty}I_{v, u}(\alpha_n),
\end{equation}
where $\alpha$ is a given continuous $3$-form on $\RR^3$, and $(\alpha_n)$ is any sequence of smooth, compactly supported $3$-forms converging uniformly to $\alpha$ on a neighborhood of $\supp(I_{v, u})$, as can be produced by mollifying and cutting off. It is standard to check that $I_{v, u}(\alpha_n)$ does converge, and that the limit is independent of the choice of approximating sequence, so the extension is well-defined. Moreover, it remains true that
\begin{equation}\label{eq:extension-bound}
|I_{v, u}(\alpha)| \leq \bM(I_{v, u}) \cdot \|\alpha\|_{\infty;\, \supp(I_{v, u})},
\end{equation}
where $\bM(I_{v, u})$ refers to the mass before extension.
\end{rmk}
We can now recall the definition of the enclosed volume following~\cite[Definition 3.5]{BogeleinDuzaarScheven2013}. In fact, we only need a special case.
\begin{defi}\label{defi:enclosed-volume}
Given $u, v \in L^{\infty} \cap W^{1, 2}(\bB; \RR^3)$ such that $v - u \in W^{1, 2}_{0}(\bB)$, along with $H \in C^{0}(\RR^3; \RR)$, in view of Remark~\ref{rmk:current-extension}, it makes sense to define
\begin{equation}\label{eq:volume-term-definition}
V_H(v, u) = I_{v, u}(2H\cdot \vol_{\RR^3}).
\end{equation}
From~\eqref{eq:extension-bound} and the isoperimetric inequality~\eqref{eq:isoperimetric-ineq}, it follows that
\begin{equation}\label{eq:isoperimetric-after-extension}
|V_H(v, u)| \leq c\, \|H\|_{\infty;\, \supp(I_{v, u})} \cdot \big(D(u) + D(v)\big)^{\frac{3}{2}},
\end{equation}
where $c$ is a universal constant. Also, given $v_1 \in L^{\infty} \cap W^{1, 2}(\bB; \RR^3)$ satisfying $v_1 - u \in W^{1, 2}_0(\bB)$, by the uniqueness part of Lemma~\ref{lemm:I-vu-construction} we see that $I_{v_1, v} + I_{v, u} = I_{v_1, u}$, and thus
\begin{equation}\label{eq:volume-additive}
V_{H}(v_1, v) = V_{H}(v_1, u) - V_{H}(v, u),
\end{equation}
by~\eqref{eq:I-vu-extension}.
\end{defi}

From the above definition, and arguments similar to those used in the proof of Lemma~\ref{lemm:I-vu-construction}, we have the following convergence result.
\begin{lemm}\label{lemm:volume-convergence}
Suppose $H \in C^0(\RR^3; \RR)$. Given $u, v$ as in Definition~\ref{defi:enclosed-volume}, let $(u_n)$ and $(v_n)$ be sequences in $L^{\infty} \cap W^{1, 2}(\bB; \RR^3)$ such that $v_n - u_n \in W^{1, 2}_{0}(\bB)$ for all $n$. Suppose also that $u_n \to u$ and $v_n \to v$ strongly in $W^{1, 2}(\bB)$, and that 
\[
\sup_{n \in \NN}\big(\|u_n\|_{\infty; \bB} + \|v_n\|_{\infty; \bB}\big) < \infty.
\]
Then $\lim_{n \to \infty}V_H(v_n, u_n) = V_H(v, u)$.
\end{lemm}
\begin{proof}
Suppose by contradiction that there exists $\ep > 0$ such that along a subsequence which we do not relabel, there holds
\begin{equation}\label{eq:volume-convergence-contradiction}
|V_H(v_n, u_n) - V_H(v, u)| \geq \ep, \quad\text{for all }n.
\end{equation}
Taking a further subsequence if needed, we can also assume that $|u_n - u| + |v_n - v| \to 0$ a.e. on $\bB$, in which case~\eqref{eq:continuity-of-pushforward-operation} gives
\begin{equation}\label{eq:J-un-to-J-u}
J_{v_n} \to J_{v},\quad J_{u_n} \to J_{u}.
\end{equation}
Noting from~\eqref{eq:isoperimetric-ineq}, and respectively~\eqref{eq:Ju-mass-bound}, that
\begin{equation}\label{eq:volume-convergence-mass-bound-on-In}
\bM(I_{v_n, u_n}) \leq c \cdot \big( D(u_n) + D(v_n) \big)^{\frac{3}{2}}, 
\end{equation}
and that
\[
\bM(\partial I_{v_n, u_n}) = \bM(J_{v_n} - J_{u_n}) \leq D(u_n) + D(v_n),
\]
we deduce from the strong $W^{1, 2}$-convergence of $u_n$ and $v_n$ that the sequence $(I_{v_n, u_n})$ of integer multiplicity $3$-currents satisfies
\[
\sup_{n \in\NN}\, (\bM(I_{v_n,u_n}) + \bM(\partial I_{v_n, u_n})) < \infty.
\]
Thus, again by the compactness theorem~\cite[Theorem 27.3]{LeonGMT}, after taking a further subsequence, we may assume that $I_{v_n, u_n}$ converges to some integer multiplicity $3$-current $I$, which satisfies
\begin{equation}\label{eq:volume-convergence-mass-bound-on-I}
\bM(I) \leq \liminf \bM(I_{v_n, u_n}) \leq c \cdot \big( D(u) + D(v) \big)^{\frac{3}{2}} < \infty.
\end{equation}
Since $\partial I = J_{v} - J_{u}$ by~\eqref{eq:J-un-to-J-u}, we conclude that $I = I_{v, u}$ by the uniqueness part of Lemma~\ref{lemm:I-vu-construction}. 

Now, by Lemma~\ref{lemm:I-vu-construction}(b), along with the assumptions of the present lemma, there exists $R > 0$ such that 
\[
\supp(I_{v_n, u_n}),\ \supp(I_{v, u}) \subset B_{R},\quad\text{for all }n.
\]
Let $(H_k)$ be a sequence in $C^{\infty}_{c}(\RR^3; \RR)$ converging uniformly to $H$ on $B_{R + 2}$, and define
\[
\alpha_k = 2H_{k} \cdot \vol_{\RR^3}, \quad \alpha = 2H \cdot \vol_{\RR^3}.
\]
Also, take a smooth cut-off function $0 \leq \zeta \leq 1$ that equals $1$ on $B_{R+1}$ and is supported in $B_{R + 2}$. Then, by the mass bounds~\eqref{eq:volume-convergence-mass-bound-on-In} and~\eqref{eq:volume-convergence-mass-bound-on-I}, as soon as $n$ is sufficiently large, we have for all $k, l \in \NN$ that
\begin{align*}
&\big| I_{v, u}(\alpha_{l}) - I_{v_n, u_n}(\alpha_l) \big|\\
&\leq \big| I_{v, u}(\zeta\alpha_{l} - \zeta\alpha_k) \big| + |I_{v, u}(\alpha_{k}) - I_{v_n, u_n}(\alpha_k)| + |I_{v_n, u_n}(\zeta\alpha_k - \zeta\alpha_{l})|\\
&\leq  2c\cdot \big(D(u) + D(v) + 1 \big)^{\frac{3}{2}} \cdot \| \alpha_l - \alpha_k \|_{\infty; B_{R + 2}} + |I_{v, u}(\alpha_{k}) - I_{v_n, u_n}(\alpha_k)|.
\end{align*}
Letting $l \to \infty$, we see from~\eqref{eq:volume-term-definition} and Remark~\ref{rmk:current-extension} that
\[
|V_H(v, u) - V_H(v_n, u_n)| \leq 2c\cdot \big(D(u) + D(v) + 1 \big)^{\frac{3}{2}} \cdot \| \alpha - \alpha_k \|_{\infty; B_{R + 2}} + |I_{v, u}(\alpha_{k}) - I_{v_n, u_n}(\alpha_k)|.
\]
Having shown in the previous paragraph that $I_{v_n, u_n} \to I_{v, u}$ as currents, we further deduce 
\[
\limsup_{n \to \infty} \big| V_H(v, u) - V_H(v_n, u_n) \big| \leq 2c\cdot \big(D(u) + D(v) + 1 \big)^{\frac{3}{2}} \cdot \| \alpha - \alpha_k \|_{\infty; B_{R + 2}}.
\]
Sending $k \to \infty$ produces a contradiction to~\eqref{eq:volume-convergence-contradiction}, and we are done.
\end{proof}

Lemma~\ref{lemm:volume-formula-H} below is an immediate consequence of Definition~\ref{defi:enclosed-volume} and Lemma~\ref{lemm:I-vu-construction}(c). Lemma~\ref{lemm:volume-formula} contains alternative expressions for the enclosed volume under stronger assumptions on $H$.
\begin{lemm}\label{lemm:volume-formula-H}
Suppose $H \in C^{0}(\RR^3; \RR)$, and let $u, v$ be maps in $W^{1, 2}\cap L^{\infty}(\bB;\RR^3)$ such that $v - u \in W^{1, 2}_0(\bB)$. Then we have
\begin{equation*}
V_H(v, u) = 2\int_{[0, 1] \times \bB} H(sv + (1-s)u)\, (v - u)\cdot (sv + (1-s)u)_{x^1} \times (sv + (1-s)u)_{x^2} \, dxds.
\end{equation*}
\end{lemm}
\begin{proof}
Choose $R> \max\{\|u\|_{\infty; \bB},\, \|v\|_{\infty; \bB}\}$ such that $\supp(I_{v, u}) \subset B_{R}$, and let $(H_n)$ be a sequence in $C^{\infty}_{c}(\RR^3; \RR)$ converging uniformly to $H$ on a neighborhood of $\overline{B_{R}}$. Then by Definition~\ref{defi:enclosed-volume} we have
\[
V_H(v, u) = \lim_{n \to \infty}I_{v, u}(2H_{n} \cdot \vol_{\RR^3}).
\]
On the other hand, Lemma~\ref{lemm:I-vu-construction}(c) gives
\[
\begin{split}
&I_{v, u}(2H_n \cdot\vol_{\RR^3}) \\
=\ &  2\int_{[0, 1] \times \bB} H_n(sv + (1-s)u)\, (v-u)\cdot (sv + (1-s)u)_{x^1} \times (sv + (1-s)u)_{x^2} \, dxds.
\end{split}
\]
Since $|sv(x) + (1-s)u(x)| < R$ almost everywhere on $[0, 1] \times \bB$, and since $\sup_{|y| \leq R} |H_n(y) - H(y)| \to 0$, we get the asserted formula upon letting $n \to \infty$.
\end{proof}
\begin{lemm}\label{lemm:volume-formula}
Suppose $u, v \in W^{1, 2}\cap L^{\infty}(\bB;\RR^3)$ and that $v - u \in W^{1, 2}_0(\bB)$. Let $H \in C^{1}(\RR^3; \RR)$ and suppose $Q:\RR^3 \to \RR^3$ is a $C^{1}$-vector field satisfying~\eqref{eq:Q-H-relation}; that is, $\Div Q(\cdot) = 2H(\cdot)$. Then the following hold.
\vskip 1mm
\begin{enumerate}
\item[(a)] We have
\begin{equation}\label{eq:volume-formula}
V_H(v, u) = \int_{\bB}  Q(v)\cdot v_{x^1} \times v_{x^2} dx - \int_{\bB} Q(u)\cdot u_{x^1} \times u_{x^2} dx.
\end{equation}
\vskip 1mm
\item[(b)] Writing $\varphi$ for $v - u$, we have
\begin{equation}\label{eq:volume-difference-alt-formula}
\begin{split}
V_{H}(v, u) =\ & 2\int_{\bB}H(u)\, u_{x^1} \times u_{x^2} \cdot \varphi\, dx + \int_{\bB} Q(u + \varphi) \cdot \varphi_{x^1} \times \varphi_{x^2} \, dx\\
& + \int_{\bB}(Q(u + \varphi) - Q(u)) \cdot (u_{x^1} \times \varphi_{x^2} + \varphi_{x^1} \times u_{x^2})\, dx\\
& + \int_{\bB} (Q(u + \varphi) - Q(u) - (DQ)_{u}(\varphi)) \cdot u_{x^1} \times u_{x^2}\, dx.
\end{split}
\end{equation}
\end{enumerate}
\end{lemm}
\begin{rmk}
A $C^{1}$-vector field satisfying~\eqref{eq:Q-H-relation} always exists whenever $H(\cdot)$ is of class $C^{1}$. One particular choice, essentially the same as~\cite[equation (4.3)]{GulliverSpruck1971}, is given by
\begin{equation*}
Q(y) = \Big(\int_{0}^{1}2t^2H(ty)\, dt\Big)\, y.
\end{equation*}
\end{rmk}
\begin{proof}[Proof of Lemma~\ref{lemm:volume-formula}]
For part (a), choose $R> 0$ such that
\begin{equation}\label{eq:R-large-enough}
\max\{\|u\|_{\infty; \bB},\,\|v\|_{\infty; \bB}\} < R,\quad \text{and}\quad \supp(I_{v, u}) \subset B_{R}.
\end{equation}
By mollification we obtain a sequence $(Q_n)$ of smooth vector fields on $\RR^3$ such that 
\begin{equation}\label{eq:C1-approximation-Q}
\sup_{|y| \leq R+2} \big(\,|Q_n(y) - Q(y)| + |(D Q_n)_{y} - (D Q)_{y}|\,\big) \longrightarrow 0,\quad\text{as }n \to \infty.
\end{equation}
In particular, letting $\beta_n: = \iota_{Q_n}\vol_{\RR^3}$, we have by~\eqref{eq:Q-H-relation} that
\begin{equation}\label{eq:volume-form-converge}
d\beta_{n}  = \Div Q_n \vol_{\RR^3} \to 2H \vol_{\RR^3},\quad\text{uniformly on }\overline{B_{R+2}}.
\end{equation}
Next take a cut-off function $\zeta \in C^{\infty}_{c}(B_{R + 2})$ that equals $1$ on $\overline{B_{R + 1}}$. Then $(d(\zeta\beta_n))$ is a sequence of smooth, compactly supported $3$-forms that converges uniformly to $2H \vol_{\RR^3}$ on $B_{R + 1}$, which is a neighborhood of $\supp(I_{v, u})$. Thus $V_H(v, u)$ can be computed as
\[
V_H(v, u) = \lim_{n \to \infty} I_{v, u}(d(\zeta\beta_n)).
\]
Since $\partial I_{v, u} = J_{v} - J_{u}$, we have
\begin{equation}\label{eq:limit-Ju-Jv}
\begin{split}
I_{v, u}(d(\zeta\beta_n)) = \partial I_{v, u}(\zeta\beta_n) =\ & J_v(\zeta\beta_n) - J_u(\zeta\beta_n)\\
=\ & \int_{\bB}Q_n(v) \cdot v_{x^1} \times v_{x^2} dx - \int_{\bB}Q_n(u) \cdot u_{x^1} \times u_{x^2} dx,
\end{split}
\end{equation}
where the last step uses the first condition in~\eqref{eq:R-large-enough} and our choice of $\zeta$. Noting from~\eqref{eq:R-large-enough} and~\eqref{eq:C1-approximation-Q} that
\[
\|Q_n(u) - Q(u)\|_{\infty; \bB} + \|Q_{n}(v) - Q(v)\|_{\infty; \bB} \to 0 \quad\text{as }n \to \infty,
\]
and passing to the limit in~\eqref{eq:limit-Ju-Jv}, we get~\eqref{eq:volume-formula}.

For part (b), we first use (a) to compute
\begin{equation}\label{eq:volume-difference-1}
\begin{split}
V_H(u + \varphi, u) =\ & \int_{\bB}(Q(u + \varphi) - Q(u)) \cdot u_{x^1} \times u_{x^2}\, dx + \int_{\bB }Q(u + \varphi) \cdot \varphi_{x^1} \times \varphi_{x^2}\, dx\\
& +\int_{\bB} (Q(u + \varphi) - Q(u)) \cdot (u_{x^1} \times \varphi_{x^2} + \varphi_{x^1} \times u_{x^2})\, dx\\
& + \int_{\bB} Q(u) \cdot (u_{x^1} \times \varphi_{x^2} + \varphi_{x^1} \times u_{x^2})\, dx.
\end{split}
\end{equation}
To continue, as in the proof of Lemma~\ref{lemm:I-vu-construction}, we let $(\varphi_n)$ be a sequence in $C^{\infty}_{c}(\bB; \RR^3)$ such that 
\[
\|\nabla \varphi_n - \nabla \varphi\|_{2; \bB} \to 0, \quad \varphi_n \to \varphi \quad\text{a.e. on }\bB,\quad \|\varphi_n\|_{\infty;\bB} \leq \|\varphi\|_{\infty; \bB},
\]
and let $(u_n)$ be a sequence in $C^{\infty}(\overline{\bB}; \RR^3)$ that approximates $u$ in the same way. Then we have
\[
\|\nabla(Q(u_n)) - (DQ)_{u}(\nabla u)\|_{2; \bB} \to 0.
\]
Integrating by parts in the last term in~\eqref{eq:volume-difference-1}, which can be justified using these approximations, we get
\begin{equation}\label{eq:volume-difference-div}
\begin{split}
&\int_{\bB} Q(u) \cdot (u_{x^1} \times \varphi_{x^2} + \varphi_{x^1} \times u_{x^2}) \,dx \\
=\ &  \int_{\bB} \varphi \cdot (u_{x^1} \times (DQ)_{u}(u_{x^2}) + (DQ)_{u}(u_{x^1}) \times u_{x^2})\, dx\\
=\ & 2\int_{\bB}H(u)\, \varphi \cdot u_{x^1} \times u_{x^2} - (DQ)_{u}(\varphi) \cdot u_{x^1} \times u_{x^2}\, dx,
\end{split}
\end{equation}
where for the second equality we used~\eqref{eq:Q-H-relation} and the general fact that
\begin{equation}\label{eq:Lie-algebra-action}
(\tr A) \cdot \vol_{\RR^3} = \frac{d}{dt}\Big|_{t=0}(e^{tA})^*\vol_{\RR^3}
\end{equation}
for any $3\times 3$ matrix $A$. Substituting~\eqref{eq:volume-difference-div} back into~\eqref{eq:volume-difference-1} gives~\eqref{eq:volume-difference-alt-formula}. 
\end{proof}

\section{Convexity of the \texorpdfstring{$H$}{H}-energy}\label{sec:convexity}
This section is devoted to the proofs of Theorems~\ref{thm:uniqueness} and~\ref{thm:H-convexity}. We use freely the Hardy inequality for functions in $W^{1, 2}_0(\bB)$, as stated for instance in~\cite[Lemma 3.1]{LuWang2012}: 
\begin{equation}\label{eq:Hardy}
\int_{\bB} \frac{\varphi^2}{(1 - |x|)^2}\, dx \leq 4\int_{\bB} |\nabla \varphi|^2\, dx, \quad\text{for all }\varphi \in W^{1, 2}_0(\bB).
\end{equation}
In addition, we need the celebrated Wente inequality for solutions to Poisson equations of the form
\[
\Delta \psi = a_{x^1}b_{x^{2}} - a_{x^2}b_{x^1} \quad\text{on }\bB;
\]
that is, with the right-hand side being a determinant term. The original version of the result concerns the Dirichlet problem; see for instance~\cite[Lemma A.1]{BrezisCoron1984} or~\cite[Theorem 3.1.2]{Helein2002}. With Neumann boundary conditions, the Wente inequality holds provided either $a$ or $b$ vanishes on $\partial \bB$, as stated in~\cite[Theorem 1.3]{DaLioPalmurella2017} and proved in~\cite[Lemma A.6]{DaLioPalmurellaRiviere2020}. In this paper we use both the Dirichlet and the Neumann versions, the latter of which we recall in Appendix~\ref{sec:Neumann-Wente}. It should also be noted that the Wente inequality fails in general for Neumann problems. Counterexamples have been constructed in~\cite{Hirsch2019} and~\cite{DaLioPalmurella2017}.

\begin{proof}[Proof of Theorem~\ref{thm:uniqueness}]
Since $H$ is continuously differentiable, it is well-known that 
\[
u, v \in C^{2, \alpha}_{\loc}(\bB),\quad \text{for any }\alpha \in (0, 1).
\]
See for instance~\cite{Bethuel1992} or~\cite{Strzelecki2003}, which in fact require only that $H$ be bounded and Lipschitz. Fixing a particular $q > 4$ and letting $\alpha = 1 - \frac{4}{q}$, we get from Corollary~\ref{coro:grad-estimate-scaled} that, provided $\ep_{\uni}$ is below a threshold depending only on $\Lambda_0$, there holds the following gradient estimate:
\begin{equation}\label{eq:grad-estimate}
|\nabla u(x)|^2 \leq C_{\Lambda_0} \frac{D(u)}{(1 - |x|)^2}\,, \quad\text{for all }x \in \bB\,,
\end{equation}
and a similar estimate holds with $v$ in place of $u$. (Finer estimates for a more general class of equations can be found in the work of Lamm and the second named author~\cite[Corollary 1.4, Remark 1.5]{LammLin2013}.) By~\eqref{eq:grad-estimate} and~\eqref{eq:Hardy}, we have for all $\zeta \in W^{1, 2}_0(\bB; \RR^3)$ that $|\nabla u||\zeta| \in L^2(\bB)$, and that
\begin{equation}\label{eq:grad-estimate-and-Hardy}
\int_{\bB} |\nabla u|^2 |\zeta|^2\, dx \leq C_{\Lambda_0} D(u) \int_{\bB} \frac{|\zeta|^2}{(1 - |x|)^2}\, dx \leq 4C_{\Lambda_0}D(u)\int_{\bB}|\nabla\zeta|^2\, dx.
\end{equation}
A standard approximation argument then shows that, in~\eqref{eq:H-sys-weak}, we can use test functions that belong to $W^{1, 2}_0(\bB; \RR^3)$ alone. 

To continue, we use $u - v$ as a test function in the systems satisfied by $u$ and $v$, respectively, which is permitted by the discussion in the previous paragraph. After some routine estimates, we get
\begin{align}
\int_{\bB}|\nabla u - \nabla v|^2\, dx =\ & -2\int_{\bB} (H(u) - H(v))\, u_{x^1} \times u_{x^2} \cdot (u - v)\, dx  \nonumber\\
& -2 \int_{\bB} H(v)(u_{x^1} \times u_{x^2} - v_{x^1} \times v_{x^2}) \cdot (u- v)\, dx\nonumber \\
\leq\ &  2 \| \nabla H \|_{\infty}\int_{\bB} |u - v|^2 |\nabla u|^2\, dx\nonumber \\
&+ 2 \|H\|_{\infty}\int_{\bB} (|\nabla u| + |\nabla v|) |\nabla u - \nabla v||u - v|\, dx.\label{eq:uniqueness-computation}
\end{align}
Applying Young's inequality, followed by~\eqref{eq:grad-estimate-and-Hardy}, which is valid for $v$ as well, we get
\begin{equation}\label{eq:uniqueness-estimate-after-Young}
\begin{split}
\frac{1}{2}\int_{\bB}|\nabla u - \nabla v|^2\, dx \leq\ & \big( 2\|\nabla H\|_{\infty} + 4\|H\|_{\infty}^2 \big) \int_{\bB} |u - v|^2(|\nabla u|^2 + |\nabla v|^2)\, dx\\
\leq\ & C_{\Lambda_0, \Lambda_1}\cdot (D(u) + D(v)) \int_{\bB}|\nabla u - \nabla v|^2\, dx.
\end{split}
\end{equation}
Since $D(u) + D(v) \leq \ep_{\uni}$ by assumption, we get that $u = v$ provided $\ep_{\uni}$ is less than a threshold depending only on $\Lambda_0$ and $\Lambda_1$. The finishes the proof.
\end{proof}

\begin{proof}[Proof of Theorem~\ref{thm:H-convexity}]
Since $u$ is a weak solution of~\eqref{eq:H-system}, we again have that $u \in C^{2, \alpha}_{\loc}(\bB)$, so the estimate~\eqref{eq:grad-estimate} continues to hold, provided $\ep_{\conv}$ is small enough depending only on $\Lambda_0$. Consequently, we also have~\eqref{eq:grad-estimate-and-Hardy} available for $u$. Next, we compute using~\eqref{eq:H-system} that
\begin{align}
D(v) - D(u) =\ &  \frac{1}{2} \int_{\bB}\left(|\nabla v|^2 - |\nabla u|^2\right) \,dx= \int_{\bB}\bangle{\frac{\nabla v + \nabla u}{2}, \nabla v - \nabla u} \,dx \nonumber\\
=\ & \frac{1}{2}\int_{\bB}|\nabla v - \nabla u|^2 \,dx - 2\int_{\bB} H(u) (v - u)\cdot u_{x^1} \times u_{x^2} \,dx.\label{eq:D-difference}
\end{align}
Combining this with Lemma~\ref{lemm:volume-formula}(b) gives
\begin{align}
D(v) - D(u) + V_H(v, u) =\ & \frac{1}{2}\int_{\bB} |\nabla v - \nabla u|^2  \, dx + I_1 + I_2 + I_3, \label{eq:Ef-difference}
\end{align}
where $I_1, I_2$ and $I_3$ stand respectively for the following integrals:
\begin{align*}
I_1 = \ & \int_{\bB} \big( Q(v) - Q(u) - (D Q)_u(v - u)\big) \cdot u_{x^1} \times u_{x^2}   \, dx\,, \\ 
I_2 =\ & \int_{\bB} \big( Q(v) - Q(u)\big) \cdot \big( u_{x^1} \times (v-u)_{x^2} + (v-u)_{x^1} \times u_{x^2}\big)  \, dx \,, \\ 
I_3=\ &  \int_{\bB} Q(v) \cdot (v - u)_{x^1} \times (v - u)_{x^2}  \, dx. 
\end{align*}
To estimate the term $I_1$, notice that 
\[
|Q(v) - Q(u) - (DQ)_{u}(v - u)| \leq C\|D^2Q\|_{\infty} \cdot |v - u|^2.
\]
Combining this with~\eqref{eq:grad-estimate-and-Hardy}, we obtain
\begin{align}
|I_1| \leq C\|D^2 Q\|_{\infty} \int_{\bB} |v - u|^2 |\nabla u|^2 \, dx \leq\ & C_{\Lambda_0}\|D^2 Q\|_{\infty}\cdot \ep_{\conv}\int_{\bB}|\nabla v - \nabla u|^2 \, dx.\label{eq:I_1-estimate}
\end{align}
For $I_2$, we use H\"older's inequality and~\eqref{eq:grad-estimate-and-Hardy} to get
\begin{align}
|I_{2}| \leq\ & C\|DQ\|_{\infty} \int_{\bB} |v - u||\nabla u| \cdot |\nabla v - \nabla u|\, dx \nonumber\\
\leq\ & C_{\Lambda_0} \|DQ\|_{\infty} \sqrt{\ep_{\conv}} \int_{\bB}|\nabla v - \nabla u|^2\, dx.\label{eq:I_2-estimate}
\end{align}
Next we consider the term $I_3$. Since $u - v \in W^{1, 2}_0(\bB)$, by Lemma~\ref{lemm:Neumann-Wente} (and Remark~\ref{rmk:Neumann-Wente}) we obtain some $\psi \in C^0(\overline{\bB}; \RR^3) \cap W^{1, 2}(\bB; \RR^3)$ such that
\[
-\int_{\bB} \bangle{\nabla \zeta, \nabla \psi}\, dx = \int_{\bB}\zeta\cdot (v-u)_{x^1}\times (v-u)_{x^2}\, dx, \quad\text{for all }\zeta \in W^{1,2} \cap L^{\infty}(\bB; \RR^3),
\]
and that
\begin{equation}\label{eq:Wente-fo-I3}
\|\nabla\psi\|_2 \leq C\|\nabla v - \nabla u\|_2^2,
\end{equation}
for some universal constant $C$. Thus, noting that $Q(v) \in W^{1, 2} \cap L^{\infty}(\bB; \RR^3)$, we get
\begin{equation}\label{eq:I_3-estimate}
\begin{split}
|I_3| = \Big|-\int_{\bB}\big\langle (DQ)_v(\nabla v), \nabla \psi \big\rangle \, dx\Big| \leq\ & C\|D Q\|_{\infty}\|\nabla v\|_2 \|\nabla v - \nabla u\|_2^2\\
\leq\ & C\|DQ\|_{\infty}\cdot \sqrt{\ep_{\conv}}\cdot \|\nabla v - \nabla u\|_2^2\,.
\end{split}
\end{equation}
Combining this with~\eqref{eq:I_2-estimate},~\eqref{eq:I_1-estimate}, and~\eqref{eq:Ef-difference}, we conclude that
\[
D(v)- D(u) + V_H(v, u) \geq (\frac{1}{2} - C_{\Lambda_0} \cdot \Lambda' \cdot \sqrt{\ep_{\conv}}) \int_{\bB}|\nabla v - \nabla u|^2\,dx \,,
\]
which establishes the desired convexity estimate, provided $\ep_{\conv}$ is sufficiently small.
\end{proof}
\section{Regular solutions to the \texorpdfstring{$H$}{H}-surface flow with small energy data}\label{sec:regular-solutions}

Fixing some $q > 4$, and setting $\alpha = 1 - \frac{4}{q}$, in this section we consider the Cauchy-Dirichlet problem~\eqref{eq:H-flow} in the case where $u_0 \in C^{2, \alpha}(\overline{\bB}; \RR^3)$. We begin by recalling the uniqueness and short-time existence results due to Chen--Levine~\cite{ChenLevine2002}. Below, by $C^{2+\alpha, 1 + \frac{\alpha}{2}}(\overline{\bB} \times [0, T))$ we mean the intersection $\cap_{0 < T' < T}C^{2 + \alpha, 1+\frac{\alpha}{2}}(\overline{\bB} \times [0, T'])$, where the latter spaces are defined according to~\cite[pages 7-8]{LSU}. (See also Appendix~\ref{sec:further-a-priori}.) 

\begin{lemm}[\cite{ChenLevine2002}, Lemma 4.8; \cite{Struwe1985}, Lemma 3.12]
\label{lemm:uniqueness}
Let $H \in C^{1}(\RR^3; \RR)$. Given $T > 0$ and $u_0 \in C^{2, \alpha}(\overline{\bB}; \RR^3)$, suppose $u, v \in C^{2+\alpha, 1 + \frac{\alpha}{2}}(\overline{\bB} \times [0, T); \RR^3)$ are both solutions to~\eqref{eq:H-flow}. Then in fact $u = v$.
\end{lemm}
\begin{proof}
Since we restrict to regular solutions, the argument becomes rather simple. Fix any $T_1 \in (0, T)$, and define
\begin{align*}
R: = \ & \sup_{\overline{\bB} \times [0, T_1]}( |u| + |\nabla u| + |v| + |\nabla v| ),\\
\Lambda: =\ &  \sup_{|y| \leq R} \big( |H(y)| + |\nabla H(y)| \big).
\end{align*}
Given $t \in [0, T_1]$, since $u - v = 0$ on $\partial \bB \times [0, T)$, upon taking its inner product with the equations satisfied $u$ and $v$, and integrating over $\bB \times \{t\}$, we get
\begin{align*}
&\int_{\bB \times \{t\}} (u_{t} - v_{t}) \cdot (u - v) + |\nabla u - \nabla v|^2\, dx \\
=\ & -2\int_{\bB \times \{t\}} \big( H(u)\, u_{x^1} \times u_{x^2} - H(v)\, v_{x^1} \times v_{x^2} \big) \cdot (u - v)\, dx\\
\leq \ & 2\Lambda \int_{\bB \times \{t\}} |u - v|^2 |\nabla u|^2 + |u -v| |\nabla u - \nabla v|(|\nabla u| + |\nabla v|)\, dx,
\end{align*}
where the last line follows as in~\eqref{eq:uniqueness-computation}. With the help of Young's inequality, we get
\begin{align*}
&\frac{d}{dt}\Big(\int_{\bB \times \{t\}} |u - v|^2\, dx \Big) + \int_{\bB \times \{t\}} |\nabla u - \nabla v|^2 \, dx\\
&\leq C(\Lambda + \Lambda^2)\int_{\bB \times \{t\}} |u - v|^2 (|\nabla u| + |\nabla v|)^2\, dx \leq C_{\Lambda, R} \int_{\bB \times \{t\}} |u - v|^2\, dxdt.
\end{align*}
Recalling that $\int_{\bB \times \{0\}}|u - v|^2\, dx = 0$, we see that $u = v$ on $\bB \times [0, T_1]$ upon integrating the above differential inequality. This finishes the proof since $T_1 \in (0, T)$ is arbitrary.
\end{proof}

\begin{prop}[\cite{ChenLevine2002}, Theorem 3.2]
\label{prop:short-time-existence}
Suppose $H \in C^{1}(\RR^3; \RR)$. Given $u_0 \in C^{2, \alpha}(\overline{\bB}; \RR^3)$, there exists $\tau = \tau(u_0) > 0$ such that~\eqref{eq:H-flow} admits a solution $u \in C^{2 + \alpha, 1 + \frac{\alpha}{2}}(\overline{\bB} \times [0, \tau]; \RR^3)$. Moreover, given a bounded subset $\sF$ of $C^{2, \alpha}(\overline{\bB}; \RR^3)$, we can choose $\{\tau(u_0)\}_{u_0 \in \sF}$ in such a way that
\[
\inf_{u_0 \in \sF}\tau(u_0) > 0.
\]
\end{prop}
\begin{proof}
We indicate the key elements of the proof given in~\cite{ChenLevine2002}. The notation from~\cite[Chapter I]{LSU} for Lebesgue and Sobolev spaces will be used, which differs slightly from that used elsewhere in this paper. With $q > 4$ being the exponent fixed at the start of this section, we define the following Banach spaces
\begin{align*}
\cX_{q} = \ & W^{2, 1}_{q}(\bB \times (0, 1)), \\
\cV_{q} =\ & \big\{(h, k) \in  W_{q}^{2 - \frac{2}{q}}(\bB) \times W_{q}^{2 -\frac{1}{q}, 1 - \frac{1}{2q}}(\partial \bB \times (0, 1))\ \big| \ h|_{\partial\bB} = k|_{\partial \bB \times \{0\}} \big\},\\
\cY_{q} =\ & L_{q}(\bB \times (0, 1)) \times \cV_{q},
\end{align*}
and identify $W^{2}_{q}(\bB)$ with its image in $\cX_{q}$ via the embedding that takes $u$ to the function $(x, t) \mapsto u(x)$. The two central objects in the proof of~\cite[Theorem 3.2]{ChenLevine2002} are the map $K:\cX_{q}\rightarrow \cY_{q}$ defined by
\[
K(u) = (u_{t} - \Delta u + 2H(u)u_{x^1} \times u_{x^2},\ u|_{\bB \times \{0\}},\ u|_{\partial \bB \times [0, 1]}),
\]
and, with $\tau$ to be determined, the map $S_{\tau}: \cY_{q} \rightarrow \cY_{q}$ defined by
\[
S_{\tau}(g, u, \chi) =(g_{\tau}, u, \chi),
\]
where we let
\[
g_{\tau}= g \text{ on }\bB \times [\tau, 1), \quad g_{\tau} = 0 \text{ otherwise}.
\]
Clearly $S_{\tau}$ is a bounded linear map. On the other hand, by standard parabolic Sobolev embeddings~\cite[Lemma II.3.3]{LSU} and trace theorems~\cite[Lemma II.3.4]{LSU}, the map $K$ does indeed have the stated codomain, and is a $C^1$-map. To put briefly the subsequent argument in~\cite{ChenLevine2002}, given $u_0 \in W^{2}_{q}(\bB) \hookrightarrow \cX_{q}$, one uses~\cite[Theorem IV.9.1]{LSU} and the inverse function theorem to produce neighborhoods $\cN_{1}$ and $\cN_{2}$ of $u_0$ and $K(u_0)$, respectively, such that 
\[
K|_{\cN_{1}}: \cN_{1} \to \cN_{2}
\]
is a $C^1$-diffeomorphism. A small enough $\tau = \tau(u_0)$ is then found so that $S_{\tau}( K(u_0)) \in \cN_{2}$, in which case the definition of $S_{\tau}$ ensures that 
\[
u: = (K|_{\cN_{1}})^{-1} \big( S_{\tau}( K(u_0))\big)
\]
restricts to a solution of~\eqref{eq:H-flow} of class $W^{2, 1}_{q}(\bB \times (0, \tau))$. Since $q >4$, Sobolev embedding (\cite[Lemma II.3.3]{LSU}) can be used again to show that the inhomogeneous term $2H(u)u_{x^1} \times u_{x^2}$ in~\eqref{eq:H-flow} lies in $C^{\alpha, \frac{\alpha}{2}}(\overline{\bB} \times [0, \tau])$. If further $u_0 \in C^{2,\alpha}(\overline{\bB})$, then the combination of~\cite[Theorems IV.5.2 and IV.9.1]{LSU} implies that $u \in C^{2 + \alpha, 1 + \frac{\alpha}{2}}(\overline{\bB} \times [0, \tau])$.

To get the second conclusion, notice that, by applying the argument in the previous paragraph to each $u_0 \in W^{2}_{q}(\bB)$ and taking the union of the resulting open sets $\cN_{2}$, we get a neighborhood $\sU$ of $K(W^{2}_{q}(\bB))$ in $\cY_{q}$ such that $\sU \subset K(\cX_{q})$. On the other hand, writing $\overline{\sF}$ for the $W^{2}_{q}$-closure of $\sF$, we see from the compact embedding $C^{2,\alpha}(\overline{\bB}) \hookrightarrow W^{2}_{q}(\bB)$ and the continuity of $K$ that $K(\overline{\sF})$ is a compact subset of $\sU$, and thus there is some $\rho_* > 0$ such that 
\begin{equation}\label{eq:contained-in-image}
\cup_{u_0 \in \sF}\, B_{\rho_*}^{\cY_{q}}(K(u_0)) \subset \sU \subset K(\cX_{q}).
\end{equation}
Again using the boundedness of $\sF$ in $C^{2,\alpha}(\overline{\bB})$, by a direct computation we have for all $u_0 \in \sF$ that 
\[
\begin{split}
\|S_{\tau} ( K(u_0)) - K(u_0)\|_{\cY_{q}} =\ & \tau^{\frac{1}{q}}  \cdot \|\Delta u_0 - 2H(u_0)(u_0)_{x^1} \times (u_0)_{x^2}\|_{q; \bB}\leq C \cdot \tau^{\frac{1}{q}},
\end{split}
\]
where $C$ depends only on $q$, $\sF$, and $H$. In particular, we can find $\tau_* > 0$ such that
\[
S_{\tau_*}( K(u_0)) \in  B^{\cY_{q}}_{\rho_*}(K(u_0)),\quad \text{for all }u_0 \in \sF.
\]
We get the asserted positivity upon combining this with~\eqref{eq:contained-in-image}.
\end{proof}

As is well-known, one of the major difficulties in dealing with the $H$-surface flow is that, in contrast with the harmonic map flow, the monotone quantity in this case, namely the $E_H$-functional as defined by~\eqref{eq:energy-functional}, does not necessarily yield an energy bound along the flow. In Lemma~\ref{lemm:E-H-monotone} we recall the said monotonicity property of $E_H$. Then, in Lemma~\ref{lemm:energy-bound-from-isoperimetric} we show that, when $D(u_0)$ is small enough, one can go around the problem mentioned above by using the isoperimetric inequality in a way similar to Duzaar--Steffen~\cite[Theorem 4.4(ii)]{DuzaarSteffen1999} and B\"ogelein--Duzaar--Scheven~\cite[Remark 8.2]{BogeleinDuzaarScheven2013}.

\begin{lemm}[\cite{ChenLevine2002}, Lemma 4.3]
\label{lemm:E-H-monotone}
Let $H \in C^{1}(\RR^3; \RR)$. Suppose $u_0 \in C^{2, \alpha}(\overline{\bB}; \RR^3)$, and let $u \in C^{2 + \alpha, 1 + \frac{\alpha}{2}}(\overline{\bB} \times [0, T); \RR^3)$ be a solution of~\eqref{eq:H-flow}. Then, given $t_1 < t_2$ so that $[t_1, t_2] \subset [0, T)$, there holds
\begin{equation}\label{eq:E-H-monotone}
\int_{t_1}^{t_2}\int_{\bB} |u_t|^2 dxdt = E_{H}(u(\cdot, t_1), u_0) - E_{H}(u(\cdot, t_2), u_0).
\end{equation}
\end{lemm}
\begin{proof}
We use a difference quotient argument similar to the one leading up to~\cite[page 5280, equation (5.6)]{LuWang2012}. Given $h \in \RR$ with $0 < |h| < T$, we define
\begin{equation}\label{eq:diff-quotient-in-t}
u^{(h)}(x, t) := u(x, t + h),\quad v_{h}(x, t) := \frac{1}{h}(u(x,t) - u^{(-h)}(x, t)),
\end{equation}
for $t \in [-h, T-h)$ and, respectively, $t \in [0, T) \cap [h, T+h)$. Notice by the boundary condition in~\eqref{eq:H-flow} that 
\begin{equation}\label{eq:t-step-vanish-on-boundary}
v_{h}(x, t) = 0,\quad\text{whenever }x\in \partial\bB.
\end{equation}
To prove~\eqref{eq:E-H-monotone}, we consider first the case $t_1 > 0$, and define 
\[
\delta: = \frac{1}{4}\min\{t_1, t_2 - t_1, T - t_2\}.
\]
For $h \in (0, \delta)$, we multiply the differential equation in~\eqref{eq:H-flow} by $v_{h}$ and integrate over $[t_1, t_2] \times \bB$ to get
\begin{equation}\label{eq:H-flow-tested}
\int_{t_1}^{t_2}\int_{\bB} v_{h} \cdot u_{t} + \bangle{\nabla v_{h}, \nabla u}\, dxdt = -2\int_{t_1}^{t_2}\int_{\bB}H(u)\,v_{h} \cdot u_{x^{1}} \times u_{x^2}\, dxdt.
\end{equation}
Noting from the splitting $\nabla u = \frac{ \nabla u + \nabla u^{(-h)}}{2} + \frac{\nabla u - \nabla u^{(-h)}}{2}$ that
\begin{equation}\label{eq:H-flow-tested-lhs-observation}
\bangle{\nabla v_{h}, \nabla u} \geq \frac{|\nabla u|^2 - |\nabla u^{(-h)}|^2}{2h},
\end{equation}
and sending $h \to 0$, we get
\begin{equation}\label{eq:functional-monotonicity-liminf}
-\int_{t_1}^{t_2}\int_{\bB} 2H(u)\, u_{t}\cdot u_{x^1} \times u_{x^2}\, dxdt\geq \int_{t_1}^{t_2}\int_{\bB}|u_{t}|^2dxdt + D(u(\cdot, t_2)) - D(u(\cdot, t_1)).
\end{equation}
The reverse inequality is proved by a similar argument, with $v_{h}$ replaced by the forward difference quotient $v_{-h}$, and with~\eqref{eq:H-flow-tested-lhs-observation} replaced by
\[
\bangle{\nabla v_{-h}, \nabla u} \leq \frac{|\nabla u^{(h)}|^2 - |\nabla u|^2}{2h}.
\]

To relate the left-hand side in~\eqref{eq:functional-monotonicity-liminf} to the enclosed volume, we define
\[
f_h(\lambda, x, t) = \lambda u(x, t-h) + (1-\lambda)u(x, t),
\]
for $(\lambda, x, t) \in [0, 1] \times \bB \times [t_1, t_2]$. Then Lemma~\ref{lemm:volume-formula-H} implies that
\begin{equation}\label{eq:functional-monotonicity-volume}
\int_{t_1}^{t_2}\frac{V_H(u(\cdot, t), u(\cdot, t-h))}{h} \, dt = 2\int_{[0, 1] \times\bB \times [t_1, t_2]} H(f_h)\, v_h \cdot (f_h)_{x^1} \times (f_h)_{x^2}\, d\lambda dxdt.
\end{equation}
For the left-hand side, noting that $t \mapsto V_H(u(\cdot, t), u_0)$ is continuous by Lemma~\ref{lemm:volume-convergence}, and using also~\eqref{eq:volume-additive}, we have
\begin{equation}\label{eq:volume-diff-quotient-limit}
\lim_{h \to 0}\int_{t_1}^{t_2}\frac{V_H(u(\cdot, t), u(\cdot, t-h))}{h} \, dt = V_H(u(\cdot, t_2), u_0) - V_H(u(\cdot, t_1), u_0).
\end{equation}
For the right-hand side in~\eqref{eq:functional-monotonicity-volume}, observe that, as $h \to 0$, the integrand satisfies
\[
H(f_h)\, v_h \cdot (f_h)_{x^1} \times (f_h)_{x^2} \to H(u)\, u_t \cdot u_{x^1} \times u_{x^2},
\]
pointwise on $[0, 1] \times \bB \times [t_1, t_2]$. Moreover, letting
\[
R =  \sup_{(x, t) \in \bB \times [t_1 - \delta, t_2 + \delta]}(|u(x,t)| + |\nabla u(x, t)| + |u_t(x, t)|),
\]
we have for all $h \in (0, \delta)$ the following bound, again pointwise on $[0, 1] \times \bB \times [t_1, t_2]$:
\[
\big| H(f_h)\, v_h \cdot (f_h)_{x^1} \times (f_h)_{x^2}\big| \leq R^3 \cdot \sup_{|y |\leq R}|H(y)|.
\]
Thus, letting $(h_n)$ be a sequence in $(0, \delta)$ converging to $0$, by the dominated convergence theorem, we obtain from~\eqref{eq:functional-monotonicity-volume} and~\eqref{eq:volume-diff-quotient-limit} that
\[
2\int_{\bB \times [t_1, t_2]} H(u)\, u_{t}\cdot u_{x^1} \times u_{x^2}\, dxdt  = V_H(u(\cdot, t_2), u_0) - V_H(u(\cdot, t_1), u_0).
\]
Combining the above with~\eqref{eq:functional-monotonicity-liminf} and the remarks thereafter, we arrive at~\eqref{eq:E-H-monotone} in the case $t_1 > 0$, from which the case $t_1 = 0$ follows since, by Lemma~\ref{lemm:volume-convergence}, we have
\[
\lim_{t \to 0^{+}} E_H(u(\cdot, t), u_0)  = E_H(u(\cdot, 0), u_0).
\]
This concludes the proof.
\end{proof}

\begin{lemm}\label{lemm:energy-bound-from-isoperimetric}
Suppose $H \in C^{1}(\RR^3; \RR)$, and that $\|H\|_{\infty; \RR^3} \leq\Lambda_0$. There exists $\ep_{0} > 0$, depending only on $\Lambda_0$, such that for any $T > 0$, and any $u_0 \in C^{2, \alpha}(\overline{\bB}; \RR^3)$ satisfying $D(u_0) < \ep_{0}$, if $u \in  C^{2 + \alpha, 1 + \frac{\alpha}{2}}(\overline{\bB} \times [0, T); \RR^3)$ is a solution of~\eqref{eq:H-flow}, then we have for all $t \in [0, T)$ that  
\[
\big| V_H(u(\cdot, t), u_0)\big| \leq \frac{1}{8}\big( D(u(\cdot, t)) + D(u_0) \big),
\]
and that
\[
D(u(\cdot, t)) + \int_{0}^{t}\int_{\bB}|u_{t}|^2 dx dt \leq 2D(u_0).
\]
\end{lemm}
\begin{proof}
With $c > 0$ being the constant in the estimate~\eqref{eq:isoperimetric-after-extension}, we require that $\ep_0$ satisfy
\[
c\Lambda_0\ep_0^{\frac{1}{2}} < \frac{1}{16}.
\]
Next, define 
$$T_1 := \sup\{ t \in [0, T) \mid D(u(\cdot, s)) \leq 2\varepsilon_0 \text{ for all }s \in [0, t] \}\,.$$
Since $u$ is of class $C^{2 + \alpha, 1 + \frac{\alpha}{2}}$, and since $D(u(\cdot, 0)) < \ep_{0}$, we see that $T_1 > 0$. Suppose towards a contradiction that $T_1 < T$. Then we have $u \in C^{2 + \alpha, 1 + \frac{\alpha}{2}}(\overline{\bB} \times [0, T_1]; \RR^3)$, and also that
\begin{equation}\label{eq:equality-at-T2}
D(u(\cdot, T_1)) = 2\ep_{0}.
\end{equation}
For all $t \in [0, T_1]$, since $u(\cdot, t)$ and $u(\cdot, 0)$ both lie in $W^{1, 2}\cap L^{\infty}(\bB; \RR^3)$, and also agree on $\partial \bB$, we can apply~\eqref{eq:isoperimetric-after-extension} to deduce that
\begin{align}\label{volume-est-1}
|V_H(u(\cdot, t), u(\cdot, 0))| 
\leq\ & c\, \|H\|_{\infty}\big( D(u(\cdot, t )) + D(u(\cdot, 0)) \big)^{\frac{3}{2}} \notag\\
\leq\ & c\Lambda_0 \cdot (4\ep_{0})^{\frac{1}{2}}\cdot \big( D(u(\cdot, t)) + D(u_0) \big)\notag\\
\leq\ & \frac{1}{8}\big( D(u(\cdot, t)) + D(u_0) \big),
\end{align} 
where the last inequality follows from our choice of $\ep_0$. Combining this with Lemma~\ref{lemm:E-H-monotone} and rearranging, we obtain for all $t \in [0, T_1]$ that
\begin{equation}\label{eq:energy-condition-1}
D(u(\cdot, t)) + \int_0^t \int_{\bB} |u_t|^2 \, dx\, dt
\leq \frac{3}{2} D(u_0) < 2\varepsilon_0,
\end{equation}
which contradicts~\eqref{eq:equality-at-T2}. Therefore, we must have $T_1 = T$, so that 
\[
D(u(\cdot, t)) \leq 2\ep_0\quad\text{for all }t \in [0, T).
\]
The argument leading to~\eqref{volume-est-1} and~\eqref{eq:energy-condition-1} can then be applied to any $t \in [0, T)$. This gives both of the asserted estimates.
\end{proof}
The following long-time existence and uniqueness result is essentially contained in the work of Chen–Levine~\cite{ChenLevine2002}, and suffices for our purposes, even though, to reiterate what was said in the introduction, existence and regularity results are available under much weaker assumptions on $H$ thanks to the work of B\"ogelein--Duzaar--Scheven \cite{BogeleinDuzaarScheven2013,BogeleinDuzaarScheven2015}. 

\begin{prop}\label{prop:existence-flow}
Given $\Lambda_0 > 0$, let $\ep_0$ be as in Lemma~\ref{lemm:energy-bound-from-isoperimetric}. There exists $\ep_1 < \ep_0$, depending only on $\Lambda_0$, such that if $H \in C^1(\RR^3; \RR)$ and $u_0 \in C^{2, \alpha}(\overline{\bB}; \RR^3)$ satisfy
\begin{equation}\label{eq:bounds-for-existence}
\|H\|_{\infty; \RR^3} \leq\Lambda_0, \quad \|\nabla H\|_{\infty; \RR^3} < \infty,\quad  D(u_0) < \ep_1,
\end{equation}
then there exists a unique solution $u \in C^{2 + \alpha, 1 + \frac{\alpha}{2}}(\overline{\bB}\times [0,\infty); \RR^3)$ to the Cauchy-Dirichlet problem~\eqref{eq:H-flow}. This solution satisfies for all $(x, t) \in \bB \times (0, \infty)$ that 
\begin{equation}\label{eq:existence-flow-grad-estimate}
|\nabla u(x, t)|^2 \leq C_{\Lambda_0} \cdot \max\{t^{-1}, (1 - |x|)^{-2}\} \cdot D(u_0),
\end{equation}
and that 
\begin{equation}\label{eq:existence-flow-sup-estimate}
|u(x, t)|^2 \leq C_{\Lambda_0} \cdot \max\{1, t^{-1}\}\cdot \big(D(u_0) + \sup_{\partial \bB}|u_0|^2\big).
\end{equation}
\end{prop}
\begin{proof}
By Proposition~\ref{prop:short-time-existence}, it makes sense to define
\[
T_{\max}: = \sup\{T  > 0\ |\ \text{\eqref{eq:H-flow} has a solution of class $C^{2 + \alpha, 1 + \frac{\alpha}{2}}$ on $\overline{\bB} \times [0, T)$} \},
\]
where it is understood that $T_{\max} = \infty$ when the above set is unbounded, which turns out to be the case as we will demonstrate. To start, we use Lemma~\ref{lemm:uniqueness} and the definition of $T_{\max}$ to see that~\eqref{eq:H-flow} admits a solution $u \in  C^{2 + \alpha, 1 + \frac{\alpha}{2}}(\overline{\bB} \times [0, T_{\max}); \RR^3)$. Requiring that
\[
\ep_1 < \frac{1}{16}\min\{\ep_0, \eta'\},
\]
where $\ep_0$ is the threshold from Lemma~\ref{lemm:energy-bound-from-isoperimetric}, while $\eta' = \eta'(q, \Lambda_0)$ is given by Lemma~\ref{lemm:eta-to-gradient-bound-interior} with $\Lambda_0$ being as in~\eqref{eq:bounds-for-existence}, we get
\begin{equation}\label{eq:smallness-finiteness-condition}
D(u(\cdot, t)) + \int_{0}^{t}\int_{\bB}|u_{t}|^2 dx dt \leq 2D(u_0), \quad \text{for all } t \in [0, T_{\max}),
\end{equation}
and that, by Corollary~\ref{coro:grad-estimate-scaled} together with the above estimate,
\begin{equation}\label{eq:interior-estimate-in-existence-proof}
|\nabla u(x, t)|^2 \leq C_{\Lambda_0} \cdot \frac{D(u_0)}{\min\{t, (1 - |x|)^{2}\}},\quad\text{for all } (x, t) \in \bB\times (0, T_{\max}).
\end{equation}
Next, similar to~\cite[Section V]{Chang1989}, we define
\begin{equation}\label{eq:theta-definition}
\theta_{T}: = \sup_{(x, t) \in \bB \times [0, T]}|\nabla u(x, t)|,\quad\text{for }T \in (0, T_{\max}),
\end{equation}
which is non-decreasing in $T$, and suppose by contradiction that 
\begin{equation}\label{eq:theta-blows-up}
\theta_{T} \to \infty \quad\text{as}\quad T \to T_{\max}^{-}.
\end{equation}
Choose an arbitrary sequence $(T_{k})$ that increases strictly to $T_{\max}$. For each $k$ sufficiently large, let $(x_k, t_k)$ be a point in $\overline{\bB} \times [0, T_{k}]$ at which
\[
|\nabla u(x_k, t_k)|  = \theta_{T_{k}} =: r_{k}^{-1}.
\]
Since $\theta_{T_{k}} \to \infty$, while $|\nabla u|$ is bounded on compact subsets of $\overline{\bB} \times [0, T_{\max})$, we must have $t_k \to T_{\max}^{-}$. Thus, taking a subsequence if needed, we may assume that the sequence $(t_k)$ is strictly increasing, in which case~\eqref{eq:interior-estimate-in-existence-proof} gives
\[
r_k^{-2}(1 - |x_k|)^2 \leq C_{\Lambda_0} \cdot D(u_0) \cdot \max\{1, t_1^{-1}\}.
\]
Consequently $(x_k)$ converges to some $z_0 \in \partial \bB$. Letting $z_k = |x_k|^{-1}x_k$, we see that $z_k \to z_0$ as well. Moreover, the above estimate implies the existence of some $L \in (1, \infty)$ such that 
\begin{equation}\label{eq:closer-to-boundary}
\frac{|x_k - z_k|}{r_k} \leq L,\quad \text{for all }k.
\end{equation}

Our next step is to perform a sequence of rescalings. Denote by $\Omega_{R}$ the open disk $\bB_{R}({\rm i}R) \subset \CC \simeq\RR^2$. Then since $z_k \in \partial \bB$, we can find a rotation $A_k: \RR^2 \to \RR^2$ such that 
\[
z_k + A_{k}(\Omega_1) = \bB.
\]
Letting $R_k := (4Lr_k)^{-1}$ and
\[
\Phi_{k}(y): = z_k + R_{k}^{-1}A_k(y), \quad\text{for }y \in \overline{\Omega_{R_{k}}},
\]
we see from~\eqref{eq:closer-to-boundary} that 
\begin{equation}\label{eq:y-k-quarter}
y_k: = \Phi_{k}^{-1}(x_k) \in \overline{\bB_{\frac{1}{4}}} \cap \overline{\Omega_{R_k}}.
\end{equation}
Thus, taking a subsequence if needed, we have that
\begin{equation}\label{eq:max-point-scaled-limit}
y_k \to y_0 \in \overline{\bB_{\frac{1}{4}}^+} \quad \text{as }k \to \infty.
\end{equation}
Rescaling $u$ by setting
\[
\widetilde{u}_{k}(y, s) = u(\Phi_{k}(y), t_k + R_k^{-2} s), \quad\text{for }(y, s) \in \overline{\Omega_{R_k}} \times (-R_k^2 t_k, 0],
\]
then~\eqref{eq:H-flow} implies
\begin{equation}\label{eq:flow-equation-blown-up}
\left\{
\begin{array}{ll}
(\widetilde{u}_{k})_{t} - \Delta\widetilde{u}_{k} = -2(H\circ\widetilde{u}_{k})\cdot (\widetilde{u}_{k})_{x^1} \times (\widetilde{u}_{k})_{x^2},& \text{ on }\Omega_{R_k} \times (-R_k^2 t_k, 0],\\
\\
\widetilde{u}_{k}(y, s) = u_0(\Phi_{k}(y))= : \widetilde{u}_{0, k}(y),&\text{for }(y, s) \in \partial\Omega_{R_k} \times (-R_k^2 t_k, 0].
\end{array}
\right.
\end{equation}
Since $t_k \to T_{\max}$ and $R_k \to \infty$, we have $R_{k}^2 t_k \to \infty$. Also, by our choice of $(x_k, t_k)$, we have
\begin{equation}\label{eq:gradient-max-attained}
|\nabla \widetilde{u}_{k}(y_k, 0)| = 4L,
\end{equation}
and that
\begin{equation}\label{eq:gradient-max-rescaled}
|\nabla \widetilde{u}_{k}(y, s)| \leq 4 L, \quad\text{for all }(y, s) \in \overline{\Omega_{R_k}} \times (-R_k^2t_k, 0].
\end{equation}
As a result of the gradient bound, we also have
\begin{equation}\label{eq:Linfty-bound}
\begin{split}
|\widetilde{u}_{k}(y, s)| \leq\ & |\widetilde{u}_{k}(y, s) - \widetilde{u}_k(0, s)| + |u_0(z_k)|\\
\leq\ & 4L \cdot |y| + \|u_0\|_{\infty;\bB}, \quad \text{for all }(y, s) \in \Omega_{R_k} \times (-R_k^2t_k, 0].
\end{split}
\end{equation}
As for the rescaled boundary data, we have for all $\rho > 1$ that 
\begin{equation}\label{eq:C-D-data-scaled-1}
\begin{split}
\lim_{k \to \infty}|\widetilde{u}_{0, k} - u_0(z_0)|_{2, \alpha; \bB_{\rho} \cap \Omega_{R_k}} = 0,
\end{split}
\end{equation}
\begin{equation}\label{eq:C-D-data-scaled-2}
\limsup_{k \to \infty}|\widetilde{u}_{0, k}|_{2, \alpha; \bB_{\rho} \cap \Omega_{R_k}}\leq |u_{0}|_{2, \alpha; \bB}.
\end{equation}
In particular, eventually we may apply the gradient H\"older estimate from Lemma~\ref{lemm:gradient-holder} to deduce that
\[
\begin{split}
[\nabla \widetilde{u}_{k}]_{\alpha, \frac{\alpha}{2}; (\bB_{\frac{1}{4}} \cap \Omega_{R_k} )\times (-\frac{1}{16}, 0]}\leq\ & C_{L,\Lambda_0}\cdot ( \|\nabla\widetilde{u}_{k}\|_{2; (\bB \cap \Omega_{R_k}) \times (-1, 0]} + \|\widetilde{u}_{0, k}\|_{2,q; \bB \cap \Omega_{R_k}})\\
\leq\ & C_{L,\Lambda_0}\cdot(L + |u_0|_{2, \alpha; \bB}),
\end{split}
\]
Recalling that $y_k \in \overline{\bB_{\frac{1}{4}}} \cap \overline{\Omega_{R_k}}$, we obtain from the above estimate and~\eqref{eq:gradient-max-attained} some $\tau_0 \in (0, \frac{1}{16})$, independent of $k$, such that
\begin{equation}\label{eq:gradient-lowerbound}
|\nabla \widetilde{u}_{k}(y_k, s)| \geq  2L, \quad\text{for all } s \in [-\tau_0, 0]\text{ and large enough }k.
\end{equation}
Again using~\eqref{eq:gradient-max-rescaled},~\eqref{eq:Linfty-bound}, along with~\eqref{eq:C-D-data-scaled-2} and the finiteness of $\|H\|_{\infty} + \|\nabla H\|_{\infty}$, we can apply the $C^{2+\alpha, 1+\frac{\alpha}{2}}$-estimates from Lemmas~\ref{lemm:gradient-holder} and~\ref{lemm:gradient-holder-interior} to see that, for all $\rho > 1$ and for any sequence $(s_k)$ in $[-\tau_0, 0]$, there holds
\begin{equation}\label{eq:slices-C2a-bound}
\limsup_{k \to \infty}\big(|\widetilde{u}_{k}(\cdot, s_k)|_{2, \alpha; \bB_{\rho} \cap \Omega_{R_k}} + |(\widetilde{u}_{k})_{t}(\cdot, s_k)|_{0, \alpha;  \bB_{\rho} \cap \Omega_{R_k}} \big) < \infty.
\end{equation}
To choose $(s_k)$ with which to apply the above, note that by~\eqref{eq:smallness-finiteness-condition} and a change of variables, we have (regardless of whether $T_{\max} = \infty$ or $T_{\max} < \infty$)
\[
\int_{-\tau_0}^{0}\int_{\Omega_{R_k}}|(\widetilde{u}_{k})_{t}|^2 = \int_{t_k - R_k^{-2}\tau_0}^{t_k}\int_{\bB}|u_t|^2\to 0,\quad \text{as }k \to \infty.
\]
As a result we obtain a sequence $(s_k)$ in $[-\tau_0, 0]$ such that 
\begin{equation}\label{eq:good-slice}
\lim_{k \to \infty}\int_{\Omega_{R_k}}|(\widetilde{u}_{k})_{t}(\cdot, s_{k})|^2 = 0.
\end{equation}

Towards extracting a limit out of the above estimates, consider, for each $k$, the map
\[
F_{k}(z) := \frac{2{\rm i}R_k z}{2{\rm i}R_k + z},
\]
which is biholomorphic from the upper half-plane onto $\Omega_{R_k}$, and converges smoothly locally on $\overline{\RR^2_+}$ to the identity map as $k \to \infty$. Also, given $\rho > 1$, there holds for all sufficiently large $k$ that
\begin{equation}\label{eq:F_k-images}
\bB_{\frac{r}{2}} \cap \Omega_{R_{k}} \subset F_{k}(\bB_{r}^{+}) \subset \bB_{2r} \cap \Omega_{R_{k}},\quad\text{for all }r \in (0, \rho].
\end{equation}
In particular, with $y_{k}$ as in~\eqref{eq:y-k-quarter}, we have eventually that $p_{k}:=F_{k}^{-1}(y_{k}) \in \overline{\bB_{\frac{1}{2}}^{+}}$, and thus
\begin{equation}\label{eq:pre-image-of-max-point}
|p_{k} - y_0| \leq |p_{k} - F_{k}(p_k)| + |y_{k} - y_0| \to 0 \quad\text{as }k \to \infty.
\end{equation}
Now let 
\[
v_{k} = \widetilde{u}_{k}(\cdot, s_{k}) \circ F_k,\quad w_{k} = (\widetilde{u}_{k})_{t}(\cdot, s_{k})\circ F_k.
\]
Then by~\eqref{eq:good-slice} and~\eqref{eq:slices-C2a-bound}, up to taking a subsequence, there exists $v \in C^{2, \alpha}(\overline{\RR^2_+})$ such that
\begin{equation}\label{eq:C2-convergence-of-good-slices}
v_{k} \to v \quad\text{in }C^2_{\loc}(\overline{\RR^2_+}),\quad\quad w_{k}\to 0 \quad \text{in }C^0_{\loc}(\overline{\RR^2_+}).
\end{equation}
Furthermore, with $\lambda_{k}: = |(F_{k})_{z}|$, we compute using~\eqref{eq:flow-equation-blown-up} that
\[
\lambda_{k}^{-2}\Delta v_{k} = 2\lambda_{k}^{-2}H(v_{k})(v_{k})_{x^1} \times (v_{k})_{x^2} + w_{k},
\]
which together with~\eqref{eq:C-D-data-scaled-1} shows that the limiting function $v$ is a solution of
\begin{equation}\label{eq:blown-up-PDE}
\left\{
\begin{array}{ll}
\Delta v = 2H(v)v_{x^1} \times v_{x^2}& \text{ in }\RR^2_+,\\
v = u_0(z_0)& \text{ on }\partial \RR^2_+.
\end{array}
\right.
\end{equation}
Using the energy upper bound in~\eqref{eq:smallness-finiteness-condition} and the conformal invariance of the Dirichlet integral, we have for all $\rho > 1$ that
\[
\int_{\bB_{\rho}^{+}}|\nabla v_{k}|^2  \leq \int_{\bB} |\nabla u(\cdot, t_k + R_k^{-2}s_k)|^2  < 4\ep_1,
\]
and consequently
\begin{equation}\label{eq:blown-up-finite-energy}
\int_{\RR^2_+}|\nabla v|^2 < \infty.
\end{equation}
From~\eqref{eq:blown-up-PDE},~\eqref{eq:blown-up-finite-energy}, and a generalization of Wente's uniqueness result~\cite{Wente1975} due to Wang~\cite[Lemma 4.2]{WangGuoFang1992}, which we recall in Lemma~\ref{lemm:Wente-constant}, we see that $v$ must be constant. However, by~\eqref{eq:pre-image-of-max-point}, and noting also that
\[
|\nabla v_{k}(p_{k})| = \lambda_{k}(p_k)\cdot|\nabla \widetilde{u}_{k}(y_{k}, s_{k})|,
\]
we get from~\eqref{eq:gradient-lowerbound} (with $s$ taken to be $s_k$) that $|\nabla v(y_0)| \geq 2L > 0$, which is a contradiction. This being the result of assuming~\eqref{eq:theta-blows-up}, we conclude that 
\[
\lim_{T \to T_{\max}^{-}}\theta_T < \infty.
\]
Combining this with the boundary condition $u(\cdot, t) = u_0$ on $\partial \bB$, we get
\begin{equation}\label{eq:global-gradient-bound}
\sup_{(x, t) \in \bB \times (0, T_{\max})}|u(x, t)| + |\nabla u(x, t)| < \infty.
\end{equation}
Thus we may apply Lemma~\ref{lemm:gradient-holder} and Lemma~\ref{lemm:gradient-holder-interior} to see that, similar to~\eqref{eq:slices-C2a-bound}, there holds
\begin{equation}\label{eq:slices-form-bounded-set}
\limsup_{t \to T_{\max}^{-}}\big(|u(\cdot, t)|_{2, \alpha; \bB} + |u_{t}(\cdot, t)|_{0,\alpha; \bB} \big) < \infty.
\end{equation}
Proposition~\ref{prop:short-time-existence} then implies the existence of some $\tau > 0$ such that for all $t_0 \in (0, T_{\max})$ close to $T_{\max}$, the Cauchy--Dirichlet problem~\eqref{eq:H-flow}, with $u_0$ replaced by $u(\cdot, t_0)$, admits a solution of class $C^{2 + \alpha, 1 + \frac{\alpha}{2}}$ on $\overline{\bB} \times [0, \tau]$. This forces $T_{\max} = \infty$, and we have obtained a solution to~\eqref{eq:H-flow} in $C^{2+\alpha, 1 + \frac{\alpha}{2}}(\overline{\bB} \times [0, \infty))$. That there is no other solution in $C^{2+\alpha, 1 + \frac{\alpha}{2}}(\overline{\bB} \times [0, \infty);\RR^3)$ follows again from Lemma~\ref{lemm:uniqueness}.

The asserted gradient estimate~\eqref{eq:existence-flow-grad-estimate} is just a restatement of~\eqref{eq:interior-estimate-in-existence-proof}, while the estimate~\eqref{eq:existence-flow-sup-estimate} on $|u|$ is essentially contained in the work of Qing~\cite{Qing1993}. Specifically, in~\cite[Proposition 1]{Qing1993}, the map is assumed to be weakly harmonic only to make available the interior gradient estimate in~\cite[Lemma 4]{Qing1993}. Replacing the latter with~\eqref{eq:existence-flow-grad-estimate}, and using also the energy bound in~\eqref{eq:smallness-finiteness-condition}, we may follow the proof of~\cite[Proposition 1]{Qing1993}, with $u$ and $P$ there taken to be $u(\cdot, t)$ and the origin in $\RR^3$, respectively, to get the desired bound on $|u|$. 
\end{proof}

In the next two results, we study the behavior for large time of the global regular solution constructed in Proposition~\ref{prop:existence-flow}. In particular, Proposition~\ref{prop:flow-convexity} is a parabolic version of the convexity estimate in Theorem~\ref{thm:H-convexity}.
\begin{lemm}
\label{lemm:speed-monotone}
Suppose $H \in C^{1}(\RR^3; \RR)$ is such that
\[
\|H\|_{\infty; \RR^3} \leq \Lambda_0 ,\quad  \|\nabla H\|_{\infty;\RR^3} \leq \Lambda_1,
\]
and let $\ep_1$ be as in Proposition~\ref{prop:existence-flow}.  There exists $\ep_2 < \ep_1$, depending only on $\Lambda_0$ and $\Lambda_1$, such that if $u_0 \in C^{2, \alpha}(\overline{\bB} ;\RR^3)$ satisfies $D(u_0) < \ep_2$, and if $u \in C^{2 + \alpha, 1 + \frac{\alpha}{2}}(\overline{\bB} \times [0, \infty); \RR^3)$ denotes the solution to~\eqref{eq:H-flow} produced by Proposition~\ref{prop:existence-flow}, then we have the following.
\vskip 1mm
\begin{enumerate}
\item[(a)] For all $T > 1 > \delta > 0$, the distributional derivative $\nabla u_{t}$ belongs to $L^2(\bB \times (\delta, T))$, and satisfies
\begin{equation}\label{eq:nabla-u-t-L2}
\int_{\delta}^{T}\int_{\bB}|\nabla u_{t}|^2\, dxdt \leq \frac{4}{\delta^2}\int_{\frac{\delta}{2}}^{\delta}\int_{\bB}|u_{t}|^2\, dxdt.
\end{equation}
\vskip 1mm
\item[(b)] The function $t \mapsto \max\{t^{-1}, 1\}\cdot\int_{\bB}|u_{t}(\cdot, t)|^2\, dx$ is non-increasing on $(0, \infty)$. As a result, we have for all $t > s > 0$ that 
\begin{equation}\label{eq:speed-monotone-consequence}
\int_{\bB}|u_{t}(\cdot, t)|^2\, dx \leq \frac{4}{t-s}\int_{s}^{t}\int_{\bB} |u_{t}|^2\, dxdt.
\end{equation}
\end{enumerate}
\end{lemm}
\begin{proof} 
We follow~\cite[Lemma 4.1]{LuWang2012} and employ second-order difference quotients. Given $\tau > 0$, for all nonnegative $\phi \in C^{\infty}_{c}((\frac{\tau}{4}, \infty))$ and $0 < |h| < \frac{\tau}{16}$, we let $v_{h}$ be as in the proof of Lemma~\ref{lemm:E-H-monotone} and consider, in the notation of~\eqref{eq:diff-quotient-in-t}, the map $-\frac{1}{h}((v_{h}\phi)^{(h)} - v_{h}\phi)$, which acts by
\[
(x, t) \longmapsto -\frac{v_{h}(x, t + h)\phi(t+h) - v_{h}(x, t)\phi(t)}{h}.
\]
Multiplying~\eqref{eq:H-flow} by the above, which vanishes whenever $t \leq \frac{3\tau}{16}$, and using the boundary condition~\eqref{eq:t-step-vanish-on-boundary} to integrate by parts in the term involving $\Delta u$, we get
\begin{align*}
&\int_{0}^{\infty}\int_{\bB} v_{h} \cdot (v_{h})_{t}  \, \phi \,dxdt + \int_{0}^{\infty}\int_{\bB} |\nabla v_{h}|^2 \phi \,dxdt\\
& = \int_{0}^{\infty}\int_{\bB} -\frac{2}{h} \cdot \big( H(u)\,  u_{x^1} \times u_{x^2} - H(u^{(-h)})\,  u^{(-h)}_{x^1} \times u^{(-h)}_{x^2}\big) \cdot v_{h}\phi \, dxdt.\nonumber
\end{align*}
To bound the right-hand side, note by our choice of $h$ that $t - h > \frac{3}{4}t$ whenever $t > \frac{\tau}{4}$, in which case the estimate~\eqref{eq:existence-flow-grad-estimate} from Proposition~\ref{prop:existence-flow} gives
\begin{align*}
|\nabla u(x, t)|^2 + |\nabla u(x, t-h)|^2 \leq\ & C_{\Lambda_0}\ep_{2}\cdot \big((1 - |x|)^{-2} + 1 + (t^{-1}-1)_{+} \big)\\
\leq\ & C_{\Lambda_0}\ep_{2} \cdot \big(2\cdot (1 - |x|)^{-2} + (t^{-1}-1)_{+} \big)
\end{align*}
Thus, we have on $\bB \times (\frac{\tau}{4}, \infty)$ that
\begin{align*}
&2|h|^{-1} |v_h| \cdot \Big|H(u)\,  u_{x^1} \times u_{x^2}  - H(u^{(-h)})\,  u^{(-h)}_{x^1} \times u^{(-h)}_{x^2} \Big|\\
\leq\ & 2\Lambda_1 |v_{h}|^2|\nabla u|^2 + 2\Lambda_0 |v_{h}||\nabla v_{h}|  ( |\nabla u| + |\nabla u^{(-h)}| )\\
\leq\ &  \frac{1}{4}|\nabla v_h|^2 + (2\Lambda_1 + 4\Lambda_0^2)\, |v_h|^2 (|\nabla u| + |\nabla u^{(-h)}|)^2 \\
\leq\ & \frac{1}{4}|\nabla v_h|^2 + C_{\Lambda_0, \Lambda_1} \ep_{2} \cdot \frac{|v_{h}|^2}{(1 - |x|)^2} + C_{\Lambda_0, \Lambda_1}\ep_{2}\cdot  |v_{h}|^2 (t^{-1} - 1)_{+}.
\end{align*}
With the help of the Hardy inequality~\eqref{eq:Hardy}, applicable because of the boundary condition~\eqref{eq:t-step-vanish-on-boundary}, we see that
\begin{align*}
&\int_{0}^{\infty}\int_{\bB} v_{h} \cdot (v_{h})_{t}  \, \phi \,dxdt + \int_{0}^{\infty}\int_{\bB} |\nabla v_{h}|^2 \phi \,dxdt, \\
&\leq \big(\frac{1}{4} + C\ep_2\big) \int_{0}^{\infty}\int_{\bB} |\nabla v_{h}|^2\phi \, dxdt + C\ep_{2}\int_{0}^{\infty}\int_{\bB} (t^{-1}-1)_{+} \cdot|v_{h}|^2\phi\, dxdt,
\end{align*}
where the constants $C$ depend only on $\Lambda_0$ and $\Lambda_1$. Requiring that $C\ep_2 < \frac{1}{4}$, and noting that the choice of nonnegative $\phi \in C^{\infty}_{c}((\frac{\tau}{4}, \infty))$ is arbitrary, we arrive at
\[
\frac{d}{dt}\Big(\int_{\bB \times \{t\}} |v_{h}|^2\, dx\Big) + \int_{\bB \times \{t\}} |\nabla v_{h}|^2\, dx \leq (t^{-1} -1)_{+}\cdot \int_{\bB\times \{t\}} |v_{h}|^2\, dx,
\]
whenever $0 < |h|< \frac{\tau}{16}$ and $t > \frac{\tau}{4}$. Observing that the function $a(t): = \max\{t^{-1}, 1\}$ is absolutely continuous on $(0, \infty)$ and satisfies
\[
a'(t) + a(t) \cdot (t^{-1}-1)_{+} \leq 0, \quad\text{whenever }t \neq 1,
\]
we deduce further that
\[
\frac{d}{dt}\Big( a(t)\int_{\bB \times \{t\}} |v_{h}|^2\, dx \Big) + a(t)\int_{\bB \times \{t\}} |\nabla v_h|^2\, dx \leq 0.
\]
Now, given $[t_1, t_2] \subset (0, \infty)$ and $s \in [\frac{t_1}{2}, t_1]$, the above differential inequality is valid on $[s, t_2]$ for all $0 < |h| < \frac{t_1}{16}$. Integrating leads to
\begin{equation}\label{eq:vh-diff-ineq-integrated}
a(t_2)\int_{\bB \times \{t_2\}} |v_h|^2\, dx + \int_{s}^{t_2}\int_{\bB}a(t)|\nabla v_h|^2\, dxdt \leq a(s) \int_{\bB \times \{s\}} |v_{h}|^2\, dx
\end{equation}
Dropping the first term on the left-hand side, and averaging with respect to $s$ over $[\frac{t_1}{2}, t_1]$, we obtain
\begin{align*}
a(t_2)\int_{t_1}^{t_2}\int_{\bB}|\nabla v_h|^2\, dxdt \leq\ & \int_{t_{1}}^{t_2}\int_{\bB} a(t)|\nabla v_{h}|^2\, dxdt \\
\leq\ & \frac{2}{t_1}\int_{\frac{t_1}{2}}^{t_1} a(t)\int_{\bB \times \{t\}} |v_{h}|^2\, dxdt \leq \frac{2}{t_1} \cdot a(\frac{t_1}{2})\int_{\frac{t_1}{2}}^{t_1}\int_{\bB} |v_{h}|^2\, dxdt.
\end{align*}
As $h \to 0$, the right-most term converges to the corresponding integral of $|u_{t}|^2$, and thus stays bounded uniformly in $h$. The conclusions of part (a) are then standard~\cite[Chapter 5.8.2, Theorem 3]{Ev}. Taking $s = t_1$ in~\eqref{eq:vh-diff-ineq-integrated}, dropping the second term on the left-hand side, and sending $h \to 0$, we get the first conclusion of (a). To prove~\eqref{eq:speed-monotone-consequence} in the case $s \in [\frac{t}{2}, t)$, note that, by the monotonicity we just established, there holds
\[
a(t)\int_{\bB}|u_{t}(\cdot, t)|^2\, dx \leq \frac{a(\frac{t}{2})}{t-s}\int_{s}^{t}\int_{\bB}|u_{t}|^2\, dxdt,
\]
which immediately gives~\eqref{eq:speed-monotone-consequence} since $a(\frac{t}{2})\leq 2a(t)$. In the case $s \in (0, \frac{t}{2})$, we apply the above with ``$s$'' on the right-hand side taken to be $\frac{t}{2}$, which gives
\[
\int_{\bB}|u_{t}(\cdot, t)|^2\, dx \leq \frac{4}{t}\int_{\frac{t}{2}}^{t}\int_{\bB} |u_{t}|^2\, dxdt \leq \frac{4}{t-s}\int_{s}^{t}\int_{\bB} |u_{t}|^2\, dxdt,
\]
so again~\eqref{eq:speed-monotone-consequence} holds. 
\end{proof}

\begin{prop}\label{prop:flow-convexity}
Let $\Lambda_0, \Lambda_1 > 0$ and $H \in C^{1}(\RR^3; \RR)$ be as in Lemma~\ref{lemm:speed-monotone}. There exists $\ep_3 < \ep_2$, depending only on $\Lambda_0$ and $\Lambda_1$, with the following property. Suppose $u_0 \in C^{2, \alpha}(\overline{\bB}; \RR^3)$ is such that $D(u_0) < \ep_3$, and let $u \in  C^{2 + \alpha, 1 + \frac{\alpha}{2}}(\overline{\bB}\times [0,\infty), \RR^3)$ be the unique global regular solution to~\eqref{eq:H-flow} given by Proposition~\ref{prop:existence-flow}. Then we have for all $t \geq s \geq 1$ that
\begin{equation}\label{eq:parabolic-convexity}
\frac{1}{8}\int_{\bB}|\nabla u(\cdot, s) - \nabla u(\cdot, t)|^2\, dx \leq E_H(u(\cdot, s), u_0) - E_H (u(\cdot, t), u_0)\,.
\end{equation}
\end{prop}
\begin{proof}
The idea is taken from~\cite[page 5182]{LuWang2012} and~\cite[Section 2]{Lin2013}. Let
\[
u(t): = u(\cdot, t), \quad u'(t):  = u_{t}(\cdot, t).
\]
Given $t > s \geq 1$, using~\eqref{eq:H-flow} and the fact that $u(s) = u(t)$ on $\partial \bB$, and following the derivation of~\eqref{eq:D-difference}, we obtain
\begin{align}\label{eq:parabolic-D-difference}
D(u(s)) - D(u(t)) =\ & \frac{1}{2}\int_{\bB}|\nabla u(s) - \nabla u(t)|^2\, dx - \int_{\bB} u'(t) \cdot(u(s) - u(t)) \,dx\nonumber\\
& - \int_{\bB}2H(u(t))\, u(t)_{x^1} \times u(t)_{x^2} \cdot (u(s) - u(t))\, dx.
\end{align}
For the second term on the right-hand side, by the fundamental theorem of Calculus we have
\begin{equation}\label{eq:u-difference-L2-by-ut}
\int_{\bB}|u(x,s) - u(x, t)|^2\, dx \leq \int_{\bB}\Big( \int_{s}^{t}|u_{t}(x, \tau)|\, d\tau \Big)^2\, dx \leq (t-s)\int_{s}^{t}\int_{\bB}|u'|^2\, dx dt,
\end{equation}
which combines with Lemma~\ref{lemm:speed-monotone}(b) to give
\begin{equation}\label{eq:flow-convexity-speed-term}
\begin{split}
\Big(\int_{\bB} |u(s) - u(t)|^2\,dx \Big)^{\frac{1}{2}} \Big( \int_{\bB} |u'(t)|^2 \,dx\Big)^{\frac{1}{2}} \leq \ & 2\int_{s}^{t}\int_{\bB}|u'|^2\,dxdt.
\end{split}
\end{equation}

To estimate the contribution of the enclosed volume term, we further let $v = u(s)$, $w = u(t)$, and $\xi = v - w$. The gradient estimate~\eqref{eq:existence-flow-grad-estimate} in Proposition~\ref{prop:existence-flow} then implies
\begin{equation}\label{eq:grad-estimate-on-slices}
|\nabla w|^2 \leq C_{\Lambda_0} \ep_{3} \cdot \big( \frac{1}{(1 - |x|)^2} + \frac{1}{t-s} \big),
\end{equation}
and we deduce using~\eqref{eq:Hardy} and~\eqref{eq:u-difference-L2-by-ut} that
\begin{equation}\label{eq:xi-estimate-for-volume-term}
\begin{split}
\int_{\bB} |\xi|^2 |\nabla w|^2\, dx \leq \ & C_{\Lambda_0}\ep_{3}\int_{\bB}\frac{|\xi|^2}{(1 - |x|)^2} + \frac{|\xi|^2}{t-s}\, dx\\
\leq\ & C_{\Lambda_0}\ep_{3}\int_{\bB} |\nabla \xi|^2\, dx + C_{\Lambda_0}\ep_{3} \int_{s}^{t}\int_{\bB}|u'|^2\, dxdt.
\end{split}
\end{equation}
Now, by Lemma~\ref{lemm:volume-formula-H}, we have
\begin{align}
&V_{H}(v, w) - 2\int_{\bB} H(w)\, (v-w) \cdot w_{x^1} \times w_{x^2}\, dx\nonumber\\
= \ & 2\int_{0}^{1}\underbrace{\Big(\int_{\bB} \big( H(w + \lambda\xi) - H(w) \big)\, \xi \cdot w_{x^1} \times w_{x^2}\, dx\Big)}_{I_{1}(\lambda)} d\lambda \nonumber\\
& + 2\int_{0}^{1}\underbrace{\Big(\int_{\bB} H(w + \lambda\xi)\, \xi \cdot \big( (w+\lambda\xi)_{x^1} \times (w + \lambda\xi)_{x^2} - w_{x^1} \times w_{x^2} \big)\, dx\Big)}_{I_{2}(\lambda)} d\lambda. \label{eq:flow-convexity-volume}
\end{align}
By~\eqref{eq:xi-estimate-for-volume-term}, there holds for all $\lambda \in [0, 1]$ that
\begin{equation}\label{eq:I1s-estimate}
\begin{split}
|I_{1}(\lambda)| \leq\ & \int_{\bB} \|\nabla H\|_{\infty} \cdot |\xi|^2 |\nabla w|^2\, dx \leq\ C_{\Lambda_0, \Lambda_1} \ep_{3} \Big( \int_{\bB}|\nabla \xi|^2\, dx + \int_{s}^{t}\int_{\bB}|u'|^2\, dxdt\Big).
\end{split}
\end{equation}
Turning to $I_{2}(\lambda)$, we notice that
\[
(w+\lambda\xi)_{x^1} \times (w + \lambda\xi)_{x^2} - w_{x^1} \times w_{x^2}  = \lambda(w_{x^1} \times \xi_{x^2} + \xi_{x^1} \times w_{x^2}) + \lambda^2 \xi_{x^1} \times \xi_{x^2},
\]
so that 
\[
\begin{split}
I_{2}(\lambda) =\ & \lambda\int_{\bB} H(w + \lambda\xi)\, \xi \cdot (w_{x^1} \times \xi_{x^2} + \xi_{x^1} \times w_{x^2})\, dx + \lambda^2 \int_{\bB} H(w + \lambda\xi)\, \xi \cdot \xi_{x^1} \times \xi_{x^2}\, dx\\
=:\ & I_{21}(\lambda) + I_{22}(\lambda).
\end{split}
\]
By H\"older's inequality, we have for all $\lambda \in [0, 1]$ that
\begin{equation}\label{eq:I21-estimate}
\begin{split}
|I_{21}(\lambda)| \leq C_{\Lambda_0} \| |\xi||\nabla w| \|_{2; \bB}\cdot \|\nabla \xi\|_{2; \bB} \leq C_{\Lambda_0} \sqrt{\ep_{3}} \cdot \big( \|\nabla \xi\|_{2; \bB}^2 + \|u'\|_{2; \bB \times [s, t]}^2 \big).
\end{split}
\end{equation}
On the other hand, Wente's inequality (see~\cite[Lemma A.1]{BrezisCoron1984} or~\cite[Theorem 3.1.2]{Helein2002}) shows that the solution $\phi$ to the Dirichlet problem
\[
\left\{
\begin{array}{ll}
\Delta \phi = \xi_{x^1} \times \xi_{x^2},& \text{ in }\bB,\\
\phi = 0, & \text{ on }\partial \bB,
\end{array}
\right.
\]
satisfies for some universal constant $C$ that
\[
\|\nabla \phi\|_{2} \leq C\|\nabla \xi\|_{2}^2.
\]
Since $\xi = 0$ on $\partial \bB$, an integration by parts yields
\begin{equation}\label{eq:I22-estimate-1}
|I_{22}(\lambda)| \leq \Big| \int_{\bB} \bangle{\nabla (H(w + \lambda\xi)\, \xi), \nabla \phi}\, dx \Big| \leq C\|\nabla (H(w + \lambda\xi)\xi)\|_{2} \cdot \|\nabla \xi\|_{2}^2.
\end{equation}
Recalling that $w + \lambda\xi = \lambda v + (1-\lambda)w$, and using also the gradient estimate~\eqref{eq:existence-flow-grad-estimate} and the assumption $t , s \geq 1$, we have
\begin{equation}\label{eq:t-s-geq-1-used}
\begin{split}
|\nabla (H(w + \lambda\xi)\, \xi)| \leq \ &\Lambda_1(|\nabla v| + |\nabla w|)|\xi| + \Lambda_0 |\nabla \xi|\\
\leq\ & C_{\Lambda_0, \Lambda_1}\sqrt{\ep_{3}} \cdot \frac{|\xi|}{(1 - |x|)} + \Lambda_0 |\nabla \xi|.
\end{split}
\end{equation}
Thus, again using~\eqref{eq:Hardy} to bound the $L^2$-norm of $(1 - |x|)^{-1}|\xi|$, we infer that
\begin{align*}
\|\nabla (H(w + \lambda\xi)\, \xi) \|_{2} \leq C_{\Lambda_0, \Lambda_1}(\sqrt{\ep_{3}} + 1)\|\nabla \xi\|_{2} \leq\ & C_{\Lambda_0, \Lambda_1}(\sqrt{\ep_{3}} + 1)(\|\nabla v\|_{2} + \|\nabla w\|_{2})\\
\leq\ & C_{\Lambda_0, \Lambda_1}(\sqrt{\ep_{3}} + 1)\cdot 4\sqrt{\ep_{3}},
\end{align*}
where the last inequality follows from Lemma~\ref{lemm:energy-bound-from-isoperimetric}. Substituting this into~\eqref{eq:I22-estimate-1} gives
\[
|I_{22}(\lambda)| \leq C_{\Lambda_0, \Lambda_1}\sqrt{\ep_{3}} \cdot \|\nabla \xi\|_{2}^2 \quad\text{for all }\lambda \in [0, 1], 
\]
which together with~\eqref{eq:I21-estimate},~\eqref{eq:I1s-estimate}, and~\eqref{eq:flow-convexity-volume} leads to
\[
\Big| V_{H}(v, w) - 2\int_{\bB} H(w)\, (v- w) \cdot w_{x^1} \times w_{x^2}\, dx \Big| \leq C_{\Lambda_0, \Lambda_1}\sqrt{\ep_{3}} \cdot (\|\nabla v - \nabla w\|_{2}^2 + \|u'\|_{2; \bB \times [s, t]}^2).
\]
Taking~\eqref{eq:parabolic-D-difference} and~\eqref{eq:flow-convexity-speed-term} into account, we arrive at
\begin{equation}\label{eq:parabolic-Ef-2}
\begin{split}
&E_H(u(s), u_0) - E_H(u(t), u_0)\\
\geq\ & \big( \frac{1}{2} - C_{\Lambda_0, \Lambda_1}  \sqrt{\ep_3}  \big) \int_{\bB}|\nabla u(s) - \nabla u(t)|^2\,dx - (2 + C_{\Lambda_0, \Lambda_1}\sqrt{\ep_3})\int_{s}^{t}\int_{\bB}|u'|^2\, dxdt.
\end{split}
\end{equation}
Recalling Lemma~\ref{lemm:E-H-monotone} and rearranging, we see that~\eqref{eq:parabolic-convexity} holds provided $\ep_3$ is sufficiently small depending only on $\Lambda_0$ and $\Lambda_1$.
\end{proof}

\section{Weak solutions to \texorpdfstring{$H$}{H}-surface flow with small energy data}\label{sec:existence-uniformity}

Throughout this section, we assume that $H \in C^1(\RR^3; \RR)$ and that $\|H\|_{\infty; \RR^3} + \|\nabla H\|_{\infty; \RR^3}$ is finite. Our notion of weak solutions for~\eqref{eq:H-flow} is explained in Definition~\ref{defi:weak-solution} below. Some preparatory remarks are in order. Given $0 < T < \infty$ and $u \in W^{1, 2}(\bB \times (0, T); \RR^3)$, we can always find a sequence $(u_n)$ in $C^{\infty}(\overline{\bB} \times [0, T];\RR^3)$ such that
\begin{equation}\label{eq:smooth-approximation}
\lim_{n \to \infty}\int_{0}^{T}\int_{\bB} |u_n - u|^2 + |\nabla u_n - \nabla u|^2 + |(u_n)_{t} - u_{t}|^2\, dxdt = 0.
\end{equation}
For instance, extending $u$ by even reflection, first across $\bB \times \{T\}$ and then across $\bB \times \{0\}$, we obtain a map $\overline{u}$ that lies in $W^{1, 2}(\bB \times (-2T, 2T); \RR^3)$. The required sequence can then be produced by mollifying dilated maps of the form $(x, t) \mapsto \overline{u}(\frac{x}{1 + \delta},\, t)$.

At any rate, using the approximations of $u$ and $u_{t}$ in~\eqref{eq:smooth-approximation}, and following the proof of~\cite[Chapter 5.9.2, Theorem 3]{Ev}, we see that there exists $u^* \in C^0([0, T]; L^2(\bB;\RR^3))$, necessarily unique, such that
\begin{equation}\label{eq:embedding-into-C0-L2}
u^*(t) = u(\cdot, t)\text{ in }L^2(\bB; \RR^3),\quad\text{for a.e. }t \in [0, T].
\end{equation}
Moreover, the function $t \mapsto \int_{\bB}|u^*(t)|^2\, dx$ is absolutely continuous, and there holds for all $s, t \in [0, T]$ that
\begin{equation}\label{eq:slice-L2-norm-AC}
\int_{\bB}|u^*(t)|^2\, dx = \int_{\bB}|u^*(s)|^2\, dx + 2\int_{s}^{t}\int_{\bB} u_{t} \cdot u\, dx dt.
\end{equation}
Averaging with respect to $s$ over $[0, T]$ gives the following standard estimate:
\begin{equation}\label{eq:C0-L2-estimate}
\sup_{0 \leq t \leq T} \|u^*(t)\|_{2; \bB}^2 \leq T^{-1} \|u\|_{2; \bB\times [0, T]}^2 + 2\|u_{t}\|_{2; \bB \times [0, T]} \cdot \|u\|_{2; \bB \times [0, T]}.
\end{equation}

For use in the proof of Lemma~\ref{lemm:weak-solution-regular}, we mention that~\eqref{eq:smooth-approximation} can also be used to justify the following estimate. Given $u \in W^{1, 2}(\bB \times (0, T); \RR^3)$ and $[t_1, t_2] \subset (0, T)$, for all sufficiently small positive $h$, we have in the notation of~\eqref{eq:diff-quotient-in-t} that
\begin{equation}\label{eq:diff-quotient-convergence}
\begin{split}
&\int_{\bB \times [t_1, t_2]}|v_{h} - u_{t}|^2\, dxdt\leq  \frac{1}{h}\int_{-h}^{0} \Big(\int_{\bB \times [t_1, t_2]} |u_{t}(x, t + s) - u_{t}(x, t)|^2 \, dx dt\Big)\, ds,
\end{split}
\end{equation}
where the right-hand side converges to zero as $h \to 0$. 
\vskip 2mm
\begin{defi}\label{defi:weak-solution}
Given $0 < T < \infty$ and $u_0 \in W^{1, 2}(\bB; \RR^3)$, by a \emph{weak solution} to~\eqref{eq:H-flow} on $\bB \times [0, T]$, we mean a map $u \in W^{1, 2}(\bB \times (0, T); \RR^3)$ such that the following hold.
\vskip 1mm
\begin{enumerate}
\item[(i)] $u(\cdot, t) - u_0 \in W^{1, 2}_0(\bB; \RR^3)$ for almost every $t \in (0, T)$.
\vskip 1mm
\item[(ii)] $u^*(0) = u_0$ almost everywhere on $\bB$, where $u^*$ is the element of $C^0([0, T]; L^2(\bB; \RR^3))$ such that~\eqref{eq:embedding-into-C0-L2} holds.
\vskip 1mm
\item[(iii)] For all $\zeta \in C^{\infty}_{c}(\bB \times (0, T); \RR^3)$, we have
\begin{equation}\label{eq:H-flow-weak}
\int_{0}^{T}\int_{\bB} u_{t} \cdot \zeta + \bangle{\nabla u, \nabla \zeta}\, dxdt = -2\int_{0}^{T}\int_{\bB} H(u)\, u_{x^1} \times u_{x^2} \cdot \zeta\, dxdt.
\end{equation}
\end{enumerate}
\end{defi}
\begin{rmk}\label{rmk:weak-solution}
It is a standard fact that, for a.e. $t \in (0, T)$, the slices $u(\cdot, t)$, $\nabla u(\cdot, t)$, and $u_t(\cdot, t)$ all belong to $L^2(\bB)$, and $\nabla u(\cdot, t)$ coincides with the distributional derivative $\nabla (u(\cdot, t))$. More importantly, as pointed out in~\cite[page 5266]{LuWang2012}, item (iii) has the following useful consequence. By taking $\zeta$ in~\eqref{eq:H-flow-weak} to have the form 
\[
\zeta(x, t) = \varphi(x)\eta(t),
\]
where $\eta \in C^{\infty}_{c}((0, T); \RR)$ is arbitrary, and $\varphi$ varies over a countable subset of $C^{\infty}_{c}(\bB; \RR^3)$ whose span is dense in $C^{1}_{c}(\bB; \RR^3)$, we see that, for a.e. $t \in (0, T)$, there holds
\begin{equation}\label{eq:weak-solution-slices}
\int_{\bB} u_t(\cdot,t) \cdot \zeta + \bangle{\nabla u(\cdot, t), \nabla \zeta}\, dx = -2\int_{\bB}H(u(\cdot, t))\, u_{x^1}(\cdot, t) \times u_{x^2}(\cdot, t) \cdot \zeta \, dx,
\end{equation}
for all $\zeta \in C^1_{c}(\bB; \RR^3)$. At such a $t$, since every map in $W^{1,2}_0 \cap L^{\infty}(\bB;\RR^3)$ can be realized as the strong $W^{1, 2}$ and a.e. pointwise limit of a sequence in $C^1_c(\bB;\RR^3)$ that is also bounded in $L^{\infty}(\bB; \RR^3)$ (see the proof of~\cite[Chapter 5.5, Theorem 2]{Ev}), the dominated convergence theorem implies that~\eqref{eq:weak-solution-slices} holds more generally for all $\zeta \in W^{1, 2}_0 \cap L^{\infty}(\bB;\RR^3)$.
\end{rmk}

The next lemma, which addresses the regularity of weak solutions under a small-energy condition, is essentially contained as a step in the work of Wang \cite{Wang1999}. Nonetheless, for the convenience of the reader we include a proof, with ideas taken from~\cite{Lin2013},~\cite{Laurain-Riviere2014}, and~\cite{LuWang2012}.
\begin{lemm}\label{lemm:weak-solution-regular}
Suppose $H \in C^1(\RR^3; \RR)$, and that
\[
\|H\|_{\infty; \RR^3} \leq \Lambda_0, \quad \|\nabla H\|_{\infty; \RR^3} \leq \Lambda_1.
\]
Let $u \in W^{1, 2}(\bB \times (0, T); \RR^3)$ be a weak solution to the differential equation in~\eqref{eq:H-flow}, in the sense that item (iii) of Definition~\ref{defi:weak-solution} holds.
There exists $\ep_{\reg} \in (0, 1)$, depending only on $\Lambda_0$ and $\Lambda_1$, such that if 
\begin{equation}\label{eq:H-flow-regularity-smallness}
D(u(\cdot, t)) \leq \ep_{\reg},\quad \text{for a.e. }t \in (0, T),
\end{equation}
then $u \in C^{2 + \mu, 1 + \frac{\mu}{2}}_{\loc}(\bB \times (0, T); \RR^3)$, for all $\mu \in (0, 1)$.
\end{lemm}
\begin{proof}
We largely follow the argument in~\cite[Theorem 3.7]{Lin2013},~\cite[Theorem 12]{Laurain-Riviere2014}, and~\cite[Lemma 2.2]{LuWang2012}. 
\vskip 2mm
\noindent\textbf{Step 1: $u(\cdot, t) \in W^{2, 2}_{\loc}(\bB; \RR^3)$ for a.e. $t \in (0, T)$.} 
\vskip 2mm
Fix one of the almost every $t_0 \in (0, T)$ such that 
\[
u(\cdot, t_0),\ \nabla u(\cdot, t_0),\ u_{t}(\cdot, t_0) \in L^2(\bB), 
\]
that $\nabla u(\cdot, t_0)$ coincides with the distributional derivative $\nabla(u(\cdot, t_0))$, that $D(u(\cdot, t_0)) \leq \ep_{\reg}$, and that~\eqref{eq:weak-solution-slices} holds for all $\zeta \in W^{1, 2}_0 \cap L^{\infty}(\bB; \RR^3)$. Introduce the abbreviations
\begin{equation}\label{eq:slice-abbreviation}
v: = u(\cdot, t_0), \quad w : = u_{t}(\cdot, t_0),
\end{equation}
and also define $\Omega: \bB \to \mathfrak{so}(3) \otimes \Lambda^{1}\RR^2$ by 
\[
\Omega_{ij} = H(v)\, \vol_{\RR^3}(\be_i, \be_j, *dv),
\]
where $\be_1, \be_2, \be_3$ are the standard basis vectors on $\RR^3$, and $*$ is the Hodge-star operator on $\RR^2$. Then, for all $\zeta \in  W^{1, 2}_0 \cap L^{\infty}(\bB;\RR^3)$, we have by~\eqref{eq:weak-solution-slices} that
\begin{equation}\label{eq:flow-equation-with-Omega}
\int_{\bB} \bangle{\nabla v^{i}, \nabla \zeta^{i}} \, dx = \int_{\bB}\big(\bangle{\Omega_{ij}, dv^{j}} - w^{i}\big)\,\zeta^{i}\, dx. 
\end{equation}
Since $\|\Omega\|_{2} \leq C\Lambda_0\cdot \sqrt{\ep_{\reg}}$, it follows from the seminal work of Rivi\`ere~\cite[Theorem I.4]{Riviere2007} that, provided $\ep_{\reg}$ is smaller than a threshold depending only on $\Lambda_0$, there exist $A \in W^{1, 2} \cap L^{\infty}(\bB; \GL(3, \RR))$ and $B \in W^{1, 2}(\bB; \RR^{3 \times 3})$ such that 
\begin{equation}\label{eq:A-B-bounds}
\|\dist(A(\cdot), \text{SO}(3)) \|_{\infty} + \|\nabla A\|_{2} + \|\nabla A^{-1}\|_{2} + \|\nabla B\|_{2} \leq C \Lambda_0 \sqrt{\ep_{\reg}},
\end{equation}
and that, with  $\nabla^{\perp}$ standing for $(-\paop{x^{2}}, \paop{x^1})$ , there holds
\begin{equation}\label{eq:almost-conservation}
\int_{\bB} \bangle{A\nabla v + B\nabla^{\perp}v, \nabla \zeta}\, dx = -\int_{\bB} Aw\cdot \zeta\, dx,\quad\text{for all }\zeta \in C^1_{c}(\bB; \RR^3).
\end{equation}
Given $\bB_{\rho}(x_0) \Subset \bB$, by~\eqref{eq:A-B-bounds} and the Wente inequality~\cite[Theorem 3.1.2]{Helein2002}, there exist $a_1, b \in W^{1, 2}_0(\bB_{\rho}(x_0); \RR^3) \cap C^0(\overline{\bB_{\rho}(x_0)}; \RR^3)$ solving respectively
\begin{equation}\label{eq:a1-b-PDE}
\Delta a_1 =  -\bangle{\nabla B, \nabla^{\perp}v},\quad \Delta b = -\bangle{\nabla A, \nabla^{\perp}v},
\end{equation}
and satisfying the estimates
\begin{equation}\label{eq:a1-b-Wente-estimate}
\|\nabla a_1\|_{2; \bB_{\rho}(x_0)} + \|\nabla b\|_{2; \bB_{\rho}(x_0)} \leq C\Lambda_0\sqrt{\ep_{\reg}} \cdot \|\nabla v\|_{2; \bB_{\rho}(x_0)}.
\end{equation}
Also, with $a_2 \in W^{1, 2}_0(\bB_{\rho}(x_0); \RR^3)$ being the unique weak solution to
\begin{equation}\label{eq:a2-PDE}
\Delta a_2 = Aw,
\end{equation}
the Poincar\'e inequality together with the $L^{\infty}$-estimate in~\eqref{eq:A-B-bounds} gives
\begin{equation}\label{eq:a2-standard-estimate}
\|\nabla a_2\|_{2; \bB_{\rho}(x_0)} \leq C \rho \|w\|_{2; \bB_{\rho}(x_0)}.
\end{equation}
Letting $a : = a_1 + a_2$, by a direct computation using~\eqref{eq:almost-conservation} and the equations satisfied by $a_1, a_2$, we have for all $\zeta \in C^1_{c}(\bB_{\rho}(x_0); \RR^3)$ that
\begin{equation}\label{eq:a-standard-PDE}
\int_{\bB_{\rho}(x_0)} \bangle{\nabla a, \nabla \zeta}\, dx = \int_{\bB_{\rho}(x_0)} \bangle{A \nabla v, \nabla \zeta}\, dx.
\end{equation}
Since $A \in L^{\infty}$, by approximation the above holds for all $\zeta \in W^{1, 2}_0(\bB_{\rho}(x_0); \RR^3)$, and in particular we can take $\zeta = a$. Upon using~\eqref{eq:A-B-bounds} to bound $\|A\|_{\infty}$, we obtain
\begin{equation}\label{eq:a-standard-estimate}
\|\nabla a\|_{2; \bB_{\rho}(x_0)} \leq C \| \nabla v \|_{2; \bB_{\rho}(x_0)}.
\end{equation}
A similar argument using the second equation in~\eqref{eq:a1-b-PDE} shows that
\begin{equation}\label{eq:b-standard-PDE}
\int_{\bB_{\rho}(x_0)}\bangle{\nabla^{\perp} b, \nabla^{\perp}\phi}\, dx = \int_{\bB_{\rho}(x_0)}\bangle{A\nabla v, \nabla^{\perp}\phi}\, dx,
\end{equation}
for all $\phi \in W^{1, 2}_0(\bB_{\rho}(x_0); \RR^3)$, and thus 
\begin{equation}\label{eq:b-standard-estimate}
\|\nabla b\|_{2; \bB_{\rho}(x_0)} \leq C \| \nabla v \|_{2; \bB_{\rho}(x_0)}.
\end{equation}
To continue, let 
\[
h: = A\nabla v - \nabla a - \nabla^{\perp}b,
\]
which is a smooth harmonic $1$-form on $\bB_{\rho}(x_0)$ thanks to~\eqref{eq:a-standard-PDE},~\eqref{eq:b-standard-PDE}, and elliptic regularity. In view of~\eqref{eq:a-standard-estimate},~\eqref{eq:b-standard-estimate}, and the subharmonicity of $|h|^2$, we get for all $\sigma \in (0, \rho)$ that
\[
\int_{\bB_{\sigma}(x_0)} |h|^2 \leq C\Big( \frac{\sigma}{\rho} \Big)^2\int_{\bB_{\rho}(x_0)} |\nabla v|^2\, dx.
\]
Combining this with~\eqref{eq:a1-b-Wente-estimate} and~\eqref{eq:a2-standard-estimate}, and using~\eqref{eq:A-B-bounds} to bound $\|A^{-1}\|_{\infty}$, we arrive at 
\[
\int_{\bB_{\sigma}(x_0)}|\nabla v|^2 \, dx \leq C\Big[ \Big( \frac{\sigma}{\rho} \Big)^2 + \Lambda_0^2\cdot\ep_{\reg} \Big]\int_{\bB_{\rho}(x_0)}|\nabla v|^2 \, dx + C\rho^2\int_{\bB_{\rho}(x_0)}|w|^2\, dx.
\]
With $\lambda \in (0, 1)$ being a universal constant to be determined later, we deduce from~\cite[Lemma 3.4]{Han-Lin} that, provided $\ep_{\reg}$ is sufficiently small depending only on $\Lambda_0$ and $\lambda$, there holds for all $\tau \in (\frac{1}{2}, 1)$, $x_0 \in \bB_{\tau}$, and $\sigma \in (0, 1 - \tau)$ that
\begin{equation}\label{eq:u-Morrey}
\int_{\bB_{\sigma}(x_0)} |\nabla v|^2\, dx \leq C_{\lambda, \tau}\cdot \sigma^{2\lambda}\Big[\int_{\bB}|\nabla v|^2\, dx + \int_{\bB}|w|^2\, dx \Big].
\end{equation}

Next, for any $\zeta \in C^{\infty}_{c}(\bB; [0, 1])$ and $\xi \in \RR^3$, note that 
\begin{equation}\label{eq:H-flow-with-cutoff-zeta}
\Delta(\underbrace{\zeta (v - \xi)}_{U}) = \underbrace{\zeta \cdot \big( 2H(v)v_{x^1} \times v_{x^2} + w \big) + 2\bangle{\nabla\zeta, \nabla v} +  (v - \xi)\Delta\zeta}_{F},
\end{equation}
in the sense of distributions on $\RR^2$. Thus, letting $\Gamma:\RR^2\setminus \{0\} \to \RR$ denote the fundamental solution for $\Delta$ on $\RR^2$, we have for all $\varphi \in C^{\infty}_{c}(\RR^2)$ that
\[
\begin{split}
-\int_{\RR^2} (\partial_{i}\Gamma * F) \Delta \varphi =\ & \int_{\RR^2} F \partial_{i}\varphi\\
=\ & \int_{\RR^2} U \Delta \partial_{i}\varphi = -\int_{\RR^2}(\zeta \partial_{i} v + (v - \xi)\partial_{i}\zeta) \Delta \varphi.
\end{split}
\]
Thus the difference $\zeta \partial_{i} v + (v - \xi)\partial_{i}\zeta - \partial_{i}\Gamma * F$, which tends to zero at infinity, is harmonic on $\RR^2$, and hence must vanish identically. Fixing $\tau \in (\frac{1}{2}, 1)$ and choosing $\zeta$ to equal $1$ on $\bB_{\frac{1+\tau}{2}}$, while also taking $\xi = \fint_{\bB}v(x)\, dx$, we infer that, for a.e. $x \in \bB_{\tau}$, 
\begin{equation}\label{eq:grad-v-riesz}
\begin{split}
|\nabla v(x)| \leq \frac{1}{2\pi}\int_{\bB}|x - y|^{-1} \big( \Lambda_0 |\nabla v(y)|^2 + |w(y)| \big) \, dy+ C_{\tau}\|\nabla v\|_{2}.
\end{split}
\end{equation}
To estimate the integral on the right-hand side, we let
\[
g(y) =\left\{ 
\begin{array}{ll}
\Lambda_0|\nabla v(y)|^2 + |w(y)|, &\text{ if }y \in \bB,\\
0, & \text{ if }y \not\in \bB,
\end{array}
\right.
\]
and also define
\[
(N_{\lambda}g)(x) = \sup_{\sigma > 0} \frac{1}{\sigma^{\lambda}}\int_{\bB_{\sigma}(x)}g(y)\, dy, \quad\quad I_1 g(x) = \int_{\RR^2} |x - y|^{-1}g(y)\, dy.
\]
It is a standard fact (see the proof of~\cite[Proposition 3.1]{Adams1975}) that 
\begin{equation}\label{eq:maximal-bound-riesz}
I_1 g(x) \leq C_{\lambda}\big(N_{\lambda}g(x)  \big)^{\frac{1}{2 - \lambda}} \cdot \big( Mg(x) \big)^{\frac{1 - \lambda}{2 - \lambda}},
\end{equation}
where $Mg$ denotes the maximal function of $g$. Now, by~\eqref{eq:H-flow-regularity-smallness} and H\"older's inequality, we have
\[
\|g\|_{1;\RR^2} \leq C\Lambda_0 \sqrt{\ep_{\reg}} \|\nabla v\|_{2} + C\|w\|_{2}.
\]
On the other hand, taking also~\eqref{eq:u-Morrey} into account, we get for all $x \in \bB_{\tau}$ that
\begin{equation}\label{eq:u-Morrey-for-g}
(N_{\lambda}g)(x) \leq C_{\lambda,\tau, \Lambda_0}\cdot \big(\|\nabla v\|_{2} +\|w\|_{2} \big),
\end{equation}
which combines with~\eqref{eq:maximal-bound-riesz} and the standard weak $L^1$-estimate for $Mg$ to give
\begin{align*}
\big| \big\{ x \in \bB_{\tau}\ \big| \ I_1g(x) > t \big\} \big| \leq C_{\lambda, \tau, \Lambda_0} \cdot \big( \|\nabla v\|_{2} + \|w\|_{2} \big)^{\frac{2-\lambda}{1 - \lambda}} \cdot t^{-\frac{2-\lambda}{1-\lambda}},
\end{align*}
for all $t > 0$. At this point we choose $\lambda \in (0, 1)$ such that $\frac{2 - \lambda}{1 -\lambda} > 4$. Then the above estimate along with~\eqref{eq:grad-v-riesz} implies that
\begin{equation}\label{eq:a.e.-slice-in-L4}
\|\nabla v\|_{4; \bB_{\tau}} \leq \|I_1 g\|_{4; \bB_{\tau}} + C_{\tau}\|\nabla v\|_{2} \leq  C_{\tau, \Lambda_0} \cdot \big( \|\nabla v\|_{2} + \|w\|_{2} \big).
\end{equation}
As a result $\bangle{\Omega, dv} - w \in L^2(\bB_{\tau}; \RR^3)$, which permits the use in~\eqref{eq:flow-equation-with-Omega} of test functions that belong to $W^{1, 2}_{0}(\bB_{\tau}; \RR^3)$. Elliptic regularity then gives $v \in W_{\loc}^{2, 2}(\bB_{\tau}; \RR^3)$. Since $\tau \in (\frac{1}{2}, 1)$ is arbitrary, we conclude that $v \in W^{2, 2}_{\loc}(\bB; \RR^3)$.
\vskip 1em
\noindent\textbf{Step 2: Integral estimates on $|\nabla^2 u|^2$ and $|\nabla u|^4$.}
\vskip 2mm
By the previous step, for a.e. $t \in (0, T)$, the distributional derivatives $(\nabla u(\cdot, t))_{x^1}$ and $(\nabla u(\cdot, t))_{x^2}$ exist and belong to $L^2_{\loc}(\bB)$. For each such $t$, it is a standard fact~\cite[Theorem 2.1.4]{Ziemer1989} that, for a.e. $x^2 \in (-1,1)$, the function $s\mapsto \nabla u(s, x^2, t)$ is absolutely continuous up to redefinition on a measure-zero set, and has derivative given by $(\nabla u(\cdot, t))_{x^1}(s, x^2)$ for almost every $s$. A similar remark applies to restrictions of the form $s \mapsto \nabla u(x^1, s, t)$. Thus, by considering difference quotients, we infer that
\[
(x, t) \mapsto \nabla^2(u(\cdot, t))(x)
\]
is a measurable function on $\bB \times (0, T)$. Again by~\cite[Theorem 2.1.4]{Ziemer1989}, this function would coincide with the second-order distributional derivative of $u$ on $\bB \times (0, T)$ provided it lies in $L^2$ locally.

To establish this last integrability condition, we again fix $t_0$ having the properties listed at the start of Step 1, and adopt the abbreviations~\eqref{eq:slice-abbreviation}. In particular $\nabla v = \nabla u(\cdot, t_0)$. Given a disk $\bB_{\rho}(x_0) \Subset \bB$ and $\sigma \in (0, \rho)$, we let $\xi = \fint_{\bB_{\rho}(x_0)}v(x)\, dx$, and also take $\zeta$ to be a cut-off function such that
\begin{equation}\label{eq:zeta-for-step2}
\zeta = 1 \text{ on }\bB_{\sigma}(x_0),\quad \supp(\zeta) \subset \bB_{\rho}(x_0), \quad |\nabla^k\zeta| \leq C_k(\rho - \sigma)^{-k}.
\end{equation}
Then, by~\eqref{eq:H-flow-with-cutoff-zeta}, now an equality between elements of $L^2(\bB)$, along with standard $W^{2, 2}$-estimates, we get
\begin{equation}\label{eq:weak-solution-W22-by-W14}
\int_{\bB}|\nabla^2(\zeta (v - \xi))|^2  \leq\ C\int_{\bB} \zeta^2|w|^2  + |\nabla \zeta|^2|\nabla v|^2 + |\nabla^2\zeta|^2 |v - \xi|^2  + C\Lambda_0^2\int_{\bB} |\nabla v|^4 \zeta^2.
\end{equation}
To estimate the second integral on the right-hand side, we apply to $|\nabla v|^2\zeta$ the Sobolev inequality associated with the embedding $W^{1, 1}_0(\bB) \hookrightarrow L^2(\bB)$. After routine applications of H\"older's inequality, and recalling also~\eqref{eq:H-flow-regularity-smallness}, we get
\begin{equation}\label{eq:weak-solution-Sobolev}
\int_{\bB} \zeta^2 |\nabla v|^4\, dx \leq C\ep_{\reg}\int_{\bB} |\nabla \zeta|^2|\nabla v|^2 + \zeta^2 |\nabla^2 v|^2\, dx.
\end{equation}
Substituting this back into~\eqref{eq:weak-solution-W22-by-W14}, and noting the pointwise inequality
\[
|\zeta \nabla^2 v| \leq |\nabla^2(\zeta (v - \xi))| + 2 |\nabla \zeta||\nabla v| + |v - \xi||\nabla^2 \zeta|,
\]
we see that, 
\begin{align*}
\int_{\bB} \zeta^2 |\nabla^2 v|^2 \leq\ & C\int_{\bB}\zeta^2 |w|^2 + |\nabla \zeta|^2|\nabla v|^2 + |\nabla^2\zeta|^2|v - \xi|^2  \\
&+ C\Lambda_0^2 \cdot \ep_{\reg} \int_{\bB}|\nabla\zeta|^2 |\nabla v|^2 + \zeta^2 |\nabla^2 v|^2.
\end{align*}
Decreasing $\ep_{\reg}$ depending only on $\Lambda_0$ if needed, recalling our choice of $\zeta$ and $\xi$, and applying Poincar\'e's inequality to the integral of $|v - \xi|^2$, we obtain
\begin{equation}\label{eq:weak-solution-W22-slice}
\int_{\bB}\zeta^2 |\nabla^2(u(\cdot, t_0))|^2\, dx \leq C\int_{\bB \times \{t_0\}}\zeta^2 |u_{t}|^2 \, dx + C (\rho - \sigma)^{-4}\rho^2 \int_{\bB_{\rho}(x_0) \times \{t_0\}} |\nabla u|^2 \, dx.
\end{equation}
Combining this with~\eqref{eq:weak-solution-Sobolev} leads to
\begin{equation}\label{eq:weak-solution-W14-slice}
\int_{\bB \times \{t_0\}}\zeta^2 |\nabla u|^4\, dx \leq C\ep_{\reg}\big( \int_{\bB \times \{t_0\}}\zeta^2 |u_{t}|^2\, dx + (\rho - \sigma)^{-4}\rho^2\int_{\bB_{\rho}(x_0) \times \{t_0\}} |\nabla u|^2\, dx\big). 
\end{equation}
This pair of estimates have the following consequences.
\vskip 1mm
\begin{itemize}
\item From~\eqref{eq:weak-solution-W14-slice}, we get 
\begin{equation}\label{eq:regularity-u-in-W14}
\nabla u \in L^4_{\loc}(\bB \times (0, T)).
\end{equation}
Consequently, given any $\zeta \in W^{1, 2}(\bB \times (0, T); \RR^3)$ which vanishes outside of a compact subset of $\bB \times (0, T)$, by mollification we obtain an approximating sequence which justifies its use as a test function in~\eqref{eq:H-flow-weak}.
\vskip 1mm
\item  From~\eqref{eq:weak-solution-W22-slice}, we see that the function $(x,t) \mapsto \nabla^2 (u(\cdot, t))(x)$ lies in $L^2_{\loc}(\bB \times (0, T))$, and hence coincides with the second-order distributional derivative of $u$ on $\bB \times (0, T)$, as remarked at the start of this step. In particular, for a.e. $t_0 \in (0, T)$, we can write the left-hand side of~\eqref{eq:weak-solution-W22-slice} as $\int_{\bB \times \{t_0\}} \zeta^2 |\nabla^2 u|^2\, dx$. Also, having shown that $\nabla^2 u$, along with $u$, $\nabla u$, and $u_t$, lie in $L^2_{\loc}(\bB \times (0, T))$, we get from the Sobolev embedding in~\cite[Lemma II.3.3]{LSU} that
\begin{equation}\label{eq:regularity-u-in-Lp}
u \in L^{p}_{\loc}(\bB \times (0, T))\quad\text{for all }2 \leq p < \infty.
\end{equation}
\end{itemize}

We next want to control the integral of $\zeta^2|u_t|^2$ on the right-hand side of~\eqref{eq:weak-solution-W22-slice} and~\eqref{eq:weak-solution-W14-slice}. To that end, take any $[t_1, t_2] \subset (0, T)$. For sufficiently small $h > 0$, we consider, in the notation of~\eqref{eq:diff-quotient-in-t}, the map
\[
v_{h} : = \frac{u - u^{(-h)}}{h}.
\]
Then, with $\bB_{\sigma}(x_0) \Subset \bB_{\rho}(x_0) \Subset \bB$ as above, and with the same choice of $\zeta$ as in~\eqref{eq:weak-solution-W22-slice} and~\eqref{eq:weak-solution-W14-slice}, we have for all $\phi \in C^{\infty}_{c}((t_1, t_2); [0, 1])$ that 
\[
v_h \zeta^2 \phi \in W^{1, 2}(\bB \times (0, T); \RR^3).
\]
Since $v_h \zeta^2 \phi$ also vanishes outside of $\bB_{\rho}(x_0) \times (t_1, t_2)$, the observation made after~\eqref{eq:regularity-u-in-W14} shows that it is an admissible test function in~\eqref{eq:H-flow-weak}. Recalling the lower bound~\eqref{eq:H-flow-tested-lhs-observation}, and letting $\phi$ increase to the characteristic function $\chi_{(t_1, t_2)}$, we obtain with the help of the dominated convergence theorem that
\begin{align*}
&\int_{t_1}^{t_2}\int_{\bB} \zeta^2 u_{t}\cdot v_{h}\, dxdt + \frac{1}{2h}\int_{t_1}^{t_2}\int_{\bB}\zeta^2 (|\nabla u|^2 - |\nabla u^{(-h)}|^2)\, dxdt\\
\leq\ & \int_{t_1}^{t_2}\int_{\bB} 2|\nabla\zeta| \zeta |\nabla u| |v_{h}|\, dxdt + C\Lambda_0\int_{t_1}^{t_2}\int_{\bB} \zeta^2 |\nabla u|^2 |v_{h}|\, dxdt.
\end{align*}
Changing variables in the second integral on the left-hand side, and applying Young's inequality to each of the two integrals on the right-hand side, we get
\begin{align*}
&\int_{t_1}^{t_2}\int_{\bB} \zeta^2 u_{t}\cdot v_{h}\, dxdt + \frac{1}{2h}\int_{t_2-h}^{t_2}\int_{\bB}\zeta^2 |\nabla u|^2\, dxdt\\
\leq\ & \frac{1}{2h}\int_{t_1-h}^{t_1}\int_{\bB}\zeta^2 |\nabla u|^2\, dxdt + \frac{1}{4}\int_{t_1}^{t_2}\int_{\bB} \zeta^2 |v_{h}|^2\, dxdt \\
&+ C\int_{t_1}^{t_2}\int_{\bB} |\nabla\zeta|^2 |\nabla u|^2 + C\Lambda_0^2\int_{t_1}^{t_2}\int_{\bB} \zeta^2 |\nabla u|^4\, dxdt.
\end{align*}
Dropping the second term on the left-hand side, and also using~\eqref{eq:diff-quotient-convergence} and the assumption~\eqref{eq:H-flow-regularity-smallness}, we obtain
\begin{equation}\label{eq:speed-by-W14}
\begin{split}
\int_{t_1}^{t_2}\int_{\bB}\zeta^2|u_{t}|^2\leq\ & C\ep_{\reg} + C\int_{t_1}^{t_2}\int_{\bB}|\nabla\zeta|^2 |\nabla u|^2+ C\Lambda_0^2 \int_{t_1}^{t_2}\int_{\bB} \zeta^2 |\nabla u|^4 .
\end{split}
\end{equation}
Integrating~\eqref{eq:weak-solution-W14-slice} with respect to $t_0$ and substituting the result into the above, we get upon further decreasing $\ep_{\reg}$ depending on $\Lambda_0$ that
\begin{equation}\label{eq:weak-solution-speed-L2}
\int_{t_{1}}^{t_{2}}\int_{\bB}\zeta^2|u_{t}|^2\, dxdt \leq C\ep_{\reg} + C(\rho - \sigma)^{-4}\rho^2\int_{t_1}^{t_2}\int_{\bB_{\rho}(x_0)} |\nabla u|^2 \, dxdt.
\end{equation}
Combining this with the result of integrating~\eqref{eq:weak-solution-W22-slice} over $[t_1, t_2]$ yields
\begin{equation}\label{eq:weak-solution-W22}
\int_{t_1}^{t_2}\int_{\bB_{\sigma}(x_0)} |\nabla^2 u|^2\, dxdt\leq C\ep_{\reg} + C(\rho - \sigma)^{-4}\rho^2\int_{t_1}^{t_2}\int_{\bB_{\rho}(x_0)} |\nabla u|^2 \, dxdt.
\end{equation}
\vskip 1em
\noindent\textbf{Step 3: Improved estimate on $|u_{t}|^2$ and bootstrapping.}
\vskip 2mm
Given $(x_0, t_0) \in \bB \times (0, T)$ and $r >0$ such that 
\[
\bB_{3r}(x_0) \times (t_0 - 9r^2, t_0 + 9r^2) \subset \bB \times (0, T),
\]
we let $\eta$ be a cut-off function such that
\[
\eta = 1 \text{ on }\bB_{r}(x_0), \quad \supp(\eta) \subset \bB_{2r}(x_0), \quad
|\nabla^k\eta| \leq C_{k} r^{-k}.
\] 
For any $h \in (0, r^2)$, and any smooth function $\phi: \RR \to [0, 1]$ such that $\phi(t) = 0$ whenever $|t - t_0| \geq 4r^2$, we define $w_h: = -\frac{1}{h}((\eta^2 v_{h}\phi)^{(h)} - \eta^2 v_{h}\phi)$ in the notation of~\eqref{eq:diff-quotient-in-t}; that is,
\[
w_{h}(x, t) = -\big(\eta(x)\big)^2 \cdot \frac{v_{h}(x, t+h)\phi(t + h) - v_{h}(x, t) \phi(t)}{h}.
\]
By virtue of the cutting off, both $\eta^2 v_h \phi$ and $(\eta^2 v_h \phi)^{(h)}$ lie in $W^{1, 2}(\bB \times (0, T); \RR^3)$, and hence so does $w_h$. Since $w_h$ also vanishes outside of $\bB_{2r}(x_0) \times (t_0 - 5r^2, t_0 + 5r^2)$, we may use it as a test function in~\eqref{eq:H-flow-weak}. Following a computation similar to the one in the proof of Lemma~\ref{lemm:speed-monotone}, we get
\begin{align*}
&\int_{0}^{T}\int_{\bB} (v_{h})_{t} \cdot v_{h} \eta^2 \phi  + \int_{0}^{T}\int_{\bB} |\nabla v_{h}|^2 \eta^2 \phi - 2\eta |\nabla \eta| \phi |v_{h}| |\nabla v_{h}| \\
\leq\ & 2 \Lambda_1\int_{0}^{T}\int_{\bB}  |v_{h}|^2 |\nabla u|^2 \eta^2 \phi  + 2\Lambda_0\int_{0}^{T}\int_{\bB} |\nabla v_{h}| (|\nabla u| + |\nabla u^{(-h)}|) |v_{h}| \eta^2 \phi,
\end{align*}
the integrals on the right-hand side being finite due to~\eqref{eq:regularity-u-in-W14} and~\eqref{eq:regularity-u-in-Lp}. A couple of applications of Young's inequality leads to
\begin{align}
&\int_{0}^{T}\int_{\bB} (v_{h})_{t} \cdot v_{h} \eta^2 \phi + \frac{3}{4} \int_{0}^{T}\int_{\bB} |\nabla v_{h}|^2 \eta^2 \phi \nonumber\\
\leq\ & C(\Lambda_1 + \Lambda_0^2)\int_{0}^{T}\int_{\bB} (|\nabla u| + |\nabla u^{(-h)}|)^2|v_{h}|^2\eta^2 \phi +C\int_{0}^{T}\int_{\bB} \phi |\nabla\eta|^2 |v_{h}|^2.\label{eq:twice-difference-quotient-in-t}
\end{align}
For later use, we mention that, similar to~\eqref{eq:slice-L2-norm-AC}, for all $\tau_1, \tau_2$ in a subset of $[t_0 - 4r^2, t_0 + 4r^2]$ with full measure, there holds
\begin{equation}\label{eq:vh-L2-norm-AC}
\int_{\bB \times \{\tau_2\}}|v_{h}|^2 \eta^2\, dx = \int_{\bB \times \{\tau_1\}}|v_{h}|^2 \eta^2\, dx + 2\int_{\tau_1}^{\tau_2}\int_{\bB} (\eta v_{h})_{t} \cdot (\eta v_{h})\, dxdt.
\end{equation}
Turning to the first integral on the right-hand side of~\eqref{eq:twice-difference-quotient-in-t}, we have for a.e. $t \in [t_0 - 4r^2, t_0 + 4r^2]$ that 
\[
(|\nabla u(\cdot, t)| + |\nabla u^{(-h)}(\cdot, t)|)\, |v_{h}(\cdot, t)|\,\eta \in W^{1, 1}_{0}(\bB),
\]
so the Sobolev embedding $W^{1, 1}_{0}(\bB) \hookrightarrow L^2(\bB)$ gives
\begin{align*}
&\int_{\bB \times \{t\}}(|\nabla u|+ |\nabla u^{(-h)}|)^2  |v_{h}|^2 \eta^2 \, dx\\
\leq\ & C\Big( \int_{\bB \times \{t\}} (|\nabla u| + |\nabla u^{(-h)}|) \cdot (|v_{h}||\nabla \eta| + |\nabla v_{h}||\eta|) +  (|\nabla^2 u| + |\nabla^2 u^{(-h)}|)\cdot |v_{h}|\eta\, dx \Big)^2\\
\leq\ & C\ep_{\reg}\int_{\bB\times \{t\}} |v_{h}|^2 |\nabla \eta|^2 + \eta^2 |\nabla v_{h}|^2 \, dx\\
&+ C\Big( \int_{\bB_{2r}(x_0) \times \{t\}}|\nabla^2 u|^2 + |\nabla^2 u^{(-h)}|^2\, dx \Big) \Big(\int_{\bB \times \{t\}} |v_{h}|^2\eta^2\, dx\Big),
\end{align*}
where for the second inequality we used~\eqref{eq:H-flow-regularity-smallness}. Multiplying by $\phi(t)$ and integrating with respect to $t$ gives
\begin{align}
&\int_{0}^{T}\int_{\bB} \big( |\nabla u| + |\nabla u^{(-h)}| \big)^2 |v_{h}|^2 \eta^2\phi \, dxdt\nonumber\\
\leq\ & C\ep_{\reg}\int_{0}^{T}\int_{\bB} |v_{h}|^2 |\nabla \eta|^2\phi + \eta^2 |\nabla v_{h}|^2\phi\, dxdt\nonumber\\
&+ \Big(\esssup_{|t - t_0| \leq 4r^2}\int_{\bB \times \{t\}}|v_{h}|^2\eta^2\phi \, dx\Big)\Big( \int_{t_0 - 4r^2}^{t_0 + 4r^2}\int_{\bB_{2r}(x_0)}|\nabla^2 u|^2 + |\nabla^2 u^{(-h)}|^2\, dxdt \Big) \nonumber\\
\leq\ &  C\ep_{\reg}\int_{0}^{T}\int_{\bB} |v_{h}|^2 |\nabla \eta|^2\phi + \eta^2 |\nabla v_{h}|^2\phi\, dxdt + C \ep_{\reg}\cdot \esssup_{|t - t_0| \leq 4r^2}\int_{\bB \times \{t\}}|v_{h}|^2\eta^2\phi \, dx, \label{eq:twice-difference-quotient-rhs-bound}
\end{align}
where in getting the last line we used~\eqref{eq:weak-solution-W22} (with $\rho = 3r$ and $\sigma = 2r$) and~\eqref{eq:H-flow-regularity-smallness}. Putting~\eqref{eq:twice-difference-quotient-rhs-bound} back into~\eqref{eq:twice-difference-quotient-in-t}, and requiring that $\ep_{\reg}$ be smaller than a threshold depending only on $\Lambda_0$ and $\Lambda_1$, we obtain
\begin{equation}\label{eq:speed-estimate-ready-to-choose-phi}
\begin{split}
\int_{0}^{T}\int_{\bB} (\eta v_{h})_{t} \cdot (\eta v_{h})  \phi &\leq C\int_{0}^{T}\int_{\bB}|\nabla\eta|^2 |v_{h}|^2\phi + \frac{1}{8} \cdot \esssup_{|t - t_0| \leq 4r^2}\int_{\bB \times \{t\}}|v_{h}|^2\eta^2\phi \, dx.
\end{split}
\end{equation}
Given $\tau_1 \in [t_0 - 4r^2, t_0 - r^2]$ and $\tau_2 \in [\tau_1, t_0 + 4r^2]$ such that~\eqref{eq:vh-L2-norm-AC} holds, we apply~\eqref{eq:speed-estimate-ready-to-choose-phi} to a sequence $(\phi_n)$ in $C^{\infty}_{c}((\tau_1, \tau_2); [0, 1])$ that increases pointwise to the characteristic function $\chi_{(\tau_1, \tau_2)}$. Combining the result with~\eqref{eq:vh-L2-norm-AC}, we get
\begin{align*}
\int_{\bB \times \{\tau_2\}} |v_{h}|^2\eta^2\, dx \leq\ & \int_{\bB \times \{\tau_1\}}|v_{h}|^2\eta^2\, dx + C\int_{\tau_1}^{t_0 + 4r^2} \int_{\bB} |\nabla \eta|^2 |v_{h}|^2\, dxdt\\
&+ \frac{1}{4}\esssup_{t \in [\tau_1, t_0 + 4r^2]}\int_{\bB \times \{t\}} |v_{h}|^2\eta^2\, dx.
\end{align*}
This being true for almost every $\tau_2 \in [\tau_1, t_0 + 4r^2]$, we deduce that
\[
\begin{split} 
\frac{3}{4}\cdot \esssup_{t \in [\tau_1, t_0 + 4r^2]}\int_{\bB \times \{t\}}|v_{h}|^2\eta^2 \, dx \leq \int_{\bB \times \{\tau_1\}}|v_{h}|^2\eta^2 \, dx + C\int_{\tau_1}^{t_0 + 4r^2}\int_{\bB}|\nabla\eta|^2 |v_{h}|^2\, dxdt,
\end{split}
\]
for almost every $\tau_1 \in [t_0 - 4r^2, t_0 - r^2]$. Averaging with respect to $\tau_1$ and recalling our choice of $\eta$, we get
\[
\esssup_{t \in [t_0 - r^2, t_0 + r^2]}\int_{\bB_{r}(x_0) \times \{t\}}|v_{h}|^2 \, dx \leq Cr^{-2}\int_{t_0 - 4r^2}^{t_0 + 4r^2}\int_{\bB_{2r}(x_0)}|v_{h}|^2\, dxdt.
\]
Now let $(h_n)$ be a positive sequence tending to $0$. From the above estimate, along with~\eqref{eq:diff-quotient-convergence}, we obtain a measure-zero set $E$ such that for every $t \in [t_0 - r^2, t_0 + r^2]\setminus E$, we have
\[
\lim_{n \to \infty}\int_{\bB \times \{t\}}|v_{h_n} - u_{t}|^2\, dx = 0,
\]
and that
\[
\int_{\bB_{r}(x_0) \times \{t\}} |v_{h_n}|^2\, dx \leq Cr^{-2}\int_{t_0 - 4r^2}^{t_0 + 4r^2}\int_{\bB_{2r}(x_0)} |v_{h_n}|^2\, dxdt,\quad\text{for all }n.
\]
Passing to the limit as $n \to \infty$ and using~\eqref{eq:diff-quotient-convergence} again, we obtain
\begin{equation}\label{eq:speed-slice}
\int_{\bB_{r}(x_0) \times \{t\}} |u_{t}|^2\, dx \leq C r^{-2} \int_{t_0 - 4r^2}^{t_0 + 4r^2}\int_{\bB_{2r}(x_0)} |u_t|^2 \, dxdt \leq C r^{-2}\ep_{\reg},
\end{equation}
where the second inequality follows from~\eqref{eq:weak-solution-speed-L2}. Returning to~\eqref{eq:weak-solution-W22-slice}, we deduce using~\eqref{eq:speed-slice} and~\eqref{eq:H-flow-regularity-smallness} that
\[
\int_{\bB_{\frac{r}{2}}(x_0) \times \{t\}} |\nabla^2 u|^2\, dx \leq C\int_{\bB_{r}(x_0) \times \{t\}} |u_{t}|^2 + r^{-2}|\nabla u|^2 \, dx \leq Cr^{-2}\ep_{\reg},
\]
for almost every $t \in [t_0 - r^2, t_0 + r^2]$. Sobolev embedding then implies that, for all $q \in [2, \infty)$,
\[
\int_{t_0-r^2}^{t_0 + r^2}\int_{\bB_{\frac{r}{2}}(x_0)}|\nabla u|^{q}\, dxdt < \infty.
\]
Returning to the flow equation~\eqref{eq:H-flow}, we thus have $u_{t} - \Delta u \in L^{p}_{\loc}(\bB \times (0, T))$ for all $p < \infty$. Taking also~\eqref{eq:regularity-u-in-Lp} into account, we can follow the argument indicated in the proof of~\cite[Lemma 3.10]{Struwe1985}, using cut-off functions and~\cite[Theorem IV.9.1]{LSU}, followed by the embedding in the second conclusion of~\cite[Lemma II.3.3]{LSU}, to get
\[
u, \nabla u \in C^{\mu, \frac{\mu}{2}}_{\loc}(\bB \times (0, T)),\quad\text{for all }\mu \in (0, 1).
\]
Again using~\eqref{eq:H-flow}, this time combined with~\cite[Theorem IV.5.2]{LSU}, we conclude that $u \in C^{2 + \mu, 1 + \frac{\mu}{2}}_{\loc}(\bB \times (0, T))$, and the proof is complete.
\end{proof}

Proposition~\ref{prop:stability} below can be seen as a quantitative uniqueness result, and is obtained by following the proof of~\cite[Lemma 3.2]{LuWang2012}, with some modifications which involve computations similar to those appearing in the proof of Theorem~\ref{thm:uniqueness}. 
\begin{prop}\label{prop:stability}
Given $\Lambda_0, \Lambda_1 > 0$, let $\ep_{\reg}$ be as in Lemma~\ref{lemm:weak-solution-regular}. There exists $\ep_{\text{st}} < \ep_{\reg}$, depending only on $\Lambda_0, \Lambda_1$, with the following property. Suppose $H \in C^1(\RR^3; \RR)$ satisfies 
\[
\|H\|_{\infty; \RR^3} \leq \Lambda_0,\quad \|\nabla H\|_{\infty; \RR^3} \leq \Lambda_1,
\]
and let $u, v \in W^{1, 2}(\bB \times (0, T);\RR^3)$ be weak solutions to~\eqref{eq:H-flow} with Cauchy--Dirichlet data $u_0, v_0 \in W^{1, 2}(\bB;\RR^3) \cap C^{0}(\overline{\bB}; \RR^3)$, respectively. Assume further that 
\begin{equation}\label{eq:stability-small-energy-assumption}
D(u(\cdot,t))+ D(v(\cdot, t))\leq \ep_{\text{st}}\quad \text{for a.e. }t\in (0,T).
\end{equation}
Then we have
\begin{equation}\label{eq:stability-estimate-W12}
\begin{split}
&\int_0^{T} \int_{\bB} |\nabla (u-v)|^2 t^{-\frac{1}{2}}\,dxdt  \leq  10\cdot T^{\frac{1}{2}} \cdot\big( \|u_0 - v_0\|_{\infty}^2 + \|\nabla u_0 - \nabla v_0\|_{2}^2 \big),
\end{split}
\end{equation}
and that 
\begin{equation}\label{eq:stability-estimate-C0}
\begin{split}
&\esssup_{0 \leq t \leq T} \int_{\bB \times \{t\}}|u - v|^2\, dx \leq C(1 + T) \cdot \big(\|u_0 - v_0\|_{\infty}^2  + \|\nabla u_0 - \nabla v_0\|_{2}^2 \big),
\end{split}
\end{equation}
where $C$ is a universal constant.
\end{prop}
\begin{proof}
We require first that $\ep_{\text{st}}$ be less than the constant $\ep_{\reg}$ given by Lemma~\ref{lemm:weak-solution-regular}, so that, fixing a particular $q > 4$ and letting $\alpha = 1 - \frac{4}{q}$, we have 
\[
u, v \in C_{\loc}^{2 + \alpha, 1 + \frac{\alpha}{2}}(\bB \times (0, T)).
\]
It follows that the energy bound~\eqref{eq:stability-small-energy-assumption} holds for \emph{all} $t \in (0, T)$, so by Corollary~\ref{coro:grad-estimate-scaled}, provided $\ep_{\text{st}}$ is below a threshold depending only on $\Lambda_0$, we have for all \( (x, t) \in \bB \times (0, T) \) that
\begin{equation}\label{eq:gradient-estimate-for-stability}
|\nabla u(x, t)|^2 + |\nabla v(x, t)|^2 \leq C_{\Lambda_0}\cdot \ep_{\text{st}} \left( t^{-1} + (1 - |x|)^{-2} \right).
\end{equation}
The gradient estimate has the following implication when combined with the Hardy inequality~\eqref{eq:Hardy}. Given  a sequence $(\zeta_n)$ in $C^{1}_{c}(\bB; \RR^3)$ that converges strongly in $W^{1, 2}$ on $\bB$ to some $\zeta \in W^{1, 2}_0(\bB; \RR^3)$, we have by~\eqref{eq:Hardy} and~\eqref{eq:gradient-estimate-for-stability} that
\begin{align*}
\int_{\bB}|\nabla u(\cdot, t)|^2 |\zeta_n - \zeta|^2\, dx \leq\ & C_{\Lambda_0}\cdot\ep_{\text{st}}\int_{\bB}\frac{|\zeta_n - \zeta|^2}{t} + \frac{|\zeta_n - \zeta|^2}{(1 - |x|)^2} \, dx\\
\leq\ & C_{\Lambda_0}\cdot \ep_{\text{st}} \big( t^{-1} \|\zeta_n - \zeta\|_{2;\bB}^2 + \|\nabla \zeta_n - \nabla\zeta\|_{2; \bB}^2\big).
\end{align*}
It is then fairly straightforward to see that, for any of the almost every $t \in (0, T)$ at which the conclusion of Remark~\ref{rmk:weak-solution} holds, we can actually use in~\eqref{eq:weak-solution-slices} test functions that belong to $W^{1, 2}_0(\bB; \RR^3)$.

Now, following~\cite[Lemma 3.2]{LuWang2012}, we define $w := u - v$ and $w_0: = u_0  - v_0$. By the remarks before Definition~\ref{defi:weak-solution}, we see that $t \mapsto \int_{\bB \times \{t\}} |w - w_0|^2\, dx$ is absolutely continuous on $[0, T]$ up to redefinition on a measure-zero set, and that
\[
\frac{d}{dt}\int_{\bB \times \{t\}} |w - w_0|^2\, dx = 2\int_{\bB \times \{t\}} w_{t} \cdot (w - w_0)\, dx,\quad\text{for a.e. }t \in [0, T].
\]
Using also item (ii) in Definition~\ref{defi:weak-solution}, we infer that 
\[
\lim_{t \to 0}\int_{\bB \times \{t\}} |w - w_0|^2 = 0.
\]
As a result, the statements~\cite[page 5274, (3.9) and (3.10)]{LuWang2012} concerning the function $t \mapsto \int_{\bB \times \{t\}} |w - w_0|^2\, dx$ remain valid in our context, since the argument involves only the properties we just discussed, and makes no use of the differential equation satisfied by $u$ and $v$. Separately, we note for later use that, for a.e. $t \in [0, T]$,
\begin{equation}\label{eq:w-slice}
\int_{\bB \times \{t\}} |w|^2 \, dx \leq 2\int_{\bB \times \{t\}} |w - w_0|^2\, dx + 2\pi \cdot\|w_0\|_{\infty}^2.
\end{equation}

Moving towards the actual proof of the asserted estimates, from item (i) in Definition~\ref{defi:weak-solution}, we see that $w(\cdot, t) - w_0 \in W^{1, 2}_0(\bB; \RR^3)$ for a.e. $t \in (0, T)$. Using this as a test function in~\eqref{eq:weak-solution-slices}, permitted by the observation made after~\eqref{eq:gradient-estimate-for-stability}, we obtain after some rearrangements that, for a.e. $t \in (0, T)$,
\begin{align}\label{eq:stability-estimate-break-down}
\int_{\bB \times \{t\}} |\nabla w|^2 dx =\ & -\int_{\bB \times \{t\}} w_t\cdot( w - w_0)\, dx  \nonumber\\
&-2 \int_{\bB \times \{t\}}\big( H(u)\,u_{x^1} \times u_{x^2} - H(v)\,v_{x^1} \times v_{x^2}\big)\cdot (w - w_0)\,  dx  \nonumber\\
& + \int_{\bB \times \{t\}}\bangle{\nabla w, \nabla w_0}\, dx  \nonumber\\
=:\ & I_{1}(t) + I_{2}(t) + I_{3}(t).
\end{align}
Given $\delta \in (0, \frac{1}{4}\min\{1, T\})$, as in the proof of \cite[Lemma 3.2]{LuWang2012}, we integrate by parts to get
\begin{equation}\label{eq:estimate-I}
\begin{split}
\int_{\delta}^{T} t^{-\frac{1}{2}} I_1(t)\, dt =\ & \delta^{-\frac{1}{2}} \int_{\bB \times \{\delta\}} \frac{|w - w_0|^2}{2} dx -T^{-\frac{1}{2}} \int_{\bB \times \{T\}} \frac{|w - w_0|^2}{2} dx\\
&- \frac{1}{4}\int_{\delta}^{T}\int_{\bB}|w - w_0|^2 t^{-\frac{3}{2}} dx dt.
\end{split}
\end{equation}
Coming to $I_{2}(t)$, we start with the pointwise bound
\begin{align*}
& 2 \big| \big(H(u)\,u_{x^1} \times u_{x^2} - H(v)\,v_{x^1} \times v_{x^2}\big) \cdot (w - w_0)\big|\\
\leq\ & 2|w - w_0|\cdot \big(\Lambda_1 |w||\nabla u|^2+ \Lambda_0 |\nabla w|(|\nabla u| + |\nabla v|)\big)\\
\leq\ & 2\Lambda_1 |w - w_0|^2 |\nabla u|^2 + 2\Lambda_1|w - w_0| |w_0| |\nabla u|^2\\
& + 2\Lambda_0 |\nabla w||w - w_0|(|\nabla u| + |\nabla v|).
\end{align*}
By Young's inequality and the gradient estimate~\eqref{eq:gradient-estimate-for-stability}, we get
\begin{align}
&2 \big|\big(H(u)\,u_{x^1} \times u_{x^2} - H(v)\,v_{x^1} \times v_{x^2}\big) \cdot (w - w_0)\big|\cdot t^{-\frac{1}{2}}\nonumber\\
\leq\ & \frac{1}{8}|\nabla w|^2 t^{-\frac{1}{2}} + \Lambda_1 |w_0|^2 |\nabla u|^2 t^{-\frac{1}{2}} + (8\Lambda_0^2 + 3\Lambda_1)|w - w_0|^2 (|\nabla u| + |\nabla v|)^2 t^{-\frac{1}{2}} \label{eq:stability-II-integrand}\\
\leq\ & \frac{1}{8}|\nabla w|^2t^{-\frac{1}{2}} + \Lambda_1 |w_0|^2 |\nabla u|^2t^{-\frac{1}{2}} + C_{\Lambda_0, \Lambda_1}\cdot\ep_{\text{st}}\big( |w - w_0|^2 t^{-\frac{3}{2}} + \frac{|w - w_0|^2}{(1 - |x|)^2}t^{-\frac{1}{2}} \big).\nonumber
\end{align}
For the second term on the last line, we notice from~\eqref{eq:stability-small-energy-assumption} that
\[
\int_{\delta}^{T}\int_{\bB} |w_0|^2 |\nabla u|^2 t^{-\frac{1}{2}} \, dxdt \leq \|w_0\|_{\infty}^2 \cdot 2\ep_{\text{st}} \cdot 2T^{\frac{1}{2}}.
\]
On the other hand, since $w(\cdot, t) - w_0 \in W^{1, 2}_0(\bB; \RR^3)$ for almost every $t \in (0, T)$, we have by~\eqref{eq:Hardy} that
\begin{align*}
\int_{\delta}^{T}\int_{\bB}\frac{|w - w_0|^2}{(1 - |x|)^2}\, t^{-\frac{1}{2}} \, dxdt \leq \ & C\int_{\delta}^{T} \int_{\bB} |\nabla w - \nabla w_0|^2\,t^{-\frac{1}{2}}\, dxdt\\
\leq\ & C\int_{\delta}^{T} \int_{\bB} |\nabla w|^2\,t^{-\frac{1}{2}}\, dxdt + C T^{\frac{1}{2}}\|\nabla w_0\|_{2}^2.
\end{align*}
Combining the two previous estimates with~\eqref{eq:stability-II-integrand} yields
\begin{align}
\int_{\delta}^{T} t^{-\frac{1}{2}} |I_{2}(t)|\, dt \leq\ & \big( \frac{1}{8} + C \ep_{\text{st}} \big)  \int_{\delta}^{T}\int_{\bB} |\nabla w|^2 t^{-\frac{1}{2}}\, dxdt + C \ep_{\text{st}} \cdot T^{\frac{1}{2}} \|\nabla w_0\|_{2}^2 \nonumber\\
& + C\ep_{\text{st}}\int_{\delta}^{T}\int_{\bB} |w - w_0|^2 \, t^{-\frac{3}{2}}\, dxdt + C \ep_{\text{st}}\cdot T^{\frac{1}{2}} \|w_0\|_{\infty}^2, \label{eq:estimate-II}
\end{align}
where the constants $C$ depend only on $\Lambda_0$ and $\Lambda_1$. For $I_3(t)$, we simply use Young's inequality to get 
\begin{equation}\label{eq:estimate-III}
\begin{split}
\int_{\delta}^{T} t^{-\frac{1}{2}}|I_3(t)|\, dt \leq\ & \int_{\delta}^{T}\int_{\bB} t^{-\frac{1}{2}} |\nabla w||\nabla w_0|\, dxdt\\
\leq\ & \frac{1}{8}\int_{\delta}^{T}\int_{\bB}t^{-\frac{1}{2}} |\nabla w|^2\, dxdt + 4T^{\frac{1}{2}} \|\nabla w_0\|_{2}^2.
\end{split}
\end{equation}
By~\eqref{eq:estimate-I},~\eqref{eq:estimate-II}, and~\eqref{eq:estimate-III}, upon multiplying~\eqref{eq:stability-estimate-break-down} by $t^{-\frac{1}{2}}$ and integrating over $[\delta, T]$, we obtain
\[
\begin{split}
&\big( \frac{3}{4} - C\cdot \ep_{\text{st}} \big)\int_{\delta}^{T}\int_{\bB} |\nabla w|^2 t^{-\frac{1}{2}} dx dt + \frac{1}{2\sqrt{T}}\int_{\bB \times \{T\}}|w - w_0|^2 dx\\
\leq\ &  \frac{1}{2\sqrt{\delta}}\int_{\bB \times \{\delta\}}|w - w_0|^2 dx - \big( \frac{1}{4} - C\ep_{\text{st}}\big) \int_{\delta}^{T}\int_{\bB} |w - w_0|^2 t^{-\frac{3}{2}} dx dt\\
& + (4 + C\ep_{\text{st}}) \sqrt{T} \cdot \big(\|w_0\|_{\infty}^2 + \|\nabla w_0\|_{2}^2 \big),
\end{split}
\]
where the constants $C$ again all depend only on $\Lambda_0$ and $\Lambda_1$. Thus, provided that $\ep_{\text{st}}$ is sufficiently small, upon letting $\delta \to 0$ and using~\cite[page 5274, (3.10)]{LuWang2012}, we obtain
\[
\begin{split}
&T^{\frac{1}{2}}\int_{0}^{T}\int_{\bB} |\nabla w|^2 t^{-\frac{1}{2}} dx dt + \int_{\bB \times \{T\}}|w - w_0|^2 dx \leq  10 T \cdot \big(\|w_0\|_{\infty}^2 + \|\nabla w_0\|_{2}^2 \big),
\end{split}
\]
which gives~\eqref{eq:stability-estimate-W12}. Applying the above estimate to sub-intervals $[0, \tau] \subset [0, T]$, and recalling~\eqref{eq:w-slice}, we get~\eqref{eq:stability-estimate-C0}.
\end{proof}

We finally come to the proof of Theorem~\ref{thm:uniformity}, restated as the following result.
\begin{thm}[Theorem~\ref{thm:uniformity} restated]
\label{thm:uniformity-restated}
Given $\Lambda_0, \Lambda_1> 0$, let $\ep_{3}$ and $\ep_{\text{st}}$ be the thresholds given by Propositions~\ref{prop:flow-convexity} and~\ref{prop:stability}, respectively, and let $H(\cdot)$ be as in the latter result. Suppose $u_0 \in C^0(\overline{\bB}; \RR^3) \cap W^{1, 2}(\bB; \RR^{3})$ is such that
\begin{equation}\label{eq:uniformity-small-energy}
D(u_0) < \min\{\ep_3, \frac{\ep_{\text{st}}}{4}\}.
\end{equation}
Then there exists a unique weak solution in $\cap_{T>0} W^{1,2} (\bB \times (0,T); \RR^3)$ to the Cauchy-Dirichlet problem~\eqref{eq:H-flow} subject to the requirement that
\begin{equation}\label{eq:small-energy-main-thm}
D(u(\cdot, t))\leq \min\{2\ep_3, \frac{\ep_{\text{st}}}{2}\},\quad\text{for a.e. }t > 0.
\end{equation}
This solution in fact belongs to $C^{2 + \mu, 1 + \frac{\mu}{2}}_{\loc}(\bB \times (0, \infty) ; \RR^3)$ for all $\mu \in (0, 1)$. Moreover, the following additional properties hold.
\vskip 1mm
\begin{enumerate}
\item[(a)] The map $t \mapsto u(\cdot, t)$ is continuous from $(0, \infty)$ to $W^{1, 2}(\bB; \RR^3)$, and we have $u(\cdot, t) - u_0 \in W^{1, 2}_0(\bB; \RR^3)$ for \emph{all} $t > 0$.
\vskip 1mm
\item[(b)] $u$ extends continuously to $\overline{\bB} \times (0, \infty)$.
\vskip 1mm
\item[(c)] There holds for all $t \geq s \geq 1$ that
\begin{equation}\label{eq:convexity-main-thm}
\frac{1}{8}\int_{\bB}|\nabla u(\cdot, s) - \nabla u(\cdot, t)|^2\, dx \leq E_H(u(\cdot, s), u_0) - E_H (u(\cdot, t), u_0)\,.
\end{equation}
\vskip 1mm
\item[(d)] There exists $v\in \big(u_0 + W_0^{1, 2}(\bB; \RR^3)\big) \cap C^{2}_{\loc} (\bB; \RR^3)$ such that, as $t \to \infty$, we have
\[
u(\cdot, t) \to v\quad\text{in }C^{2}_{\loc}(\bB), \text{ and strongly in }W^{1, 2}(\bB).
\]
The map $v$ satisfies the energy bound $D(v) \leq 2D(u_0)$, and solves the equation
\begin{equation}\label{eq:uniformity-limit-PDE}
\Delta v = 2H(v)\,v_{x^1} \times v_{x^2} \quad \text{in }\bB.
\end{equation}
In addition, $v$ extends continuously to $\overline{\bB}$, and the convergence $u(\cdot, t) \to v$ holds also in $C^0(\overline{\bB}; \RR^3)$.
\end{enumerate}
\end{thm}
\begin{proof}
Uniqueness follows from Proposition~\ref{prop:stability}. For existence, let $(u_0^{m})$ be a sequence of maps in $C^{\infty}(\overline{\bB}; \RR^3)$ such that 
\begin{equation}\label{eq:approximate-C-D-data}
\|\nabla u_0^{m} - \nabla u_0\|_{2; \bB} + \sup_{\bB}|u_0^{m} - u_0| \to 0 \quad \text{as }m \to \infty.
\end{equation}
Dropping finitely many initial terms if necessary, we can also assume that
\begin{equation}\label{eq:u0m-small-energy}
D(u_0^{m}) < \min\{\ep_3, \frac{\ep_{\text{st}}}{4}\}, \quad \text{for all }m.
\end{equation}
Let $\alpha$ be the exponent fixed at the start of Section~\ref{sec:regular-solutions}. By Proposition~\ref{prop:existence-flow}, for each $m$, there exists a regular solution \( u^m(x,t) \in C^{2+\alpha, 1+\frac{\alpha}{2}} (\overline{\bB}\times [0,\infty); \RR^3) \) to the Cauchy–Dirichlet problem~\eqref{eq:H-flow} with $u^{m}_{0}$ in place of $u_0$. The results of Section~\ref{sec:regular-solutions} then yield the following estimates:
\vskip 1mm
\begin{enumerate}
\item[(e1)] By Lemma~\ref{lemm:energy-bound-from-isoperimetric} and~\eqref{eq:u0m-small-energy}, we have for all $m \in \NN$ that
\begin{subequations}
\begin{equation}\label{eq:approximation-energy-bound}
\sup_{t > 0}D(u^m(\cdot, t)) \leq 2D(u_0^m) < \min\{2\ep_{3}, \frac{\ep_{\text{st}}}{2}\},
\end{equation}
\begin{equation}\label{eq:approximation-ut-bound}
\int_{0}^{\infty}\int_{\bB}|u^{m}_{t}|^2 \, dxdt  \leq 2D(u_0^m) < \min\{2\ep_{3}, \frac{\ep_{\text{st}}}{2}\}.
\end{equation}
\end{subequations}
\vskip 1mm
\noindent In addition, there holds for all $t > 0$ that
\begin{equation}\label{eq:approximation-isoperimetric}
\big|V_H(u^m(\cdot, t), u_0^m)\big| \leq \frac{1}{8}\big( D(u^m(\cdot, t)) + D(u_0^m) \big).
\end{equation}
\vskip 2mm
\item[(e2)] Thanks to~\eqref{eq:approximation-energy-bound},  given $k, l \in \NN$, the small energy assumption~\eqref{eq:stability-small-energy-assumption} is fulfilled with $u, v$ replaced by $u^k$ and $u^{l}$, and we deduce from Proposition~\ref{prop:stability} and~\eqref{eq:approximate-C-D-data} that, for all $T \in (0, \infty)$, there holds
\begin{equation}\label{eq:approximation-convergence-in-W12}
\int_{0}^{T}\int_{\bB} |\nabla u^{k} - \nabla u^{l}|^2 \, dxdt + \sup_{0 \leq t \leq T} \int_{\bB}|u^{k}(\cdot, t) - u^{l}(\cdot, t)|^2\, dx \rightarrow 0,
\end{equation}
as $k, l \to \infty$.
\vskip 2mm
\item[(e3)] By Lemma~\ref{lemm:speed-monotone}(a), for all $m \in \NN$ and $\tau \in (0, 1)$, the distributional derivative $\nabla u^m_{t}$ belongs to $L^2(\bB \times (\tau, \infty))$, and satisfies by~\eqref{eq:nabla-u-t-L2} and~\eqref{eq:approximation-ut-bound} the estimate
\begin{equation}\label{eq:nabla-um-t-L2}
\int_{\tau}^{\infty}\int_{\bB} |\nabla u_{t}^{m}|^2\, dxdt \leq \frac{4}{\tau^2}\int_{\frac{\tau}{2}}^{\tau}\int_{\bB}|u^{m}_{t}|^2\, dxdt \leq \frac{8}{\tau^2} D(u^{m}_{0}).
\end{equation}
Thus, given $k,l \in \NN$ and $T > 1 > \tau > 0$, arguing as in the remarks leading up to~\eqref{eq:C0-L2-estimate} before Definition~\ref{defi:weak-solution}, we see that the function $t \mapsto \int_{\bB \times \{t\}} |\nabla u^{k} - \nabla u^{l}|^2\, dx$, continuous to begin with, is absolutely continuous on $[\tau, T]$, and there holds 
\[
\begin{split}
\sup_{\tau \leq t \leq T}\|\nabla u^{k}(\cdot, t) - \nabla u^{l}(\cdot, t)\|_{2; \bB}^2 \leq\ & (T - \tau)^{-1}\cdot\|\nabla u^{k} - \nabla u^{l}\|_{2; \bB \times [\tau, T]}^2 \\
& + 2\|\nabla u^{k} - \nabla u^{l}\|_{2; \bB \times [\tau, T]}\cdot \|\nabla u^{k}_t - \nabla u^{l}_t\|_{2; \bB \times [\tau, T]}.
\end{split}
\]
Using the first part of~\eqref{eq:approximation-convergence-in-W12} along with~\eqref{eq:nabla-um-t-L2}, we deduce
\begin{equation}\label{eq:approximation-convergence-in-W12-slice}
\sup_{\tau \leq t \leq T}\int_{\bB} |\nabla u^{k}(\cdot, t) - \nabla u^{l}(\cdot, t)|^2\, dx \to 0,\quad\text{as }k, l \to \infty.
\end{equation}
\vskip 2mm
\item[(e4)] By the estimates~\eqref{eq:existence-flow-grad-estimate} and~\eqref{eq:existence-flow-sup-estimate} in Proposition~\ref{prop:existence-flow}, and the uniform upper bounds on $D(u_0^m)$ and $\sup_{\bB}|u_0^m|$ resulting from~\eqref{eq:approximate-C-D-data}, along with the interior $C^{2 + \alpha, 1 + \frac{\alpha}{2}}$-estimate in Lemma~\ref{lemm:gradient-holder-interior}, we see that for all $\delta > 0$, there exists $C> 0$, depending only on $\delta$, $\Lambda_0$, $\Lambda_1$, $D(u_0)$, and $\sup_{\bB}|u_0|$, such that
\begin{equation}\label{eq:approximation-C2a-estimate}
\sup_{m \in \NN}|u^{m}|_{2 + \alpha, 1 + \frac{\alpha}{2};\, \bB_{1 - \delta} \times (\delta^2, \infty)} \leq C.
\end{equation}
\end{enumerate}

By~\eqref{eq:approximation-C2a-estimate}, there exists a limiting map $u \in C^{2 + \alpha, 1 + \frac{\alpha}{2}}_{\loc}(\bB \times (0, \infty); \RR^3)$, along with a subsequence of $(u^{m})$, which we do not relabel, such that 
\begin{equation}\label{eq:local-C21-convergence}
\begin{array}{l}
\nabla^i u^{m} \to \nabla^i u\ (i = 0, 1, 2) \\[2pt]
(u^{m})_{t} \to u_{t}
\end{array},\quad
\text{uniformly locally on }\bB \times (0, \infty).
\end{equation}
The estimates in Proposition~\ref{prop:existence-flow} carry over to the limit, so we have for all $(x, t) \in \bB \times (0, \infty)$ that
\begin{equation}\label{eq:gradient-estimate-passes-to-limit}
|\nabla u(x, t)|^2 \leq  C \max\{(1 - |x|)^{-2}, t^{-1}\} \cdot D(u_0),
\end{equation}
\begin{equation}\label{eq:C0-estimate-passes-to-limit}
|u(x, t)|^2 \leq C\max\{1, t^{-1}\} \cdot (\|u_0\|_{\infty}^2 + D(u_0)).
\end{equation}
Next, using the $L^2$-bounds provided by~\eqref{eq:approximation-ut-bound} and~\eqref{eq:nabla-um-t-L2}, and passing again to a subsequence if needed, we see that $(u^{m}_{t})$ converges weakly in $L^2(\bB \times (0, \infty))$, while $(\nabla u^m_{t})$ does so in $L^2(\bB \times (\tau, \infty))$ for all $\tau \in (0, 1)$. By the local uniform convergence $u^m \to u$ in~\eqref{eq:local-C21-convergence}, the weak limits coincide respectively with $u_{t}$ and $\nabla u_{t}$, which inherit from $(u^m_t)$ and $(\nabla u^m_t)$ the following bounds:
\begin{equation}\label{eq:speed-L2-passes-to-limit}
\int_{0}^{\infty}\int_{\bB}|u_{t}|^2\, dxdt \leq 2D(u_0),
\end{equation}
\begin{equation}\label{eq:nabla-um-t-L2-passes-to-limit}
\int_{\tau}^{\infty}\int_{\bB} |\nabla u_t|^2\, dxdt \leq \frac{8}{\tau^2}D(u_0), \quad\text{for all }\tau \in (0, 1).
\end{equation}
For all $T > \tau > 0$, by~\eqref{eq:approximation-convergence-in-W12} and~\eqref{eq:approximation-convergence-in-W12-slice}, along with~\eqref{eq:local-C21-convergence} and Fatou's lemma, we have further that
\begin{equation}\label{eq:L2-W12-convergence}
\int_{0}^{T}\int_{\bB} |u^{m} -u |^2 + |\nabla u^{m} - \nabla u|^2\, dxdt \longrightarrow 0,   
\end{equation}
\begin{equation}\label{eq:C0-W12-convergence}
\sup_{0 < t \leq T} \|u^{m}(\cdot, t) - u(\cdot, t)\|_{2; \bB}  + \sup_{\tau \leq t \leq T} \|\nabla u^{m}(\cdot, t) - \nabla u(\cdot, t)\|_{2; \bB} \longrightarrow 0.
\end{equation}

\vskip 1em

We now prove that $u$ is a weak solution of~\eqref{eq:H-flow} on $\bB \times (0, T)$ for all $T > 1$, and that it satisfies~\eqref{eq:small-energy-main-thm}. To start, the estimate~\eqref{eq:speed-L2-passes-to-limit}, together with~\eqref{eq:L2-W12-convergence}, implies that $u$ lies in the correct function space for Definition~\ref{defi:weak-solution}, namely $W^{1, 2}(\bB \times (0, T); \RR^3)$. To check item (i) in the definition, we note by~\eqref{eq:approximate-C-D-data} and~\eqref{eq:C0-W12-convergence} that $u^{m}(\cdot, t) - u_{0}^{m} \to u(\cdot, t) - u_0$ strongly in $W^{1, 2}(\bB)$ for all $t > 0$, and hence
\begin{equation}\label{eq:limit-boundary-condition}
u(\cdot, t) - u_0 \in W^{1, 2}_0(\bB)\quad\text{for all }t > 0.
\end{equation}
For item (ii) of Definition~\ref{defi:weak-solution}, we use the first part of~\eqref{eq:C0-W12-convergence}, as well as~\eqref{eq:approximate-C-D-data} and the condition $u^m(\cdot, 0) = u^m_0$, to see that $t \mapsto u(\cdot, t)$ is continuous from $(0, \infty)$ to $L^2(\bB; \RR^3)$, and that
\[
\lim_{t \to 0^+} \|u(\cdot, t) - u_0\|_{2; \bB} = 0.
\]
Finally, given $\zeta \in C^{\infty}_{c}(\bB \times (0, T); \RR^3)$, an integration by parts gives
\begin{equation}\label{eq:um-PDE}
\int_{0}^{T}\int_{\bB} u^m_{t} \cdot \zeta + \bangle{\nabla u^{m}, \nabla \zeta}\, dxdt = -2\int_{0}^{T}\int_{\bB} H(u^m)\, u^m_{x^1} \times u^m_{x^2}\cdot \zeta\, dxdt,
\end{equation}
and we deduce from~\eqref{eq:local-C21-convergence} that~\eqref{eq:H-flow-weak} holds. Thus we have verified that $u$ is a weak solution to~\eqref{eq:H-flow} on each $\bB \times (0, T)$. Furthermore, by~\eqref{eq:C0-W12-convergence} and~\eqref{eq:approximation-energy-bound}, we have for all $t > 0$ that
\begin{equation}\label{eq:energy-bound-passes-to-limit}
D(u(\cdot, t)) = \lim_{m \to \infty}D(u^{m}(\cdot, t)) \leq 2D(u_0) < \min\{2\ep_3, \frac{\ep_{\text{st}}}{2}\},
\end{equation}
which in particular gives~\eqref{eq:small-energy-main-thm}.

We proceed to establish the asserted additional properties of $u$. Recalling that $\ep_{\text{st}} < \ep_{\reg}$, we see from~\eqref{eq:energy-bound-passes-to-limit} and Lemma~\ref{lemm:weak-solution-regular} that $u \in C^{2 + \mu, 1 + \frac{\mu}{2}}_{\loc}(\bB \times (0, \infty))$ for all $\mu \in (0, 1)$. For property (a), the continuity of $t \mapsto u(\cdot, t)$ as a map from $(0, \infty)$ to $W^{1, 2}(\bB; \RR^3)$ follows from~\eqref{eq:C0-W12-convergence}, while the second part of (a) is~\eqref{eq:limit-boundary-condition}. For property (b), we use Lemma~\ref{lemm:boundary-continuity}. First we observe that, for all $0 < a < b < \infty$, the collection
\[
\mathscr{K}(a, b) : = \{u^{m}(\cdot, t)\ |\ m \in \NN,\ t \in [a, b]\}
\]
is a pre-compact subset of $W^{1, 2}(\bB; \RR^3)$ thanks to~\eqref{eq:C0-W12-convergence} and the continuity property we just noted. In particular, given $[t_1, t_2] \subset (0, \infty)$ and $\ep \in (0, \eta')$, with $\eta'$ being the threshold from Corollary~\ref{coro:grad-estimate-scaled}, there exists $ \delta \in (0, \frac{1}{4}\min\{\sqrt{t_1}, 1\})$ such that 
\[
\int_{\bB \setminus \bB_{1 - 2\delta}} |\nabla u^{m}(x, t)|^2\, dx < \ep, \quad\text{for all }m \in \NN,\ t \in [\frac{t_1}{2}, t_2].
\]
For all $t_0 \in [t_1, t_2]$ and $x_0 \in \bB \setminus \bB_{1 - \delta}$, we let $r = (1 - |x_0|)/2$, and define the rescaled maps
\[
\tilde{u}^m(x, t) = u^m(x_0 + rx, t_0 - r^2 + r^2 t),
\]
for $(x, t) \in \bB \times (0, 1]$. From the inclusions $\bB_{r}(x_0) \subset \bB \setminus \bB_{1 - 2\delta}$ and $(t_0 - r^2, t_0] \subset [\frac{t_1}{2}, t_2]$, we have
\[
\sup_{t \in (0, 1]}\int_{\bB}|\nabla\tilde{u}^m(x, t)|^2\, dx  = \sup_{t \in (t_0 - r^2, t_0]}\int_{\bB_{r}(x_0)} |\nabla u^m(x, t)|^2\, dx \leq \ep < \eta',
\]
in which case Corollary~\ref{coro:grad-estimate-scaled} gives
\[
(1 - |x_0|)^2 |\nabla u^m(x_0, t_0)|^2 = 4\cdot|\nabla \tilde{u}^m(0, 1)|^2 \leq C_{\Lambda_0}\cdot \ep.
\]
In other words, we have shown that, for all $m \in \NN$ and $t \in [t_1, t_2]$,
\begin{equation}\label{eq:uniform-Hardy-bound}
\sup_{x \in \bB \setminus \bB_{1 - \delta}} (1 - |x|)^2 |\nabla u^m(x, t)|^2 \leq C_{\Lambda_0}\cdot \ep,
\end{equation}
which implies that the convergence~\eqref{eq:Hardy-bound} takes place uniformly over the collection $\mathscr{K}(t_1, t_2)$. Since, by~\eqref{eq:approximate-C-D-data}, the corresponding set of boundary traces, namely $\{u_0^{m}\}_{m \in \NN}$, is pre-compact in $C^0(\partial \bB ;\RR^3)$, we have verified all the assumptions of Lemma~\ref{lemm:boundary-continuity}, so for all $\mu > 0$, there exists $r > 0$ with the property that
\[
|u^m(x, t) - u_0^m(p)| < \mu,
\]
for all $m\in \NN$, $t \in [t_1, t_2]$, $p \in \partial \bB$, and $x \in \bB \cap \bB_{2r}(p)$. From this and the triangle inequality, we deduce that, for all $k, l \in \NN$,
\[
\sup_{(\bB \setminus \bB_{1-r}) \times [t_1, t_2]} |u^{k} - u^{l}| \leq 2\mu + \sup_{\partial \bB}|u_0^k - u_0^l|.
\]
Each $u^m$ being continuous on $\overline{\bB} \times [t_1, t_2]$, we get further that
\begin{equation}\label{eq:uk-uniformly-Cauchy}
\sup_{\overline{\bB} \times [t_1, t_2]}|u^k - u^{l}| \leq \sup_{\bB_{1 - r} \times [t_1, t_2]}|u^k - u^l| + \big( 2\mu + \sup_{\partial \bB}|u_0^k - u_0^l|\big),
\end{equation}
which together with~\eqref{eq:local-C21-convergence},~\eqref{eq:approximate-C-D-data}, and the arbitrariness of $\mu$ shows that $(u^m)$ converges uniformly on $\overline{\bB} \times [t_1, t_2]$. Again by~\eqref{eq:local-C21-convergence}, the limiting map agrees with $u$ on $\bB \times [t_1, t_2]$. This proves property (b), since $[t_1, t_2] \subset (0, \infty)$ is also arbitrary.

For part (c), we first note that for all  $t > 0$, since $u(\cdot, t)$ lies in $L^{\infty}(\bB; \RR^3)$ by~\eqref{eq:C0-estimate-passes-to-limit}, and satisfies the boundary condition~\eqref{eq:limit-boundary-condition}, the enclosed volume $V_H(u(\cdot, t), u_0)$ indeed makes sense, and hence so does $E_H(u(\cdot, t), u_0)$. By the convergences~\eqref{eq:C0-W12-convergence} and~\eqref{eq:approximate-C-D-data}, along with the $L^{\infty}$-bound on $u^{m}(\cdot, t)$ given by~\eqref{eq:existence-flow-sup-estimate}, we may invoke Lemma~\ref{lemm:volume-convergence}, which together with~\eqref{eq:energy-bound-passes-to-limit} gives
\begin{equation}\label{eq:E-H-functional-passes-to-limit}
\lim_{m \to \infty} E_H(u^{m}(\cdot, t), u_0^{m}) = E_H(u(\cdot, t), u_0),\quad\text{for all }t > 0.
\end{equation}
Since each $u^{m}$ satisfies the convexity estimate in Proposition~\ref{prop:flow-convexity}, we conclude that~\eqref{eq:convexity-main-thm} holds for all $t \geq s \geq 1$, as asserted.

For part (d), by the isoperimetric estimate~\eqref{eq:approximation-isoperimetric}, along with~\eqref{eq:E-H-functional-passes-to-limit} and~\eqref{eq:approximate-C-D-data}, we have
\[
E_H(u(\cdot, t), u_0) \geq -\frac{1}{8}D(u_0),\quad\text{for all }t > 0.
\]
The function $t \mapsto E_{H}(u(\cdot, t), u_0)$ being non-increasing and bounded from below, we see that its limit exists as $t \to \infty$, in which case~\eqref{eq:convexity-main-thm} implies
\begin{equation}\label{eq:Cauchy-in-W12}
\lim_{s, t \to \infty}\int_{\bB}|\nabla u(\cdot, s) - \nabla u(\cdot, t)|^2\, dx = 0.
\end{equation}
Combining this with~\eqref{eq:limit-boundary-condition} and the Poincar\'e inequality, we obtain some $v \in u_0 + W^{1, 2}_0(\bB; \RR^3)$ such that
\begin{equation}\label{eq:flow-W12-limit}
u(\cdot, t) \to v\quad\text{strongly in }W^{1, 2}(\bB; \RR^3),\quad\text{as }t \to \infty.
\end{equation}
Noting that $u$ in fact satisfies the differential equation in~\eqref{eq:H-flow} classically on $\bB \times (0, \infty)$, we get from~\eqref{eq:gradient-estimate-passes-to-limit},~\eqref{eq:C0-estimate-passes-to-limit}, and Lemma~\ref{lemm:gradient-holder-interior} that
\begin{equation}\label{eq:weak-solution-C2a-bound-slices}
\limsup_{t \to \infty}\big( |u(\cdot, t)|_{2, \alpha; \bB_{r}} +|u_{t}(\cdot, t)|_{0, \alpha; \bB_{r}} \big) < \infty,\quad\text{for all }r \in (0, 1).
\end{equation}
It follows that the limit $v$ in~\eqref{eq:flow-W12-limit} lies in $C^{2, \alpha}_{\loc}(\bB;\RR^3)$, and that 
\begin{equation}\label{eq:slices-converge-in-C2}
u(\cdot, t) \to v\quad\text{in }C^{2}_{\loc}(\bB;\RR^3),\quad\text{as }t \to \infty.
\end{equation}
In addition, from~\eqref{eq:speed-L2-passes-to-limit} we obtain a sequence $t_k \to \infty$ such that $u_t(\cdot, t_k) \in L^2(\bB;\RR^3)$ for all $k$, and that 
\[
\int_{\bB}|u_{t}(\cdot, t_k)|^2\, dx \to 0,
\]
which together with~\eqref{eq:weak-solution-C2a-bound-slices} shows that $u_{t}(\cdot, t_k) \to 0$ uniformly locally on $\bB$. Combining this with~\eqref{eq:slices-converge-in-C2}, we conclude that $v$ solves~\eqref{eq:uniformity-limit-PDE}. The bound $D(v) \leq 2D(u_0)$ follows from~\eqref{eq:energy-bound-passes-to-limit} and~\eqref{eq:flow-W12-limit}. 

For the remaining assertions of (d), we again use Lemma~\ref{lemm:boundary-continuity}. By the continuity of $t \mapsto u(\cdot, t)$ from $(0, \infty)$ to $W^{1, 2}(\bB; \RR^3)$, together with the convergence~\eqref{eq:flow-W12-limit}, we have for all $\tau > 0$ that $\{u(\cdot, t)\ |\ t \geq \tau\}$ is a pre-compact subset of $W^{1,2}(\bB)$. Then, following the argument leading up to~\eqref{eq:uniform-Hardy-bound}, we see that~\eqref{eq:Hardy-bound} holds uniformly on $\{u(\cdot, t)\ |\ t \geq 2\}$. Since each member of this collection has trace $u_0$ on $\partial \bB$ by~\eqref{eq:limit-boundary-condition}, we have again shown that Lemma~\ref{lemm:boundary-continuity} is applicable. Thus, for all $\ep > 0$, there is $\delta > 0$ such that
\[
|u(x, t) - u_0(p)| < \ep,\quad\text{for all }t \geq 2,\, p \in \partial \bB,\, x \in \bB \cap \bB_{2\delta}(p).
\]
Repeating the argument leading to~\eqref{eq:uk-uniformly-Cauchy}, we get for all $t \geq s \geq 2$ that
\[
\sup_{\overline{\bB}}|u(\cdot, s) - u(\cdot, t)| \leq \sup_{\bB_{1 - \delta}}|u(\cdot, s) - u(\cdot, t)| + 2\ep.
\]
From this and~\eqref{eq:slices-converge-in-C2}, we conclude that $u(\cdot, t)$ converges uniformly on $\overline{\bB}$ as $t\to \infty$, and the limit agrees with $v$ on $\bB$. This finishes the proof of property (d), and of the theorem.
\end{proof}

\appendix 

\section{Standard estimates assuming bounded gradient}\label{sec:further-a-priori}
For use primarily in the proof of Proposition~\ref{prop:existence-flow} and Theorem~\ref{thm:uniformity-restated}, in this appendix we recall some standard a priori estimates for solutions to the $H$-surface flow~\eqref{eq:H-flow}. Given $(x, t) \in \RR^2 \times \RR$ and $r > 0$, we define
\[
\bP_{r}(x, t) = \bB_{r}(x) \times (t-r^2, t],
\]
and write $\bP_{r}$ for $\bP_{r}(0, 0)$. With $\RR^{2}_{+}$ being the open upper half-plane, we write $\bB_r^+(x)$ and $\bT_{r}(x)$ for the intersections of $\bB_{r}(x)$ with $\RR^2_{+}$ and $\partial \RR^2_{+}$, respectively, and set
\[
\bP_{r}^+(x, t) = \bB_{r}^+(x) \times (t - r^2, t], \quad \bL_{r}(x, t) = \bT_{r}(x) \times (t - r^2, t].
\]
Next, for $R \geq 1$, we define
\[
\Omega_{R} = \bB_R({\rm i}R) \subset \CC.
\]
By a scaling argument, there is a universal constant $\theta_0 > 0$ such that
\begin{equation}\label{eq:measure-lower-bound}
|\bB_{r}(x_0) \cap \Omega_{R}| \geq \theta_0 r^2, \quad\text{for all }x_0 \in \overline{\Omega_R} \text{ and }r \leq \frac{R}{4}.
\end{equation}
Also, for all $\delta \in (0, 1)$, there exists $R_{\delta} > 1$ such that whenever $R \in [R_{\delta}, \infty)$, the conformal map
\begin{equation}\label{eq:flattening-conformal-map}
\begin{array}{cccc}
F_{R}:& \RR^2_+&  \longrightarrow\ & \Omega_{R}\\[2pt]
& z  & \longmapsto\ &  \frac{2{\rm i}R z}{z+ 2{\rm i}R},
\end{array}
\end{equation}
which we often denote simply by $F$, satisfies that
\begin{equation}\label{eq:F-close-to-id}
\|F - \id\|_{C^3(\bB^+)} + \|F^{-1} - \id\|_{C^3(\bB \cap \Omega_{R})} < \delta,
\end{equation}
and that
\begin{equation}\label{eq:F-bilipschitz}
\bB_{(1-\delta)r} \cap \Omega_{R} \subset F(\bB^+_{r}) \subset \bB_{(1 + \delta)r} \cap \Omega_{R}, \text{ for all }r \in (0, 1].
\end{equation}

Given $\alpha \in (0, 1)$ and a continuous function $f$ on a bounded domain of the form $E = U \times (a, b)$ in $\RR^2 \times \RR$, similar to~\cite[pages 7-8]{LSU}, we define
\begin{align*}
[f]^{(x)}_{\alpha; E} :=\ & \sup_{\substack{(x_1, t), (x_2, t) \in E\\x_1 \neq x_2}}\frac{|f(x_1, t) - f(x_2, t)|}{|x_1 - x_2|^{\alpha}},\nonumber\\
[f]^{(t)}_{\alpha; E} :=\ & \sup_{\substack{(x, t_1), (x, t_2) \in E\\t_1 \neq t_2}}\frac{|f(x, t_1) - f(x, t_2)|}{|t_1 - t_2|^{\alpha}},\nonumber\\
[f]_{\alpha, \frac{\alpha}{2}; E} :=\ & [f]^{(x)}_{\alpha; E} + [f]^{(t)}_{\frac{\alpha}{2}; E}.
\end{align*}
The space $C^{2 + \alpha, 1 + \frac{\alpha}{2}}(\overline{E})$ then consists of functions such that $f$, $\nabla f$, $\nabla^2 f$, and $f_{t}$ are continuous on $E$, and that the following norm is finite:
\[
\begin{split}
|f|_{2 + \alpha, 1 + \frac{\alpha}{2}; E} :=\ & \|f\|_{\infty; E} + \|\nabla f\|_{\infty; E} + \|\nabla^2 f\|_{\infty; E} + \|f_{t}\|_{\infty; E}\\
&+ [\nabla f]^{(t)}_{\frac{1 + \alpha}{2}; E} + [\nabla^2 f]_{\alpha, \frac{\alpha}{2}; E} + [f_{t}]_{\alpha, \frac{\alpha}{2}; E},
\end{split}
\]
in which case the functions appearing on the right-hand side all extend continuously to $\overline{E}$.

The lemma below is a Poincar\'e-type inequality adapted from~\cite{Krylov-book}. A proof is included for the reader's convenience.
\begin{lemm}[\cite{Krylov-book}, Chapter 4.2, Lemmas 1 and 2]
\label{lemm:parabolic-Poincare}
Given $x_0 \in \overline{\RR^{2}_{+}}$ and $ r > 0$, let $u$ be a function in $C^{2 + \alpha, 1 + \frac{\alpha}{2}}(\overline{\bP^{+}_{r}(x_0, t_0)})$, and assume, in the case $\bL_{r}(x_0, t_0) \neq \emptyset$, that
\begin{equation}\label{eq:u-t-vanish-on-flat-part}
u_{t} = 0 \quad \text{on }\bL_{r}(x_0, t_0).
\end{equation}
Then we have
\begin{equation}\label{eq:parabolic-Poincare}
\int_{\bP_{r}^+(x_0, t_0)}|\nabla u - (\nabla u)_{\bP_r^+(x_0, t_0)}|^2dx dt \leq Cr^2\int_{\bP_{r}^+(x_0, t_0)} \big(|\nabla^2 u |^2 + |u_t|^2\big)dxdt,
\end{equation}
where $C$ is a universal constant.
\end{lemm}
\begin{proof}
The proof is a straightforward adaptation of that of~\cite[page 98, Lemma 1]{Krylov-book}. Fix a cutoff function $\zeta \in C^{\infty}_{c}(\bB; [0, 1])$ that equals $1$ on $\bB_{\frac{1}{2}}$ and notice that for all $r > 0$, because $x_0^{2} \geq 0$, we have by a change of variables that
\begin{equation}\label{eq:Poincare-cutoff-integral-bounds}
r^{-2}\int_{\bB_{r}^+(x_0)} \zeta\big(\frac{x - x_0}{r}\big) dx = \int_{\bB \cap \{y^2 \geq -\frac{x_0^2}{r}\}} \zeta(y) dy \geq \int_{\bB^+} \zeta =: c_0  > 0.
\end{equation}
Next, for $t \in (t_0 - r^2, t_0]$, we define
\[
g(t) =\frac{\int_{\bB_{r}^+(x_0)} \zeta\big(\frac{y - x_0}{r}\big)\nabla u(y, t) dy}{\int_{\bB_{r}^+(x_0)}\zeta\big(\frac{y - x_0}{r} \big)dy}.
\]
For all $x \in \bB_{r}^+(x_0)$, a standard application of H\"older's inequality gives
\[
|\nabla u(x, t) - g(t)|^2 \leq \frac{\int_{\bB_{r}^+(x_0)} \zeta\big(\frac{y - x_0}{r}\big) \big| \nabla u(x, t) - \nabla u(y, t) \big|^2 dy}{\int_{\bB_{r}^+(x_0)}\zeta\big(\frac{y - x_0}{r} \big)dy},
\]
from which we get, after integrating with respect to $x$ over $\bB_{r}^+(x_0)$ and using~\eqref{eq:Poincare-cutoff-integral-bounds}, that
\[
\begin{split}
\int_{\bB_{r}^{+}(x_0)} |\nabla u(x, t) - g(t)|^2 dx \leq \ & c_{0}^{-1}r^{-2} \int_{\bB_{r}^{+}(x_0)} \int_{\bB_{r}^{+}(x_0)} |\nabla u(x, t) - \nabla u(y, t)|^2 dy dx\\
\leq\ & 2c_0^{-1}r^{-2}\int_{\bB_{r}^{+}(x_0)} \Big( \int_{\bB_{r}^{+}(x_0)} |\nabla u(x, t) - (\nabla u(\cdot, t))_{\bB_{r}^{+}(x_0)}|^2 dx\Big) dy\\
& + 2c_0^{-1}r^{-2}\int_{\bB_{r}^{+}(x_0)} \Big(\int_{\bB_{r}^{+}(x_0)} |\nabla u(y, t) - (\nabla u(\cdot, t))_{\bB_{r}^{+}(x_0)}|^2 dy\Big) dx\\
\leq \ & C \int_{\bB_{r}^{+}(x_0)} |\nabla u(x, t) - (\nabla u(\cdot, t))_{\bB_{r}^{+}(x_0)}|^2 dx.
\end{split}
\]
By the Poincar\'e inequality for $\bB \cap \{y^2 \geq -\frac{x_0^2}{r}\}$, in which the constant can be chosen to be independent of $x_0 \in \overline{\RR^2_+}$ and $ r > 0$, we infer from the above estimate that
\begin{equation}\label{eq:poincare-in-space}
\int_{\bB_{r}^+(x_0)} |\nabla u(x, t) - g(t)|^2 dx \leq Cr^2\int_{\bB_{r}^+(x_0)} |\nabla^2 u(x, t)|^2 dx.
\end{equation}

To continue, we define
\[
\overline{g} : = \frac{1}{r^2}\int_{t_0-r^2}^{t_0} g(t) dt.
\]
Noting that $(\nabla u)_{\bP_r^+(x_0, t_0)}$ is the choice of constant that minimizes the left-hand side of~\eqref{eq:parabolic-Poincare}, we find that
\begin{equation}\label{eq:parabolic-poincare-almost}
\begin{split}
&\int_{\bP_{r}^+(x_0, t_0)}|\nabla u - (\nabla u)_{\bP_r^+(x_0, t_0)}|^2dx dt\leq\int_{\bP_{r}^{+}(x_0, t_0)} |\nabla u - \overline{g}|^2 dx dt \\
\leq\ & 2\int_{\bP_{r}^+(x_0, t_0)} |\nabla u(x, t) - g(t)|^2 + |g(t) - \overline{g}|^2 dx dt\\
\leq\ & Cr^2 \int_{\bP_{r}^+(x_0, t_0)} |\nabla^2 u|^2 dx dt + Cr^2 \int_{t_0 - r^2}^{t_0} |g(t) - \overline{g}|^2 dt,
\end{split}
\end{equation}
where we integrated~\eqref{eq:poincare-in-space} with respect to $t$ in getting the last line. For the very last integral, we first return to the definition of $g$ and integrate by parts to get
\[
g(t) = -\frac{\big(0, \int_{\bT_{r}(x_0)}\zeta\big(\frac{y - x_0}{r} \big) u(y, t) dy^{1}\big) + \int_{\bB_{r}^+(x_0)}\nabla\big( \zeta\big(\frac{y - x_0}{r}\big)\big) u(y, t) dy}{\int_{\bB_{r}^+(x_0)}\zeta\big(\frac{y - x_0}{r} \big)dy}.
\]
Differentiating with respect to $t$ and using the assumption~\eqref{eq:u-t-vanish-on-flat-part}, along with the lower bound~\eqref{eq:Poincare-cutoff-integral-bounds}, we see that $g'$ can be estimated as follows:
\begin{equation}\label{eq:g-prime-estimate}
|g'(t)|^2 \leq c_0^{-2}r^{-6}\|\nabla \zeta\|_{\infty}^2\Big( \int_{\bB_{r}^+(x_0)}  |u_{t}(x, t)| dx \Big)^2 \leq Cr^{-4}\int_{\bB_{r}^+(x_0)}|u_t|^2 dx,
\end{equation}
where the last step is H\"older's inequality. With this, and the fundamental theorem of Calculus, we estimate, for $t \in (t_0 - r^2, t_0]$, 
\begin{equation}\label{eq:poincare-in-time-1}
\begin{split}
|g(t) - \overline{g}|^2  =\ & \Big| \frac{1}{r^2}\int_{t_0 - r^2}^{t_0}g(t) - g(s) ds \Big|^2 \leq \frac{1}{r^2}\int_{t_0 - r^2}^{t_0}|g(t) - g(s)|^2 ds\\
\leq\ & \Big( \int_{t_0 - r^2}^{t_0} |g'(t)|dt \Big)^2 \leq r^2 \int_{t_0 - r^2}^{t_0} \big| g'(t) \big|^2 dt  \leq Cr^{-2} \int_{\bP_{r}^{+}(x_0, t_0)}|u_{t}|^2 dx dt,
\end{split}
\end{equation}
where~\eqref{eq:g-prime-estimate} is used in the last step. Recalling~\eqref{eq:parabolic-poincare-almost}, we get~\eqref{eq:parabolic-Poincare} as asserted.
\end{proof}

\begin{rmk}\label{rmk:parabolic-Poincare}
A similar argument shows that with $x_0 \in \overline{\RR^2_{+}}$, $r > 0$, and $u$ as in Lemma~\ref{lemm:parabolic-Poincare}, but without assuming~\eqref{eq:u-t-vanish-on-flat-part}, we have 
\begin{equation}\label{eq:parabolic-Poincare-easy}
\int_{\bP_{r}^+(x_0, t_0)}|u - (u)_{\bP_r^+(x_0, t_0)}|^2dx dt \leq Cr^2\int_{\bP_{r}^+(x_0, t_0)} \big(|\nabla u |^2 + r^2|u_t|^2\big)dxdt,
\end{equation}
where $C$ is again a universal constant. Indeed, notice that, for all $t \in (t_0 - r^2, t_0]$, there holds
\[
\begin{split}
|(u)_{\bP_{r}^+(x_0, t_0)} - (u(\cdot, t))_{\bB_{r}^{+}(x_0)}|^2 =\ & \Big| \frac{1}{r^2}\int_{t_0 - r^2}^{t_0}\fint_{\bB_{r}^+(x_0)} u(y, s) - u(y, t) \, dyds \Big|^2\\
\leq\ & \frac{1}{r^2}\int_{t_0 - r^2}^{t_0}\fint_{\bB_{r}^+(x_0)}|u(y, s) - u(y, t)|^2\, dyds.
\end{split}
\]
Estimating the last integrand by
\[
|u(y, s) - u(y, t)|^2 \leq \Big( \int_{t_0 - r^2}^{t_0}|u_{t}(y, \tau)|\, d\tau \Big)^2 \leq r^2 \int_{t_0 - r^2}^{t_0} |u_{t}(y, \tau)|^2\, d\tau,
\]
we see that
\begin{equation}\label{eq:slice-average-deviation}
|(u)_{\bP_{r}^+(x_0, t_0)} - (u(\cdot, t))_{\bB_{r}^{+}(x_0)}|^2 \leq r^{2}\int_{t_0 - r^2}^{t_0}\fint_{\bB_{r}^{+}(x_0)} |u_{t}|^2,
\end{equation}
and hence, for all $t \in (t_0 - r^2, t_0]$, 
\[
\begin{split}
\int_{\bB_{r}^{+}(x_0)}|u(\cdot, t) - (u)_{\bP_{r}^+(x_0, t_0)}|^2 \leq\ & 2 \int_{\bB_{r}^+(x_0)} |u(\cdot, t) - (u(\cdot, t))_{\bB_{r}^{+}(x_0)}|^2 + 2r^2\int_{\bP_{r}^+(x_0, t_0)}|u_t|^2\\
\leq\ & Cr^2\int_{\bB_r^+(x_0)}|\nabla u(\cdot, t)|^2 + 2r^2\int_{\bP_{r}^+(x_0, t_0)}|u_t|^2,
\end{split}
\]
where the second step uses the Poincar\'e inequality for $\bB \cap \{y^2 \geq -\frac{x_0^2}{r}\}$, as in the derivation of~\eqref{eq:poincare-in-space}. Integrating with respect to $t$ yields~\eqref{eq:parabolic-Poincare-easy} as claimed.
\end{rmk}
The next lemma is the main result of this appendix. 

\begin{lemm}\label{lemm:gradient-holder}
Given $q > 4$, and letting $\alpha = 1 - \frac{4}{q}$, there exists $R_0 = R_0(q) > 1$ with the following property. Suppose $R \in [R_0, \infty)$ and let $u \in C^{2 + \alpha, 1 + \frac{\alpha}{2}}((\overline{\bB \cap \Omega_R}) \times [-1, 0])$ be a solution of 
\begin{equation}\label{eq:H-flow-at-boundary}
\left\{
\begin{array}{ll}
u_{t} - \Delta u = -2H(u)u_{x^1} \times u_{x^2},& \text{ in }(\bB \cap \Omega_R) \times (-1, 0], \\
u(x, t) = u_0(x), & \text{ on } (\bB \cap \partial \Omega_R) \times (-1, 0],
\end{array}
\right.
\end{equation}
where $u_0 \in C^{2, \alpha}(\overline{\bB \cap \Omega_{R}})$ and $H \in C^{1}(\RR^3; \RR)$. Assume also that $\|H\|_{\infty; \RR^3} \leq \Lambda_0$, and that
\begin{equation}\label{eq:unit-gradient}
\sup_{(\bB \cap \Omega_R) \times (-1, 0]} |\nabla u|\leq K.
\end{equation}
Then we have
\begin{equation}\label{eq:C1a-estimate}
\|\nabla u\|_{\infty} + [\nabla u]_{\alpha, \frac{\alpha}{2}}  \leq C_{q, K, \Lambda_0}\cdot  \big( \|\nabla u\|_{2; (\bB \cap \Omega_R) \times (-1, 0]} +   \|u_0\|_{2, q; \bB \cap \Omega_R}\big).
\end{equation}
If in addition $\|\nabla H\|_{\infty; \RR^3}\leq \Lambda_1$, then 
\begin{equation}\label{eq:C2a-estimate} 
|u|_{2 + \alpha, 1 + \frac{\alpha}{2}} \leq C_{q, K, \Lambda_0, \Lambda_1} \cdot \big(  \|\nabla u\|_{2; (\bB \cap \Omega_R) \times (-1, 0]} +   |u_0|_{2+\alpha; \bB \cap \Omega_R}\big).
\end{equation}
In both estimates, the norms and semi-norms on the left-hand side are taken over the set $(\bB_\frac{1}{4} \cap \Omega_R) \times (-\frac{1}{16}, 0]$.
\end{lemm}
\begin{proof}
We shall take $R_0$ to be the radius $R_{\delta}$ from the remarks following~\eqref{eq:measure-lower-bound}, for some $\delta \in (0, \frac{1}{8})$ to be chosen depending only on $q$. With $R$ and $u$ as in the statement, let $F:\RR^2_{+} \to \Omega_R$ be the conformal map given by~\eqref{eq:flattening-conformal-map}. In particular, with $\lambda := |F_{z}|$, we have from~\eqref{eq:F-close-to-id} that 
\begin{equation}\label{eq:lambda-close-to-1}
\|\lambda^{-2} - 1\|_{\infty; \bB^+} < C\delta,
\end{equation}
where $C$ is a universal constant. By~\eqref{eq:F-bilipschitz} it makes sense to define
\begin{align}
v(x, t) =\ & u(F(x), t), \quad \text{for }(x, t ) \in \overline{\bP_{\frac{8}{9}}^{+}}, \label{eq:v-is-u-F}\\
v_0(x) =\ & u_0(F(x)), \quad \text{for }x \in \overline{\bB_{\frac{8}{9}}^{+}}.
\end{align}
The assumption~\eqref{eq:unit-gradient} and the bounds~\eqref{eq:F-close-to-id} then gives some universal constant such that
\begin{equation}\label{eq:unit-gradient-v}
\sup_{\bP_{\frac{8}{9}}^+}|\nabla v| \leq CK.
\end{equation}
Also, a direct computation shows that
\begin{equation}\label{eq:v-PDE}
\left\{
\begin{array}{ll}
v_{t} - \lambda^{-2}\Delta v  = -2\lambda^{-2} H(v) v_{x^1} \times v_{x^2}, &  \text{ in }\bP_{\frac{8}{9}}^+,\\
v = v_0, &  \text{ on }\bL_{\frac{8}{9}}.
\end{array}
\right.
\end{equation}
For all $r \in (0, \frac{8}{9}]$, by the Poincar\'e inequality applied to $v(\cdot, t) - v_0$ on $\bB_r^+$ for each $t$, we see that
\begin{equation}\label{eq:v-v0-estimate}
\|v - v_0\|_{2; \bP_{r}^{+}} \leq Cr\cdot\|\nabla v - \nabla v_0\|_{2; \bP_{r}^{+}},
\end{equation}
where $C$ is a universal constant. By the Sobolev inequality in~\cite[Lemma II.3.3]{LSU} (with $l = 1$, $r = 0$, $s = 0, 1$, and $n = 2$), we have
\begin{equation}\label{eq:Sobolev-from-LSU}
\begin{split}
\|v \|_{q; \bP_{\frac{4}{5}}^+}  + \|\nabla v\|_{q; \bP_{\frac{4}{5}}^{+}} \leq\ & C_{q}\big( \|v\|_{4; \bP_{\frac{4}{5}}^{+}} + \|\nabla^2 v\|_{4; \bP_{\frac{4}{5}}^{+}}  + \|v_{t}\|_{4; \bP_{\frac{4}{5}}^{+}}\big),\\
\|v\|_{4; \bP_{\frac{5}{6}}^{+}} + \|\nabla v\|_{4; \bP_{\frac{5}{6}}^{+}} \leq\ & C\big( \|v\|_{2; \bP_{\frac{5}{6}}^{+}}  + \|\nabla^2 v\|_{2; \bP_{\frac{5}{6}}^{+}} + \|v_{t}\|_{2; \bP_{\frac{5}{6}}^{+}} \big).
\end{split}
\end{equation}
Now take a cut-off function $\zeta \in C^{\infty}_{c}(\bP_{\frac{4}{5}})$ that equals $1$ on $\bP_{\frac{3}{4}}$, and let $w = \zeta(v - v_0)$. With the help of~\eqref{eq:v-PDE}, we have on $\RR^{2}_{+} \times (-1, 0)$ that
\begin{equation}\label{eq:w-PDE}
\begin{split}
w_t - \Delta w =\ & (\lambda^{-2} - 1)\Delta w +  \zeta_t (v - v_0) -2\lambda^{-2}\zeta H(v)v_{x^1} \times v_{x^2} \\
&+ \lambda^{-2}\zeta \Delta v_0 - 2\lambda^{-2}\nabla\zeta \cdot \nabla(v-  v_0) - \lambda^{-2} (v - v_0)\Delta \zeta\\
=:\ & f.
\end{split}
\end{equation}
Since $w$ vanishes on $\RR^2_{+} \times \{-1\}$, and also on $\partial \RR^2_+ \times (-1, 0)$, by~\cite[page 343, (9.7)]{LSU} (see also~\cite[pages 5-7, 273, and 301]{LSU} for some of the relevant notation),  we get that
\begin{align*}
&\|w_{t}\|_{q; \bP_{\frac{4}{5}}^+} + \|\nabla^2 w\|_{q; \bP_{\frac{4}{5}}^{+}} \leq C_{q} \|f\|_{q; \bP_{\frac{4}{5}}^+}\\
&\leq  C_{q} \cdot \delta \|\Delta w\|_{q; \bP_{\frac{4}{5}}^+} + C_{q}\cdot (1 + K\Lambda_0 )\big(\|v\|_{q; \bP_{\frac{4}{5}}^+}+ \|\nabla v\|_{q; \bP_{\frac{4}{5}}^+} + \|v_0\|_{2, q; \bB_{\frac{4}{5}}^+} \big),
\end{align*}
where to get the second inequality we used~\eqref{eq:lambda-close-to-1} to bound $\lambda^{-2} - 1$, and estimated the quadratic term $v_{x^1} \times v_{x^2}$ with~\eqref{eq:unit-gradient-v}. With a small enough choice of $\delta$ depending only on $q$, and noting that $\|\nabla^2 v\|_{q; \bP_{\frac{3}{4}}^+} \leq \|\nabla^2 v - \nabla^2 v_0\|_{q; \bP_{\frac{3}{4}}^+} + C_{q}\|\nabla^2 v_{0}\|_{q; \bB_{\frac{3}{4}}^+}$, we get
\begin{equation}\label{eq:W2q-estimate}
\|v_t\|_{q; \bP_{\frac{3}{4}}^+} + \|\nabla^2 v\|_{q; \bP_{\frac{3}{4}}^{+}} \leq C_{q}\cdot (1 + K\Lambda_0) \big( \|v\|_{q; \bP_{\frac{4}{5}}^+} +  \|\nabla v\|_{q; \bP_{\frac{4}{5}}^+} + \|v_0\|_{2,q; \bB_{\frac{4}{5}}^{+}} \big).
\end{equation}
Repeating the above argument and decreasing $\delta$ if necessary, we obtain two more estimates of the above type, with $(q, \frac{3}{4}, \frac{4}{5})$ replaced by $(4, \frac{4}{5}, \frac{5}{6})$ and by $(2, \frac{5}{6}, \frac{6}{7})$. The two Sobolev inequalities in~\eqref{eq:Sobolev-from-LSU}, used alternately with the two estimates just mentioned, lead to
\begin{equation}\label{eq:Lq-of-grad-v}
\begin{split}
\|v\|_{q; \bP_{\frac{4}{5}}^{+}} + \|\nabla v\|_{q; \bP_{\frac{4}{5}}^{+}} \leq\ & C \big( \|v\|_{2; \bP_{\frac{6}{7}}^{+}} + \|\nabla v\|_{2; \bP_{\frac{6}{7}}^{+}} + \|v_0\|_{2, 4; \bB_{\frac{6}{7}}^{+}} \big)\\
\leq\ & C\big( \|\nabla v\|_{2; \bP_{\frac{6}{7}}^{+}} + \|v_0\|_{2, 4; \bB_{\frac{6}{7}}^{+}} \big),
\end{split}
\end{equation}
where again $C = C(q, K, \Lambda_0)$, and the second inequality follows from~\eqref{eq:v-v0-estimate}. Next, since $v_t = 0$ on $\bL_{\frac{8}{9}}$, we may use Lemma~\ref{lemm:parabolic-Poincare}, followed by H\"older's inequality and the estimates~\eqref{eq:W2q-estimate} and~\eqref{eq:Lq-of-grad-v}, to see that for all $(x_0, t_0) \in \bP_{\frac{5}{8}}^+$ and $r \leq \frac{1}{8}$, there holds
\[
\begin{split}
\int_{\bP_{r}^+(x_0,t_0)} |\nabla v - (\nabla v)_{\bP_{r}^+(x_0, t_0)}|^2 \leq\ &  Cr^2 \big( r^{4}\big)^{\frac{q-2}{q}} (\|v_t\|_{q; \bP_{\frac{3}{4}}^+}^2 + \|\nabla^2 v\|_{q; \bP_{\frac{3}{4}}^+}^2)\\
\leq\ & Cr^{4 + 2(1 - \frac{4}{q})} \cdot \big( \|\nabla v\|_{2; \bP_{\frac{6}{7}}^+}^2 + \|v_0\|_{2, q; \bB_{\frac{6}{7}}^+}^2\big),
\end{split}
\]
where the second constant $C$ depends only on $q, K$, and $\Lambda_0$. Since $q > 4$, by the Morrey--Campanato characterization of H\"older continuity~\cite[Theorem 1]{Schlag-1996}, for all $(x,t), (y, s) \in \bP_{\frac{5}{8}}^+$ satisfying $\max\{|x - y|, |t - s|^{\frac{1}{2}}\} \leq \frac{1}{16}$, we have
\begin{equation}\label{eq:grad-v-holder}
|\nabla v(x, t) - \nabla v(y, s)| \leq C_{q, K,\Lambda_0}\cdot \max\{|x - y|^{\alpha}, |t - s|^{\frac{\alpha}{2}}\} \big(\|\nabla v\|_{2; \bP_{\frac{6}{7}}^+} + \|v_0\|_{2, q; \bB_{\frac{6}{7}}^+}\big).
\end{equation}
A standard argument involving integration over $\bP_{\frac{1}{16}}^{+}(x, t)$ for each $(x, t) \in \bP_{\frac{1}{2}}^{+}$ then yields a bound on $\|\nabla v\|_{\infty; \bP_{\frac{1}{2}}^{+}}$, which combines with~\eqref{eq:grad-v-holder} to give
\begin{equation}\label{eq:grad-v-estimate}
\|\nabla v\|_{\infty; \bP_{\frac{1}{2}}^{+}} + [\nabla v]_{\alpha, \frac{\alpha}{2}; \bP_{\frac{1}{2}}^{+}} \leq C_{q, K, \Lambda_0}\big( \|\nabla v\|_{2; \bP_{\frac{6}{7}}^{+}} + \|v_0\|_{2, q; \bB_{\frac{6}{7}}^{+}} \big).
\end{equation}
Recalling the relation~\eqref{eq:v-is-u-F} between $u$ and $v$, and using also~\eqref{eq:F-close-to-id} and~\eqref{eq:F-bilipschitz}, we get the asserted estimate~\eqref{eq:C1a-estimate}. For later use we need a similar estimate for $v$ itself. To that end, we use Remark~\ref{rmk:parabolic-Poincare} in place of Lemma~\ref{lemm:parabolic-Poincare} to get, again for all $(x_0, t_0) \in \bP_{\frac{5}{8}}^{+}$ and $r \leq \frac{1}{8}$, that
\[
\begin{split}
\int_{\bP_{r}^{+}(x_0, t_0)} |v - (v)_{\bP_{r}^{+}(x_0, t_0)}|^2  \leq\ & Cr^2\int_{\bP_{r}^{+}(x_0, t_0)} |\nabla v|^2 + r^2|v_t|^2 dxdt\\
\leq\ & C_{q, K, \Lambda_0} \cdot r^{4 + 2(1 - \frac{4}{q})} \cdot \big(\|\nabla v\|_{2; \bP_{\frac{6}{7}}^+}^2 + \|v_0\|_{2, q; \bB_{\frac{6}{7}}^+}^2\big),
\end{split}
\]
where the second step follows, as before, from H\"older's inequality along with~\eqref{eq:W2q-estimate} and~\eqref{eq:Lq-of-grad-v}. Following the argument leading to~\eqref{eq:grad-v-estimate}, and using again~\eqref{eq:v-v0-estimate}, we get
\begin{equation}\label{eq:v-holder}
\begin{split}
\|v\|_{\infty; \bP_{\frac{1}{2}}^{+}} + [v]_{\alpha, \frac{\alpha}{2}; \bP_{\frac{1}{2}}^{+}}  \leq\ & C_{q, K, \Lambda_0} \cdot \big( \|v\|_{2; \bP_{\frac{6}{7}}^+}  + \|\nabla v\|_{2; \bP_{\frac{6}{7}}^+} + \|v_0\|_{2, q; \bB_{\frac{6}{7}}^+}\big) \\
\leq\ & C_{q, K, \Lambda_0} \cdot \big(\|\nabla v\|_{2; \bP_{\frac{6}{7}}^+} + \|v_0\|_{2, q; \bB_{\frac{6}{7}}^+}\big).
\end{split}
\end{equation}

To continue, we return to~\eqref{eq:w-PDE} and choose $\zeta$ so that it lies in $C^{\infty}_{c}(\bP_{\frac{1}{2}})$, and equals $1$ on $\bP_{\frac{1}{3}}$. Then again $w$ vanishes on both $\partial\RR^2_+ \times (-1, 0)$ and $\RR^2_+ \times \{-1\}$. Noting also that $f = 0$ on the latter set, we obtain from~\cite[page 324, (6.5)]{LSU} that
\begin{equation}\label{eq:better-holder-with-f}
\begin{split}
[\nabla w]^{(t)}_{\frac{1 + \alpha}{2}; \bP_{\frac{1}{2}}^{+}}  + [\nabla^2 w]_{\alpha, \frac{\alpha}{2}; \bP_{\frac{1}{2}}^+} + [w_t]_{\alpha, \frac{\alpha}{2}; \bP_{\frac{1}{2}}^+} \leq\ &  C_{q}[f]_{\alpha, \frac{\alpha}{2}; \bP_{\frac{1}{2}}^+}.
\end{split}
\end{equation}
To bound $[f]_{\alpha, \frac{\alpha}{2}; \bP_{\frac{1}{2}}^+}$, we observe that by~\eqref{eq:F-close-to-id} and~\eqref{eq:lambda-close-to-1},
\begin{equation}\label{eq:Ca-estimate-on-f-1}
[(\lambda^{-2} - 1)\Delta w]_{\alpha, \frac{\alpha}{2}; \bP_{\frac{1}{2}}^{+}} \leq C\delta [\Delta w]_{\alpha, \frac{\alpha}{2}; \bP_{\frac{1}{2}}^{+}} + C\|\Delta w\|_{\infty; \bP_{\frac{1}{2}}^{+}}.
\end{equation}
Next, by~\eqref{eq:unit-gradient-v},~\eqref{eq:grad-v-estimate} and~\eqref{eq:v-holder}, we have
\[
\begin{split}
&\|H(v) v_{x^1} \times v_{x^2}\|_{\infty; \bP_{\frac{1}{2}}^{+}} +  [H(v) v_{x^1} \times v_{x^2}]_{\alpha,\frac{\alpha}{2}; \bP_{\frac{1}{2}}^{+}}\\
\leq\ & C\|H\|_{\infty} K\cdot  \big(\|\nabla v\|_{\infty; \bP_{\frac{1}{2}}^{+}} + [\nabla v]_{\alpha, \frac{\alpha}{2}; \bP_{\frac{1}{2}}^{+}}\big) + C\|\nabla H\|_{\infty} \cdot K^2 \cdot [v]_{\alpha, \frac{\alpha}{2}; \bP_{\frac{1}{2}}^{+}}\\
\leq\ & C_{q, K, \Lambda_0, \Lambda_1}\cdot  \big( \|\nabla v\|_{2; \bP_{\frac{6}{7}}^{+}} + \|v_0\|_{2, q; \bB_{\frac{6}{7}}^{+}} \big).
\end{split}
\]
Combining this with~\eqref{eq:Ca-estimate-on-f-1}, and estimating the remaining terms comprising $f$ in the straightforward manner using~\eqref{eq:grad-v-estimate} and~\eqref{eq:v-holder}, we get upon recalling~\eqref{eq:better-holder-with-f} that
\[
\begin{split}
[\nabla w]^{(t)}_{\frac{1 + \alpha}{2}; \bP_{\frac{1}{2}}^{+}} + [\nabla^2 w]_{\alpha, \frac{\alpha}{2}; \bP_{\frac{1}{2}}^+} + [w_t]_{\alpha, \frac{\alpha}{2}; \bP_{\frac{1}{2}}^+}  \leq \ & C_{q} \cdot \big(\delta[\nabla^2 w]_{\alpha, \frac{\alpha}{2}; \bP_{\frac{1}{2}}^{+}} +  \|\nabla^2 w\|_{\infty; \bP_{\frac{1}{2}}^{+}}) \\
&+ C_{q, K, \Lambda_0, \Lambda_1}\cdot \big(  \|\nabla v\|_{2; \bP_{\frac{6}{7}}^{+}} + |v_0|_{2 + \alpha; \bB_{\frac{6}{7}}^{+}} \big).
\end{split}
\]
Applying the interpolation inequality in~\cite[Lemma II.3.2]{LSU} to the term $\|\nabla^2 w\|_{\infty; \bP_{\frac{1}{2}}^{+}}$, and then using~\eqref{eq:v-holder} to bound $\|w\|_{\infty; \bP_{\frac{1}{2}}^{+}}$, we obtain, after decreasing $\delta$ further if needed,
\begin{equation}\label{eq:v-C2a-estimate}
[\nabla v]^{(t)}_{\frac{1 + \alpha}{2}; \bP_{\frac{1}{3}}^{+}} + [v_t]_{\alpha, \frac{\alpha}{2}; \bP_{\frac{1}{3}}^+} + [\nabla^2 v]_{\alpha, \frac{\alpha}{2}; \bP_{\frac{1}{3}}^+}  \leq C_{q, K, \Lambda_0, \Lambda_1}\cdot \big(  \|\nabla v\|_{2; \bP_{\frac{6}{7}}^{+}} + |v_0|_{2 + \alpha; \bB_{\frac{6}{7}}^{+}} \big).
\end{equation}
By~\eqref{eq:v-holder} and~\cite[Lemma II.3.2]{LSU} again, we deduce that
\[
|v|_{2 + \alpha, 1 + \frac{\alpha}{2}; \bP_{\frac{1}{3}}^+}  \leq C_{q, K, \Lambda_0, \Lambda_1}\cdot \big(  \|\nabla v\|_{2; \bP_{\frac{6}{7}}^{+}} + |v_0|_{2 + \alpha; \bB_{\frac{6}{7}}^{+}} \big).
\]
Recalling the definition~\eqref{eq:v-is-u-F} as well as the properties~\eqref{eq:F-close-to-id} and~\eqref{eq:F-bilipschitz} of $F$, we get the second asserted estimate,~\eqref{eq:C2a-estimate}. The proof is complete.
\end{proof}

Next we turn to the interior version of the estimates in Lemma~\ref{lemm:gradient-holder}. Thus, suppose $u \in C^{2+\alpha, 1 + \frac{\alpha}{2}}(\overline{\bP_{1}})$ satisfies
\begin{equation}\label{eq:H-flow-appendix}
u_{t} - \Delta u = -2H(u)u_{x^1} \times u_{x^2} \quad \text{in } \bP_{1},
\end{equation}
where $H \in C^{1} (\mathbb{R}^3, \mathbb{R})$, and assume also that
\begin{equation}\label{eq:bounds-for-interior-estimate}
\|H\|_{\infty; \RR^3} \leq \Lambda_0,\quad \|\nabla u\|_{\infty; \bP_{1}} \leq K.
\end{equation}
The next result provides an analogue of~\eqref{eq:v-v0-estimate} when combined with the Poincar\'e inequality.
\begin{lemm}\label{lemm:u-t-estimate}
Let $u$ be as above. Then for all $r \in (0, 1)$ we have
\begin{equation}\label{eq:u-t-estimate}
\int_{\bP_{r}}|u_{t}|^2 dxdt \leq C(1 - r)^{-2}\int_{\bP_{1}} |\nabla u|^2 dxdt,
\end{equation}
where $C$ depends only on $\Lambda_{0}$ and $K$.
\end{lemm}
\begin{proof}
Fix a smooth function $\varphi:\RR \to [0, 1]$ satisfying
\[
\varphi(t) = 1 \text{ if }t \leq \frac{1}{4},\quad \varphi(t) = 0 \text{ if }t \geq \frac{3}{4}, 
\]
Given $r \in (0, 1)$, define
\[
\zeta(x, t) = \varphi(\frac{|x| - r}{1 - r}) \cdot \varphi(\frac{|t| - r^2}{1 - r^2}),\quad\text{for }(x, t) \in \bP_{1}.
\]
We will use a difference quotient argument to derive~\eqref{eq:u-t-estimate}. Thus, for all $h \in (0, \frac{1-r^2}{4})$ and $(x, t) \in \bB \times (-1 + h, 0]$, we let
\[
u^{(-h)}(x, t): = u(x, t - h),\quad v_{h} = \frac{u - u^{(-h)}}{h}.
\]
Taking the inner product of both sides of~\eqref{eq:H-flow-appendix} with $\zeta^2 v_{h}$ and integrating over $\bP_{1}$ gives
\begin{equation}\label{eq:u-t-estimate-tested}
\begin{split}
\int_{\bP_{1}} \zeta^2 v_{h} \cdot u_t + \zeta^2\bangle{\nabla v_{h}, \nabla u} + 2\zeta \bangle{v_{h}\nabla \zeta, \nabla u} =\ & -2\int_{\bP_1}\zeta^2 H(u) u_{x^1} \times u_{x^2} \cdot v_{h}. 
\end{split}
\end{equation}
Using the expression $\nabla u = \frac{1}{2}( \nabla u + \nabla u^{(-h)}) +  \frac{1}{2}( \nabla u - \nabla u^{(-h)})$, we find that
\begin{equation}\label{eq:t-step-bound}
\begin{split}
\bangle{\nabla v_{h}, \nabla u} \geq\ & \frac{|\nabla u|^2 - |\nabla u^{(-h)}|^2}{2h}.
\end{split}
\end{equation}
Noting also that $\zeta$ vanishes on $\bB \times (-1, -1+h]$ by our choice of $h$, we get
\[
\begin{split}
&\int_{\bP_{1}} \zeta^2 \bangle{\nabla v_{h}, \nabla u}\\
\geq\ & \frac{1}{2h}\int_{-1}^{0}\int_{\bB}\zeta^2  |\nabla u|^2 dxdt- \frac{1}{2h}\int_{-1 + h}^{0}\int_{\bB}\zeta^2  |\nabla u^{(-h)}|^2 dxdt\\
=\ & \frac{1}{2h}\int_{-h}^{0}\int_{\bB}\zeta^2  |\nabla u|^2 dxdt + \int_{-1}^{-h}\int_{\bB}\frac{1}{2h}\big(\zeta(x, t)^2 - \zeta(x, t+h)^2 \big)  |\nabla u(x, t)|^2 dxdt,
\end{split}
\]
and hence
\begin{equation}\label{eq:tested-2nd-term-on-lhs}
\liminf_{h \to 0^+}\int_{\bP_{1}} \zeta^2 \bangle{\nabla v_{h}, \nabla u} \geq \int_{\bB \times \{0\}} \frac{\zeta^2|\nabla u|^2}{2} dx - \int_{\bP_{1}} \zeta \zeta_t |\nabla u|^2 dxdt.
\end{equation}
Taking the limit inferior as $h \to 0^+$ in~\eqref{eq:u-t-estimate-tested} and using~\eqref{eq:tested-2nd-term-on-lhs} yields
\[
\begin{split}
\int_{\bP_{1}} \zeta^2 |u_{t}|^2  + \int_{\bB \times \{0\}}\frac{\zeta^2 |\nabla u|^2}{2} \leq\ & \int_{\bP_{1}}\zeta |\zeta_t| |\nabla u|^2 + 2\zeta |\nabla \zeta||u_t||\nabla u| + 2K \|H\|_{\infty} \zeta^2 |\nabla u||u_t|.
\end{split}
\]
Dropping the second term on the left-hand side and applying Young's inequality then gives
\[
\int_{\bP_{1}}\zeta^2 |u_t|^2 \leq C\int_{\bP_1}(\zeta|\zeta_t| + |\nabla\zeta|^2 + K^2 \|H\|_{\infty}^2 \zeta^2)|\nabla u|^2.
\]
We get the desired estimate upon recalling the definition of $\zeta$. The proof is complete.
\end{proof}

\begin{lemm}\label{lemm:gradient-holder-interior}
Suppose $q > 4$, and define $\alpha = 1 - \frac{4}{q}$. Let $u\in C^{2 + \alpha, 1 + \frac{\alpha}{2}}(\overline{\bP_1})$ be a solution to~\eqref{eq:H-flow-appendix} with $H \in C^{1}(\RR^{3};\RR)$, and assume that~\eqref{eq:bounds-for-interior-estimate} holds. Then we have 
\begin{equation}\label{eq:C1a-interior}
\|\nabla u\|_{\infty; \bP_{\frac{1}{2}}} + [\nabla u]_{\alpha, \frac{\alpha}{2}; \bP_{\frac{1}{2}}} \leq C_{q, K, \Lambda_0} \cdot \|\nabla u\|_{2; \bP_{1}}.
\end{equation}
If in addition $\|\nabla H\|_{\infty; \RR^3} \leq \Lambda_1$, then 
\begin{equation}\label{eq:C2a-interior}
|u|_{2+\alpha, 1+\frac{\alpha}{2}; \bP_{\frac{1}{4}}} \leq  C_{q, K, \Lambda_0, \Lambda_1} \cdot \big( \|u\|_{2; \bP_{1}} + \|\nabla u\|_{2; \bP_{1}} \big).
\end{equation}
\end{lemm}
\begin{proof}
Taking $w$ to have the form $(u - \xi)\zeta$, where $\xi \in \RR^3$ is any constant vector and $\zeta$ a suitable cut-off function, and following the argument leading to~\eqref{eq:W2q-estimate} and the first inequality in~\eqref{eq:Lq-of-grad-v}, replacing the estimate~\cite[page 343, (9.7)]{LSU} with~\cite[page 345, (9.12)]{LSU}, and applying the Sobolev inequality in~\cite[Lemma II.3.3]{LSU} instead to $u - \xi$, we get
\[
\begin{split}
\|u_t\|_{q; \bP_{\frac{3}{4}}} + \|\nabla^2 u\|_{q; \bP_{\frac{3}{4}}} \leq\ & C_{q, K, \Lambda_0}(\|u - \xi\|_{q; \bP_{\frac{4}{5}}} + \|\nabla u\|_{q; \bP_{\frac{4}{5}}})\\
\leq\ & C_{q, K, \Lambda_0}(\|u - \xi\|_{2; \bP_{\frac{6}{7}}} + \|\nabla u\|_{2; \bP_{\frac{6}{7}}}).
\end{split}
\]
Choosing $\xi = (u)_{\bP_{\frac{6}{7}}}$, we see by Poincar\'e's inequality and Lemma~\ref{lemm:u-t-estimate} that 
\[
\int_{\bP_{\frac{6}{7}}} |u - (u)_{\bP_{\frac{6}{7}}}|^2 \leq C\int_{\bP_{\frac{6}{7}}}|u_t|^2 +|\nabla u|^2 \leq C_{\Lambda_0, K} \int_{\bP_{1}} |\nabla u|^2.
\]
Repeating the derivation of~\eqref{eq:grad-v-estimate}, we get 
\[
 \|\nabla u\|_{\infty; \bP_{\frac{1}{2}}} + [\nabla u]_{\alpha, \frac{\alpha}{2}; \bP_{\frac{1}{2}}} \leq C_{q, K, \Lambda_0}\cdot \|\nabla u\|_{2; \bP_{1}}.
\]
This gives~\eqref{eq:C1a-interior}. The argument leading to~\eqref{eq:v-holder}, on the other hand, gives
\[
\|u\|_{\infty; \bP_{\frac{1}{2}}} + [u]_{\alpha, \frac{\alpha}{2}; \bP_{\frac{1}{2}}} \leq C_{q, K, \Lambda_0}\cdot\big( \|u\|_{2; \bP_{1}} +  \|\nabla u\|_{2; \bP_{1}}\big).
\]
We omit the proof of~\eqref{eq:C2a-interior}, since the modification involved is more straightforward. 
\end{proof}

\section{Gradient estimates assuming small energy}\label{sec:gradient-est-app}
We adopt the notation of Appendix~\ref{sec:further-a-priori}. In this appendix, we reproduce the argument leading to the standard gradient estimate~\eqref{eq:grad-estimate}. In fact, for use in Sections~\ref{sec:regular-solutions} and~\ref{sec:existence-uniformity}, we shall consider the parabolic version, from which~\eqref{eq:grad-estimate} follows. To save space, given $R \in [1, \infty)$ and $r \in (0, 1]$, we write
\[
\begin{split}
\bP_{r, R}(x_0, t_0) :=\ & (\bB_{r}(x_0) \cap \overline{\Omega_R}) \times (t_0-r^2, t_0] = \bP_{r}(x_0, t_0) \cap (\overline{\Omega_R} \times \RR),
\end{split}
\]
and omit $(x_0, t_0)$ from the notation when it equals $(0, 0)$. Also, we write $C^{2 + \alpha, 1 + \frac{\alpha}{2}}(\bP_{1, R})$ for the intersection of the spaces $C^{2 + \alpha, 1 + \frac{\alpha}{2}}(\overline{\bP_{r, R}})$ as $r$ ranges over $(0, 1)$. The meaning of $C^{2 + \alpha, 1 + \frac{\alpha}{2}}(\bP_{1})$ is similar.
\begin{lemm}\label{lemm:eta-to-gradient-bound}
Given $q > 4$ and  $\Lambda_0, \Lambda_{\partial} > 0$, define $\alpha = 1 - \frac{4}{q}$ as before, and let $R_0 = R_0(q)$ be as given by Lemma~\ref{lemm:gradient-holder}. There exists $\eta = \eta(q, \Lambda_0, \Lambda_{\partial}) \in (0, 1)$ with the following property. Suppose $R \in [R_0, \infty)$, and that $u \in C^{2+\alpha, 1+\frac{\alpha}{2}}(\bP_{1, R})$ is a solution to~\eqref{eq:H-flow-at-boundary}, where $u_0 \in C^{2, \alpha}(\bB \cap \overline{\Omega_{R}})$, $H \in C^{1}(\RR^3; \RR)$, and 
\begin{equation}\label{eq:bounds-on-data}
\|u_0\|_{2, \infty; \bB \cap \Omega_{R}} \leq \Lambda_{\partial}, \quad \|H\|_{\infty} \leq \Lambda_0.
\end{equation}
Assume in addition that $u$ satisfies 
\begin{equation}\label{eq:interior-gradient-small-energy}
\sup_{-1<  t \leq 0}\int_{\bB \cap \Omega_{R}}|\nabla u(\cdot, t)|^2 \leq \eta.
\end{equation}
Then for all $r \in (0, 1)$ there holds 
\begin{equation}\label{eq:eta-to-grad-bound}
(1-r) \cdot \sup_{\bP_{r, R}}|\nabla u| \leq 4.
\end{equation}
\end{lemm}
\begin{proof}
The idea of proof comes from the work of Fraser~\cite[Lemma 1.6]{Fraser2000}. We first prove the lemma assuming in addition that $u \in C^{2 + \alpha, 1 + \frac{\alpha}{2}}(\overline{\bP_{1, R}})$. In this case there exist $\rho \in [0, 1]$ and $(x_0, t_0) \in \overline{\bP_{\rho, R}}$ such that 
\begin{equation}\label{eq:point-picking}
(1 - \rho)^2 |\nabla u(x_0, t_0)|^2 = (1 - \rho)^2\sup_{\bP_{\rho, R}}|\nabla u|^2 = \sup_{0 \leq r \leq 1}\big[(1 - r)^2 \sup_{\bP_{r, R}}|\nabla u|^2\big].
\end{equation}
If the left-most term vanishes, then we have $\nabla u \equiv 0$ on $\bP_{1, R}$ and there is nothing to prove. Thus below we assume that 
\[
\rho_0:= \frac{1}{2}(1 - \rho) > 0, \quad e_0: = |\nabla u(x_0, t_0)|^2 > 0.
\]
Then, from the inclusions
\[
\bP_{\rho_0}(x_0, t_0) \subset \bB_{\frac{1 + \rho}{2}} \times (-\rho^2-\rho_0^2, 0] \subset \bP_{\frac{1 + \rho}{2}},
\]
we infer that for all $(x, t) \in \bP_{\rho_0, R}(x_0, t_0)$ there holds $\big( \frac{1 - \rho}{2} \big)^2 |\nabla u(x, t)|^2 \leq (1 - \rho)^2e_0$, and hence
\begin{equation}\label{eq:grad-bound-from-point-picking}
\|\nabla u\|_{\infty; \bP_{\rho_0, R}(x_0, t_0)} \leq 2e_0^{\frac{1}{2}}.
\end{equation}
Next, suppose towards a contradiction that 
\begin{equation}\label{eq:interior-gradient-contradiction}
\rho_1:= e_0^{-\frac{1}{2}} \leq \frac{\rho_0}{2}.
\end{equation}
Then, with the help of~\eqref{eq:interior-gradient-small-energy}, we have
\begin{equation}\label{eq:infimum-upper-bound}
\begin{split}
\big( \eta^{\frac{1}{3}}\rho_1\big)^2 \cdot \eta \geq \ &\int_{\bP_{\eta^{\frac{1}{3}}\rho_1, R}(x_0, t_0)}|\nabla u|^2 \\
\geq\ & \theta_0\big( \eta^{\frac{1}{3}}\rho_1\big)^4 \cdot \inf\big\{ |\nabla u(x, t)|^2\  \big|\ (x, t) \in  \bP_{\eta^{\frac{1}{3}}\rho_1, R}(x_0, t_0) \big\},
\end{split}
\end{equation}
where we used~\eqref{eq:measure-lower-bound} for the second inequality. Imposing the requirement 
\begin{equation}\label{eq:eta-threshold-1}
\eta <\Big( \frac{\min\{1, \theta_0\}}{16}\Big)^3,
\end{equation}
we get from~\eqref{eq:infimum-upper-bound} some $(x_1, t_1) \in \bP_{\eta^{\frac{1}{3}}\rho_1, R}(x_0, t_0)$ satisfying
\[
|\nabla u(x_1, t_1)|^2 \leq \frac{2}{\theta_0}\cdot \eta^{\frac{1}{3}}\rho_1^{-2} < \frac{1}{8}\rho_1^{-2}.
\]
Recalling $|\nabla u(x_0, t_0)|^2 = e_0 = \rho_1^{-2}$, we deduce that
\begin{equation}\label{eq:difference-quotient-bound}
\frac{|\nabla u(x_0, t_0)|^2 - |\nabla u(x_1, t_1)|^2}{\max\{|x_0 - x_1|^{\alpha}, |t_0 - t_1|^{\frac{\alpha}{2}}\}} \geq \frac{1}{2}\eta^{-\frac{\alpha}{3}}\rho_1^{-2-\alpha}.
\end{equation}

To continue, we distinguish two cases. In the case where $\bB_{\frac{\rho_1}{8}}(x_0) \subset \Omega_R$, we have $\bP_{\frac{\rho_1}{8}}(x_0, t_0) = \bP_{\frac{\rho_1}{8}, R}(x_0, t_0)$, and hence~\eqref{eq:grad-bound-from-point-picking} implies that 
\begin{equation}\label{eq:point-picking-consequence-1}
\frac{\rho_1}{8}\|\nabla u\|_{\infty; \bP_{\frac{\rho_1}{8}}(x_0, t_0)} \leq \frac{1}{4}.
\end{equation}
This, and the scaling invariance of the differential equation~\eqref{eq:H-flow-appendix}, allows us to apply Lemma~\ref{lemm:gradient-holder-interior} to the map
\[
(x, t)\mapsto u(x_0 + \frac{\rho_1}{8}x, t_0 + \big(\frac{\rho_1}{8}\big)^2 t),\quad (x, t) \in \bP_{1}.
\]
In particular, scaling back the estimate~\eqref{eq:C1a-interior} and using~\eqref{eq:point-picking-consequence-1}, we get
\begin{equation}\label{eq:C1a-scaled}
\begin{split}
\rho_1^{1 + \alpha}[\nabla u]_{\alpha, \frac{\alpha}{2}; \bP_{\frac{\rho_1}{16}}(x_0, t_0)}\leq\ & C\Big( \rho_1^{-2}\int_{\bP_{\frac{\rho_1}{8}}(x_0, t_0)}|\nabla u|^2 dxdt \Big)^{\frac{1}{2}} \leq C',
\end{split}
\end{equation}
where $C'$ depends only on $q$ and $\Lambda_0$. Noting from~\eqref{eq:eta-threshold-1} that $(x_1, t_1)  \in \bP_{\frac{\rho_1}{16}}(x_0, t_0)$, and using~\eqref{eq:point-picking-consequence-1} and~\eqref{eq:C1a-scaled}, we get
\begin{align}
|\nabla u(x_0, t_0)|^2 - |\nabla u(x_1, t_1)|^2 \leq\ & 4\rho_1^{-1}|\nabla u(x_0, t_0) - \nabla u(x_1, t_1)|\nonumber\\
\leq\ & 4\rho_1^{-1}\cdot[\nabla u]_{\alpha, \frac{\alpha}{2}; \bP_{\frac{\rho_1}{16}}(x_0, t_0)} \cdot \max\{|x_0 - x_1|^{\alpha}, |t_0 - t_1|^{\frac{\alpha}{2}}\}\nonumber\\
\leq\ & K_1\rho_1^{-2-\alpha}  \cdot \max\{|x_0 - x_1|^{\alpha}, |t_0 - t_1|^{\frac{\alpha}{2}}\},\label{eq:difference-upper-bound-case-1}
\end{align}
where $K_1$ depends only on $q$ and $\Lambda_0$. 

On the other hand, in the case where $\bB_{\frac{\rho_1}{8}}(x_0) \not\subset \Omega_R$, we choose any $x_2 \in \bB_{\frac{\rho_1}{8}}(x_0) \cap \partial \Omega_R$, and observe that
\begin{equation}\label{eq:eta-to-bound-inclusions-case-2}
\bP_{\frac{\rho_1}{16}, R}(x_0, t_0) \subset \bP_{\frac{3\rho_1}{16}, R}(x_2, t_0),\quad \bP_{\frac{3\rho_1}{4}, R} (x_2, t_0)\subset \bP_{\rho_1, R}(x_0, t_0).
\end{equation}
The second of these inclusions together with~\eqref{eq:grad-bound-from-point-picking} gives 
\begin{equation}\label{eq:point-picking-consequence-2}
\frac{3\rho_1}{4} \|\nabla u\|_{\infty; \bP_{\frac{3\rho_1}{4}, R}(x_2, t_0)} \leq \frac{3}{2},
\end{equation}
so that Lemma~\ref{lemm:gradient-holder} is applicable, with $R$ replaced by $\frac{4R}{3\rho_1}$, to the maps 
\[
\begin{split}
(x, t) &\mapsto u(x_2 + \frac{3\rho_1}{4}Ax, t_0 + \big(\frac{3\rho_1}{4}\big)^2 t),\quad (x, t) \in \bP_{1, \frac{4R}{3\rho_1}},\\
x & \mapsto u_0(x_2 + \frac{3\rho_1}{4}Ax),\quad x \in \bB \cap \overline{\Omega_{\frac{4R}{3\rho_1}}},
\end{split}
\]
where $A:\RR^2 \to \RR^2$ is a suitable rotation. From the estimate~\eqref{eq:C1a-estimate} we get 
\[
\begin{split}
\rho_1^{1 + \alpha}[\nabla u]_{\alpha, \frac{\alpha}{2}; \bP_{\frac{3\rho_1}{16}, R}(x_2, t_0)} \leq\ & C \cdot \big( \rho_{1}^{-1}\|\nabla u\|_{2; \bP_{\frac{3\rho_1}{4}, R}(x_2, t_0)} +  \sum_{i=0}^{2}(\rho_1)^{i}\|\nabla^i u_0\|_{\infty; \bB_{\frac{3\rho_1}{4}}(x_2)\cap \Omega_R} \big)\\
\leq\ & C\cdot(1 + \|u_0\|_{2, \infty; \bB_{\frac{3\rho_1}{4}}(x_2)\cap \Omega_R} ),
\end{split}
\]
where both constants $C$ depend only on $q$ and $\Lambda_0$, and we used~\eqref{eq:point-picking-consequence-2} and the inequality $\rho_{1} < 1$, respectively, to bound the two terms in parentheses on the first line. Recalling that $(x_1, t_1) \in \bP_{\frac{\rho_1}{16}, R}(x_0, t_0)$, we infer from the previous estimate, together with the inclusions~\eqref{eq:eta-to-bound-inclusions-case-2} and the gradient bound~\eqref{eq:point-picking-consequence-2}, that
\begin{equation}\label{eq:difference-upper-bound-case-2}
|\nabla u(x_0, t_0)|^2 - |\nabla u(x_1, t_1)|^2 \leq K_2\rho_1^{-2-\alpha}  \cdot \max\{|x_0 - x_1|^{\alpha}, |t_0 - t_1|^{\frac{\alpha}{2}}\},
\end{equation}
where $K_2$ depends only on $q$, $\Lambda_0$, and $\Lambda_\partial$. Requiring further that $\eta$ satisfy 
\[
\eta < \big( \frac{1}{2(K_1 + K_2 + 1)} \big)^{\frac{3}{\alpha}},
\]
we get a contradiction upon combining~\eqref{eq:difference-upper-bound-case-1} and~\eqref{eq:difference-upper-bound-case-2} with~\eqref{eq:difference-quotient-bound}. Thus~\eqref{eq:interior-gradient-contradiction} cannot hold, which implies by~\eqref{eq:point-picking} that the desired inequality~\eqref{eq:eta-to-grad-bound} holds.

To remove the assumption $u \in C^{2 + \alpha, 1 + \frac{\alpha}{2}}(\overline{\bP_{1, R}})$, we consider, for $\lambda \in (0, 1)$ and $(x, t) \in \bP_{\lambda^{-1}, \lambda^{-1}R}$, the maps 
\[
\widetilde{u}(x, t) = u(\lambda x, \lambda^2 t),\quad \widetilde{u}_0(x) = u_0(\lambda x).
\] 
Then $\widetilde{u}$ lies in $C^{2 + \alpha, 1 + \frac{\alpha}{2}}(\overline{\bP_{1, \lambda^{-1}R}})$, and satisfies
\[
\left\{
\begin{array}{ll}
\widetilde{u}_{t} - \Delta \widetilde{u} = -2H(\widetilde{u})\widetilde{u}_{x^1} \times \widetilde{u}_{x^2},& \text{ in }(\bB \cap \Omega_{\lambda^{-1}R}) \times (-1, 0], \\
\widetilde{u}(x, t) = \widetilde{u}_0(x), & \text{ on } (\bB \cap \partial \Omega_{\lambda^{-1}R}) \times (-1, 0].
\end{array}
\right.
\]
Observe also that
\[
\| \widetilde{u}_0 \|_{2, \infty; \bB \cap \Omega_{\lambda^{-1}R}} \leq \|u_0\|_{2, \infty; \bB_{\lambda} \cap \Omega_{R}} \leq \Lambda_{\partial},
\]
and that
\[
\sup_{-1 < t \leq 0}\int_{\bB\cap \Omega_{\lambda^{-1}R}} |\nabla\widetilde{u}(\cdot, t)|^2 = \sup_{-\lambda^2 < t \leq 0}\int_{\bB_{\lambda} \cap \Omega_R} |\nabla u(\cdot, t)|^2 \leq \eta.
\]
Thus, for all $r \in (0, 1)$ and $(x, t) \in \bP_{r, R}$, noting that $\bP_{r, R} \subset \bP_{r, \lambda^{-1}R}$, and applying to $\widetilde{u}$ what we just proved, we get
\[
(1 - r)^2 \lambda^2 |\nabla u(\lambda x, \lambda^2 t)|^2\leq 16.
\]
Upon letting $\lambda \to 1^{-}$, we get~\eqref{eq:eta-to-grad-bound} since $r$ and $(x, t)$ are arbitrary.
\end{proof}
Lemma~\ref{lemm:eta-to-gradient-bound} has an interior version:
\begin{lemm}\label{lemm:eta-to-gradient-bound-interior}
Given $q > 4$ and $\Lambda_{0} > 0$, define $\alpha = 1 - \frac{4}{q}$. There exists $\eta' = \eta'(q, \Lambda_{0}) \in (0, 1)$ with the following property. Suppose $u \in C^{2 + \alpha, 1 + \frac{\alpha}{2}}(\bP_{1})$ is a solution to~\eqref{eq:H-flow-appendix}, where $H \in C^{1}(\RR^3, \RR)$ and $\|H\|_{\infty} \leq \Lambda_{0}$. Assume furthermore that 
\[
\sup_{-1 < t \leq 0}\int_{\bB}|\nabla u(\cdot, t)|^2 \leq \eta'.
\]
Then for all $r \in (0, 1)$ we have $(1-r)\cdot \sup_{\bP_{r}}|\nabla u| \leq 4$.
\end{lemm}
\begin{proof}
The proof is similar to that of Lemma~\ref{lemm:eta-to-gradient-bound}. We omit the details.
\end{proof}

To continue, we fix $q > 4$ and $\Lambda_0, \Lambda_{\partial} > 0$, set $\alpha = 1 - \frac{4}{q}$, and let $R_0 = R_0(q)$, $\eta = \eta(q, \Lambda_0, \Lambda_{\partial})$, and $\eta' = \eta'(q, \Lambda_0)$ be the constants produced respectively by Lemma~\ref{lemm:gradient-holder}, Lemma~\ref{lemm:eta-to-gradient-bound} and Lemma~\ref{lemm:eta-to-gradient-bound-interior}. 

\begin{coro}\label{coro:grad-estimate-scaled-boundary}
Given $R \in [R_0, \infty)$, $u_0 \in C^{2, \alpha}(\bB \cap \overline{\Omega_{R}})$, and $H \in C^{1}(\RR^3; \RR)$ such that the bounds~\eqref{eq:bounds-on-data} hold, suppose we have for some $T > 0$ a solution $u \in C^{2+ \alpha, 1+ \frac{\alpha}{2}}((\bB \cap \overline{\Omega_R}) \times (0, T])$ of
\[
\left\{
\begin{array}{ll}
u_{t} - \Delta u = -2H(u)u_{x^1} \times u_{x^2},& \text{ in }(\bB \cap \Omega_{R}) \times (0, T], \\
u(x, t) = u_0(x), & \text{ on } (\bB \cap \partial \Omega_{R}) \times (0, T],
\end{array}
\right.
\]
satisfying in addition that
\begin{equation}\label{eq:boundary-grad-estimate-smallness}
\sup_{0 < t \leq T}\int_{\bB \cap \Omega_R} |\nabla u(\cdot, t)|^2 \leq \min\{\eta, \eta'\}.
\end{equation}
Then for all $(x, t) \in (\bB \cap \overline{\Omega_{R}}) \times (0, T]$, we have
\begin{equation}\label{eq:grad-estimate-scaled-boundary}
\min\{t, (1 - |x|)^2\} \cdot |\nabla u(x,t)|^2 \leq C_{q, \Lambda_0} \cdot\big( \Lambda_{\partial}^2 + \sup_{0 < t \leq T} \int_{\bB \cap \Omega_R}|\nabla u(\cdot, t)|^2\big).
\end{equation}
\end{coro}
\begin{proof}
Given $(x_0, t_0) \in (\bB \cap \overline{\Omega_R}) \times (0, T]$, define
\[
\rho_0 : = \frac{1}{16}\min\{t_0^{\frac{1}{2}}, 1 - |x_0|\} \leq \frac{1}{16}.
\]
In the case where $\bB_{\rho_0}(x_0) \subset \Omega_R$, we have $\bP_{\rho_0}(x_0, t_0) \subset (\bB \cap\Omega_R) \times (0, T]$, so it makes sense to define
\[
\widetilde{u}(x, t) = u(x_0 + \rho_0 x, t_0 + \rho_0^2 t),\quad \text{for }(x, t)\in \bP_{1},
\]
which satisfies
\[
\widetilde{u}_{t} - \Delta \widetilde{u} = -2H(\widetilde{u}) \widetilde{u}_{x^1} \times \widetilde{u}_{x^2} \quad \text{on }\bP_1.
\]
By~\eqref{eq:boundary-grad-estimate-smallness}, we have
\[
\sup_{-1 < t \leq 0}\int_{\bB}|\nabla \widetilde{u}(\cdot, t)|^2 = \sup_{t_0 - \rho_0^2 < t \leq t_0}\int_{\bB_{\rho_0}(x_0)}|\nabla u(\cdot, t)|^2 \leq \eta'.
\]
Hence we get from Lemma~\ref{lemm:eta-to-gradient-bound-interior} that 
\[
\|\nabla \widetilde{u}\|_{\infty; \bP_{\frac{1}{2}}} \leq 8.
\]
This allows us to apply the estimate~\eqref{eq:C1a-interior}, which leads to
\begin{equation}\label{eq:boundary-grad-estimate-case-1}
\begin{split}
\rho_0^2 |\nabla u(x_0, t_0)|^2 \leq\ & C_{q, \Lambda_0} \cdot \rho_0^{-2}\int_{t_0 - \frac{\rho_0^2}{4}}^{t_0}\int_{\bB_{\frac{\rho_0}{2}}(x_0)}|\nabla u|^2\\
\leq\ & C_{q, \Lambda_0}\cdot \sup_{0 < t \leq T}\int_{\bB\cap \Omega_R}|\nabla u(\cdot, t)|^2.
\end{split}
\end{equation}
On the other hand, in the case where $\bB_{\rho_0}(x_0) \not\subset \Omega_R$, as in the proof of Lemma~\ref{lemm:eta-to-gradient-bound} we choose any $x_1 \in \bB_{\rho_0}(x_0) \cap \partial\Omega_R$ and observe that
\begin{equation}\label{eq:boundary-grad-case-2-inclusions}
\bB_{8\rho_0}(x_1) \times (t_0 - (8\rho_0)^2, t_0]  \subset \bB_{9\rho_0}(x_0) \times (t_0 - (9\rho_0)^2, t_0] \subset \bB \times (0, T].
\end{equation}
Since $\partial\Omega_R - x_1$ passes through the origin, there is a rotation $A:\RR^2 \to \RR^2$ such that
\[
A(\Omega_R) = -x_1 + \Omega_R.
\]
Defining, for $x \in \bB \cap \overline{\Omega_{\frac{R}{8\rho_0}}}$ and $t \in (-1, 0]$, 
\begin{align*}
\widehat{u}(x, t)  :=\ & u(x_1 + 8\rho_0  Ax, t_0 + (8\rho_0)^2 t),\\
\widehat{u}_0(x) :=\ & u_0(x_1 + 8\rho_0 Ax), 
\end{align*}
we find that 
\[
\left\{
\begin{array}{ll}
\widehat{u}_{t} - \Delta \widehat{u} = -2H(\widehat{u})\widehat{u}_{x^1} \times \widehat{u}_{x^2},& \text{ in }(\bB \cap \Omega_{\frac{R}{8\rho_0}}) \times (-1, 0], \\
\widehat{u}(x, t) = \widehat{u}_0(x), & \text{ on }(\bB \cap \partial\Omega_{\frac{R}{8\rho_0}}) \times (-1, 0].
\end{array}
\right.
\]
Since $\rho_0 < \frac{1}{8}$, we have by~\eqref{eq:bounds-on-data} that
\[
\|\widehat{u}_0\|_{2, \infty; \bB \cap \Omega_{\frac{R}{8\rho_0}}} \leq \|u_0\|_{2, \infty; \bB_{8\rho_0}(x_1) \cap \Omega_R} \leq \Lambda_{\partial}.
\]
Also, by~\eqref{eq:boundary-grad-estimate-smallness} we have
\[
\sup_{-1 < t \leq 0}\int_{\bB \cap \Omega_{\frac{R}{8\rho_0}}} |\nabla \widehat{u}(\cdot, t)|^2 =  \sup_{t_0 - (8\rho_0)^2 < t \leq t_0}\int_{\bB_{8\rho_0}(x_1) \cap \Omega_R}|\nabla u(\cdot, t)|^2 \leq \eta.
\]
Lemma~\ref{lemm:eta-to-gradient-bound} can then be used to get 
\[
\|\nabla \widehat{u}\|_{\infty; (\bB_{\frac{1}{2}} \cap \Omega_{\frac{R}{8\rho_0}})\times (-\frac{1}{4}, 0]} \leq 8,
\]
which allows us to apply the estimate~\eqref{eq:C1a-estimate} from Lemma~\ref{lemm:gradient-holder}. As a result we have
\[
\begin{split}
\rho_0^2 \|\nabla u\|_{\infty; (\bB_{\rho_0}(x_1) \cap \Omega_R) \times (t_0 - \rho_0^2, t_0]}^2 \leq\ & C_{q, \Lambda_0} \cdot \rho_0^{-2}\int_{t_0 - (4\rho_0)^2}^{t_0}\int_{\bB_{4\rho_0}(x_1) \cap \Omega_R} |\nabla u|^2\\
& + C_{q, \Lambda_0} \cdot \big(\|u_0\|_{2, \infty; \bB_{4\rho_0}(x_1) \cap \Omega_R}\big)^2.
\end{split}
\]
By the inclusions~\eqref{eq:boundary-grad-case-2-inclusions}, and the fact that $(x_0, t_0) \in (\bB_{\rho_0}(x_1) \cap \overline{\Omega_R}) \times (t_0 - \rho_0^2, t_0]$, we deduce from the above that
\begin{equation}\label{eq:boundary-grad-estimate-case-2}
\rho_0^2 |\nabla u(x_0, t_0)|^2 \leq C_{q, \Lambda_0}\big( \Lambda_{\partial}^2 + \sup_{0 < t \leq T}\int_{\bB \cap \Omega_R}|\nabla u(\cdot, t)|^2 \big)
\end{equation}
The asserted estimate~\eqref{eq:grad-estimate-scaled-boundary} now follows from~\eqref{eq:boundary-grad-estimate-case-2} and~\eqref{eq:boundary-grad-estimate-case-1}. 
\end{proof}
Like Lemma~\ref{lemm:eta-to-gradient-bound}, there is an interior version of Corollary~\ref{coro:grad-estimate-scaled-boundary}. Once again we record the statement only.
\begin{coro}\label{coro:grad-estimate-scaled}
Given $q > 4$ and $\Lambda_{0} > 0$, define $\alpha = 1 - \frac{4}{q}$ as before and let $\eta' = \eta'(q, \Lambda_{0})$ be the threshold given by Lemma~\ref{lemm:eta-to-gradient-bound-interior}. Suppose for some $T > 0$ that $u \in C^{2+\alpha, 1 + \frac{\alpha}{2}}(\bB \times (0, T])$ is solution of
\[
u_t - \Delta u = -2H(u) u_{x^1} \times u_{x^2} ,\quad \text{on }\bB \times (0, T],
\]
where $H \in C^{1}(\RR^3, \RR)$ and $\|H\|_{\infty} \leq \Lambda_{0}$. Assume in addition that
\begin{equation}\label{eq:grad-estimate-scaled-smallness}
\sup_{0 < t \leq T} \int_{\bB}|\nabla u(\cdot, t)|^2 \leq \eta'.
\end{equation}
Then we have for all $(x, t) \in \bB \times (0, T]$ that 
\begin{equation}\label{eq:grad-estimate-scaled}
\min\{t, (1 - |x|)^2\} \cdot |\nabla u(x, t)|^2 \leq C_{q, \Lambda_{0}}\cdot \sup_{0 < t \leq T}\int_{\bB}|\nabla u(\cdot, t)|^2.
\end{equation}
\end{coro}
\begin{proof}
We simply follow the argument leading up to~\eqref{eq:boundary-grad-estimate-case-1} in the previous proof. The details are omitted.
\end{proof}
\section{Boundary continuity and a uniqueness result}\label{sec:boundary-continuous-and-uniqueness}
The two results we recall in this appendix are essentially contained in~\cite{Qing1993} and~\cite{WangGuoFang1992}, respectively. The first, when used in conjunction with Corollary~\ref{coro:grad-estimate-scaled}, addresses the boundary behavior of solutions to~\eqref{eq:H-system} and~\eqref{eq:H-flow}, while the second is a uniqueness result for~\eqref{eq:H-system} subject to constant Dirichlet data, which generalizes the work of Wente~\cite{Wente1975}.

For the first result, we consider the class $\mathscr{C}$ of maps $u \in W^{1, 2} \cap C^{1}_{\loc}(\bB; \RR^3)$ such that
\begin{equation}\label{eq:Hardy-bound}
\sup_{x \in \bB \setminus \bB_{r}}(1 - |x|)|\nabla u(x)| \longrightarrow 0\quad\text{as }r \to 1^{-},
\end{equation}
and that the trace $u|_{\partial\bB}$ is continuous on $\partial\bB$. 
\begin{lemm}[\cite{Qing1993}, Proposition 1]
\label{lemm:boundary-continuity}
Let $\mathscr{K}$ be a subset of $\mathscr{C}$. Assume further that $\mathscr{K}$ is pre-compact in $W^{1, 2}(\bB; \RR^3)$, that the convergence~\eqref{eq:Hardy-bound} is uniform with respect to $u \in \mathscr{K}$, and that $\{u|_{\partial\bB}\}_{u \in \mathscr{K}}$ is pre-compact in $C^0(\partial \bB; \RR^3)$. Then for all $\ep > 0$, there exists $\delta > 0$ such that 
\[
\big|u(x) - (u|_{\partial\bB})(p)\big| < \ep,
\]
for all $u \in \mathscr{K}$, and for all $p \in \partial \bB$ and $x\in \bB$ satisfying $|x - p| < \delta$.
\end{lemm}
\begin{proof}
Suppose not. Then we get some $\ep > 0$ along with sequences $u_n \in \mathscr{K}$, $x_n \in \bB$, and $p_n \in \partial \bB$, such that $\lim_{n \to \infty}|x_n - p_n| = 0$, and that, with $\varphi_n : = u_n|_{\partial \bB}$
\begin{equation}\label{eq:not-equicontinuous-at-bdy}
|u_n(x_n) - \varphi_n(p_n)| \geq \ep, \quad\text{for all }n.
\end{equation}
Up to taking a subsequence, we obtain some $p_0 \in \partial \bB$ such that
\[
\lim_{n \to \infty}x_n = p_0 = \lim_{n \to \infty} p_n,
\]
and can also assume that $u_n$ converges strongly in $W^{1,2}(\bB)$, while $\varphi_n$ converges uniformly on $\partial\bB$. Thus, with $\lambda_0 \in (0, \frac{1}{20})$ being a universal constant to be determined, there exists $\delta \in (0, \frac{1}{4})$ such that, for all $n$,
\begin{equation}\label{eq:consequence-of-compactness}
\int_{\bB \setminus \bB_{1 - \delta}}|\nabla u_n|^2\, dx + \sup_{p \in  \bB_{\delta}(p_0) \cap \partial \bB} |\varphi_n(p) - \varphi_n(p_0)|^2 < (\lambda_0 \ep)^2.
\end{equation}
Since~\eqref{eq:Hardy-bound} is assumed to hold uniformly over $\mathscr{K}$, upon choosing a smaller $\delta$ if necessary, we can further arrange that
\begin{equation}\label{eq:consequence-of-uniform-Hardy-bound}
\sup_{x \in \bB \setminus \bB_{1 - \delta}} (1 -  |x|)|\nabla u_n(x)| < \lambda_0 \ep,\quad\text{for all }n.
\end{equation}
By~\eqref{eq:consequence-of-compactness} and~\eqref{eq:not-equicontinuous-at-bdy}, as soon as $n$ is large enough so that $|p_n - p| < \delta$, we have
\begin{equation}\label{eq:not-equicontinuous-at-bdy-1}
|u_n(x_n) - \varphi_n(p_0)| \geq (1 - \lambda_0)\ep.
\end{equation}

Next, we regard $\RR^2$ as the complex plane, and assume without loss of generality that $p_0 = -1$. Consider the map $F:(\RR^2_+, \partial \RR^2_+) \to (\bB, \partial \bB\setminus \{1\})$ given by 
\[
F(z) = \frac{z-i}{z+i}\,,
\]
which takes $\RR^2_+$ biholomorphically onto $\bB$, and define
\begin{align*}
v_n(z) = u_n(F(z)), \quad \psi_n(z) = \varphi_n(F(z)),
\end{align*}
for $z \in \RR^2_+$ and $z \in \partial \RR^2_+$, respectively. Then there is $r_0 > 0$ such that
\begin{equation}\label{eq:F-r0-inclusions}
|F(z) + 1| < \frac{\delta}{2}\quad\text{whenever }|z| < r_0,
\end{equation}
so the integral bound in~\eqref{eq:consequence-of-compactness} and the conformality of $F$ implies that
\begin{equation}\label{eq:integral-bound-on-Br-+}
\int_{\bB_{r_0}^+} |\nabla v_n|^2\, dx  = \int_{F(\bB_{r_0}^+)} |\nabla u_n|^2\, dx < (\lambda_0\ep)^2.
\end{equation}
Moreover, recalling that $F$ is an isometry from $(\RR^2_+, \frac{|dz|^2}{(\im z)^2})$ to $(\bB, \frac{4|dz|^2}{(1 - |z|^2)^2})$, we infer from~\eqref{eq:consequence-of-uniform-Hardy-bound} and~\eqref{eq:F-r0-inclusions} that, for all $z\in \bB_{r_0}^+$,
\begin{equation}\label{eq:Hardy-bound-on-H}
\begin{split}
(\im z) |\nabla v_n(z)|=\ & \frac{(1 - |F(z)|^2)}{2} \cdot|\nabla u_n (F(z))|\\
\leq\ & (1 - |F(z)|)\cdot |\nabla u_n(F(z))|<  \lambda_0 \ep.
\end{split}
\end{equation}

To continue, we fix a sufficiently large $N \in \NN$ such that~\eqref{eq:not-equicontinuous-at-bdy-1} holds, and that 
\[
F^{-1}(x_N) \in \bB_{\frac{r_0}{4}}^+.
\]
Writing $z_0$ for $F^{-1}(x_N)$, and dropping the subscripts from $v_N$, $u_N$, $\varphi_N$, and $\psi_N$, we see that~\eqref{eq:not-equicontinuous-at-bdy-1} translates into
\begin{equation}\label{eq:not-equicontinuous-at-bdy-2}
|v(z_0) - \psi(0)| \geq (1 - \lambda_0)\ep.
\end{equation}
Also, noting that $v$ lies in $W^{1,2}(\bB_{R}^+)$ for all $R > 0$, we obtain an element of $W^{1, 2}_{\loc}(\RR^2)$ upon extending it by even reflection across $\partial\RR^2_{+}$ as in the proof of~\cite[Proposition 1]{Qing1993}. Denoting the extended map still by $v$, we have by~\eqref{eq:integral-bound-on-Br-+} that
\begin{equation}\label{eq:energy-bound-after-reflecting}
\int_{\bB_{r_0}}|\nabla v|^2 = 2\int_{\bB_{r_0}^{+}}|\nabla v|^2 < 2(\lambda_0\ep)^2.
\end{equation}
With $(r, \theta)$ denoting polar coordinates centered at $z_0$, there is a subset $E \subset (0, \frac{r_0}{2})$ with zero measure such that for all $r \in (0, \frac{r_0}{2}) \setminus E$, the following are true:
\begin{enumerate}
\item[(i)] $v(r, \cdot)$ and $v_{\theta}(r, \cdot)$ both lie in $L^2(S^1; \RR^3)$.
\vskip 1mm
\item[(ii)]  $v(r, \cdot)$ is absolutely continuous up to redefinition on a measure-zero subset of $S^1$, and 
\begin{equation}\label{eq:oscillation-by-W12}
\osc\limits_{S^1}v(r, \cdot) \leq \Big(2\pi \int_{S^1}|v_{\theta}(r, \theta)|^2\, d\theta\Big)^{\frac{1}{2}}.
\end{equation}
\vskip 1mm
\item[(iii)] If $r \in (\im z_0, \frac{r_0}{2})$, then at the points of intersection of $\partial \bB_{r}(z_0)$ with $\partial\RR^{2}_{+}$, the maps $v$ and $\psi$ agree. Note from~\eqref{eq:F-r0-inclusions} that
\[
F(\partial \bB_{r}(z_0) \cap \partial\RR^2_+) \subset \bB_{\frac{\delta}{2}}(-1) \cap \partial\bB,
\]
so~\eqref{eq:consequence-of-compactness} gives
\begin{equation}\label{eq:boundary-value-close}
|\psi(z) - \psi(0)| < \lambda_0\ep, \quad\text{whenever }z \in \partial\bB_{r}(z_0) \cap \partial \RR^2_+.
\end{equation}
\end{enumerate}
Letting $\rho = \im z_0$, then again as in the proof of~\cite[Proposition 1]{Qing1993}, by Fubini's theorem and the energy bound~\eqref{eq:energy-bound-after-reflecting}, we obtain $r_1 \in (\frac{\rho}{4}, \frac{\rho}{2}) \setminus E$ and $r_2 \in (\rho, 2\rho) \setminus E$ such that
\begin{equation}\label{eq:compactness-bound-from-Fubini}
\int_{S^1}|v_{\theta}(r_{1}, \theta)|^2 \, d\theta + \int_{S^1}|v_{\theta}(r_{2}, \theta)|^2 \, d\theta < 32\cdot (\lambda_0\ep)^2.
\end{equation}
Fixing an arbitrary $\theta_0 \in S^1$, then from~\eqref{eq:oscillation-by-W12} and~\eqref{eq:compactness-bound-from-Fubini} we get
\begin{equation}\label{eq:oscillation-for-r1}
\sup_{\theta \in S^{1}}|v(r_1, \theta) - v(r_1, \theta_0)| < C \cdot \lambda_0 \ep,
\end{equation}
where $C$ is a universal constant. Similarly, by item (iii) above, along with~\eqref{eq:oscillation-by-W12} and~\eqref{eq:compactness-bound-from-Fubini}, we get for some other universal constant $C$ such that
\begin{equation}\label{eq:sup-for-r2}
\sup_{\theta \in S^{1}}|v(r_2, \theta) - \psi(0)| \leq \lambda_0\ep + \osc_{S^1}v(r_2, \cdot) < C\cdot \lambda_0\ep.
\end{equation}
Also, noting that $\im z \geq \frac{\rho}{2}$ for all $z \in \bB_{\frac{\rho}{2}}(z_0) \subset \bB_{r_0}^+$, we deduce from~\eqref{eq:Hardy-bound-on-H} and~\eqref{eq:not-equicontinuous-at-bdy-2} that
\begin{equation}\label{eq:lower-bound-captured}
\begin{split}
|\psi(0) - v(r_1, \theta_0)| \geq\ & |\psi(0) - v(z_0)| - |v(r_1, \theta_0) - v(z_0)|\\
>\ & (1 - 2\lambda_0)\ep.
\end{split}
\end{equation}
Now define $h:\bB_{r_2}(z_0) \setminus \bB_{r_1}(z_0) \to \RR^3$ by
\[
h(r, \theta) =  \frac{\log (r/r_1)}{\log(r_2/r_1)} \cdot (v(r_2, \theta) -  \psi(0))  +  \frac{\log (r_2/r)}{\log(r_2/r_1)}\cdot (v(r_1, \theta) - v(r_1, \theta_0)).
\]
Then by~\cite[Lemma 1]{Qing1993} along with~\eqref{eq:lower-bound-captured} we have
\[
(1 - 2\lambda_0)\ep \leq C\| \nabla (v-h) \|_{2; \bB_{r_2}(z_0) \setminus \bB_{r_1}(z_0)}.
\]
Estimating the energy of $h$ with the help of~\eqref{eq:compactness-bound-from-Fubini} through~\eqref{eq:sup-for-r2}, and recalling~\eqref{eq:energy-bound-after-reflecting}, we arrive at the following analogue of the first displayed equation on~\cite[page 464]{Qing1993}:
\[
(1 - 2\lambda_0) \ep \leq C\cdot \lambda_0\ep.
\]
A contradiction results provided $\lambda_0$ is sufficiently small, and we are done.
\end{proof}
\begin{lemm}[\cite{WangGuoFang1992}, Lemma 4.2]
\label{lemm:Wente-constant}
Suppose $H \in C^1(\RR^3; \RR)$. Given any $\xi \in \RR^3$, if $v \in C^{2}(\overline{\RR^2_+})$ is a solution to
\begin{equation}\label{eq:CMC-constant-boundary}
\left\{
\begin{array}{ll}
\Delta v = 2H(v) v_{x^1} \times v_{x^2}, &  \text{ on }\RR^2_+ ,\\
v = \xi, & \text{ on }\partial\RR^2_{+},
\end{array}
\right.
\end{equation}
satisfying in addition that
\begin{equation}\label{eq:CMC-finite-energy}
\int_{\RR^2_+}|\nabla v|^2 < \infty,
\end{equation}
then $v$ must be constant.
\end{lemm}
\begin{proof} 
With $\partial_{z} = \frac{1}{2}(\partial_{x^1} - {\rm i}\partial_{x^2})$ and $\partial_{\overline{z}} = \frac{1}{2}(\partial_{x^1} + {\rm i}\partial_{x^2})$ as usual, we let
\[
f: = v_{z} = \frac{1}{2}(v_{x^1} - {\rm i} v_{x^2}),
\]
and also define $\omega:\RR^2_{+} \to \CC^{3 \times 3}$ by
\[
\omega_{jk} = -{\rm i} H(v) \vol_{\RR^3}(\be_{j}, \be_{k}, v_{\overline{z}}),\quad\text{for }j, k \in \{1, 2, 3\}.
\]
Then, a direct computation using~\eqref{eq:CMC-constant-boundary} shows that
\begin{equation}\label{eq:CMC-complex}
\left\{
\begin{array}{ll}
f_{\overline{z}} = \omega  f, & \text{ on }\RR^2_+,\\
\re f = 0, & \text{ on }\partial \RR^2_+,
\end{array}
\right.
\end{equation}
where by $\re f$ we mean taking the real part component-wise. Since $\omega_{jk} = -\omega_{kj}$, we infer that $f \cdot f$ is holomorphic on $\RR^2_{+}$. Noting also that 
\[
\im(f\cdot f) = 2 (\re f) \cdot (\im f) = 0 \quad\text{on }\partial\RR^2_+, 
\]
we deduce from the Schwarz reflection principle, and the finiteness of $\int_{\RR^2_+}|f \cdot f|$ coming from~\eqref{eq:CMC-finite-energy}, that $f\cdot f$ vanishes identically on $\RR^2_{+}$, which together with the boundary condition in~\eqref{eq:CMC-complex} gives
\begin{equation}\label{eq:v-weakly-conformal}
f = 0 \quad\text{ on }\partial \RR^2_+.
\end{equation}

To continue, we extend $f$ and $\omega$ to $\RR^2$ by letting
\[
F(z) = \left\{
\begin{array}{ll}
f(z), & \text{ if }\im z \geq 0\\[2pt]
-\overline{f(\overline{z})} & \text{ if }\im z < 0
\end{array}
\right.; \quad\quad \Omega(z) = \left\{
\begin{array}{ll}
\omega(z), & \text{ if }\im z \geq 0\\[2pt]
\overline{\omega(\overline{z})}, & \text{ if }\im z < 0
\end{array}
\right..
\]
Note that $\Omega$ is locally bounded while $F$ is locally Lipschitz on $\RR^2$, and that
\[
F_{\overline{z}} = \Omega F,\quad\text{in distribution sense.}
\]
Now, fix any $p > 2$. For all $R >0$, since $\Omega \in L^{p}(\bB_{R}; \CC^{3 \times 3})$, we can find some $g \in W^{1, p}(\bB_{R}; \CC^{3 \times 3})$ such that
\[
g_{\overline{z}} = \Omega \quad\text{on }\bB_{R}.
\]
Then the product $h: = e^{-g}F$ is again of class $W^{1, p}$ on $\bB_{R}$, and satisfies the Cauchy-Riemann equations distributionally. It follows that $h$ is holomorphic on $\bB_{R}$ in the classical sense, so the condition~\eqref{eq:v-weakly-conformal} forces it, and consequently $F$, to vanish identically on $\bB_{R}$. The radius $R$ being arbitrary, we conclude that $f \equiv 0$ on $\RR^2_+$, so that $v$ is constant, as asserted.
\end{proof}

\section{The Wente inequality under Neumann condition}\label{sec:Neumann-Wente}
The following result, needed in the proof of Theorem~\ref{thm:H-convexity}, is a special case of~\cite[Lemma A.6]{DaLioPalmurellaRiviere2020}. A related result is~\cite[Theorem A.4]{Schikorra2018}.

\begin{lemm}[{\cite[Lemma A.6]{DaLioPalmurellaRiviere2020}}]
\label{lemm:Neumann-Wente}
Suppose $a \in W^{1, 2}_{0}(\bB)$ and $b \in W^{1, 2}(\bB)$. Then there exists a unique $\psi \in C^0(\overline{\bB})\cap W^{1, 2}(\bB)$ such that $\int_{\bB}\psi\, dx = 0$ and that
\begin{equation}\label{eq:Neumann-Wente-PDE-weak-form}
\int_{\bB} \bangle{\nabla\psi, \nabla\zeta} dx = -\int_{\bB}(a_{x^1}b_{x^2} - a_{x^2}b_{x^1})\zeta dx,\quad\text{for all }\zeta \in C^1(\overline{\bB}).
\end{equation}
Moreover, this $\psi$ satisfies the estimate
\begin{equation}\label{eq:Neumann-Wente}
\|\psi\|_{\infty; \bB} + \|\nabla \psi\|_{2; \bB} \leq C\|\nabla a\|_{2;\bB} \|\nabla b\|_{2; \bB},
\end{equation}
where $C$ is a universal constant.
\end{lemm}
\begin{rmk}\label{rmk:Neumann-Wente}
We note the following concerning the statement of Lemma~\ref{lemm:Neumann-Wente}.
\vskip 1mm
\begin{enumerate}
\item[(1)] The class of admissible test functions in~\eqref{eq:Neumann-Wente-PDE-weak-form} can be extended to $W^{1, 2} \cap L^{\infty}(\bB)$, since for any $\zeta$ in this intersection, there is a sequence $(\zeta_n)$ in $C^{1}(\overline{\bB})$ that converges strongly in $W^{1, 2}$ and pointwise a.e. to $\zeta$, and satisfies $\|\zeta_n\|_{\infty; \bB} \leq \|\zeta\|_{\infty; \bB}$ for all $n$. (See the proof of~\cite[Chapter 5.3.3, Theorem 3]{Ev}.)
\vskip 1mm
\item[(2)] The assumption that $a = 0$ on $\partial \bB$ is essential, not only for the solvability of the Neumann problem~\eqref{eq:Neumann-Wente-PDE-weak-form}, but also for the estimate~\eqref{eq:Neumann-Wente}. That is, the Wente inequality does not generally hold under Neumann boundary conditions. Counterexamples have been constructed by Hirsch \cite{Hirsch2019} and Da Lio--Palmurella \cite{DaLioPalmurella2017}.
\vskip 1mm
\item[(3)] We appear to be claiming in~\eqref{eq:Neumann-Wente} a stronger $L^{\infty}$-estimate than in~\cite[equation (A.5)]{DaLioPalmurellaRiviere2020}. Note however that, under the condition $\int_{\bB}\psi\, dx = 0$, there holds
\begin{equation}\label{eq:adjust-by-constant}
\|\psi\|_{\infty; \bB} \leq C\cdot\big( \inf_{l \in \RR}\|\psi - l\|_{\infty; \bB} + \|\nabla\psi\|_{2; \bB} \big),
\end{equation}
where again $C$ is a universal constant. Thus the statement of Lemma~\ref{lemm:Neumann-Wente} is consistent with~\cite[Lemma A.6]{DaLioPalmurellaRiviere2020}. To see~\eqref{eq:adjust-by-constant}, let $\ep >0$ and choose $l_0$ such that 
\[
\|\psi - l_0\|_{\infty; \bB} < \ep + \inf_{l\in \RR}\|\psi - l\|_{\infty; \bB}.
\]
Then we have
\[
\begin{split}
|l_0|^2 \leq\ & 2\fint_{\bB}|l_0 - \psi|^2 + 2\fint_{\bB}|\psi|^2\leq 2 \big( \ep + \inf_{l \in \RR}\|\psi -l\|_{\infty; \bB} \big)^2 + C\int_{\bB}|\nabla \psi|^2,
\end{split}
\]
where we used $\int_{\bB}\psi\, dx = 0$ and the Poincar\'e inequality for the second step. Consequently
\[
\|\psi\|_{\infty; \bB} \leq |l_0| + \ep + \inf_{l \in \RR}\|\psi - l\|_{\infty; \bB} \leq C(\ep + \inf_{l \in \RR}\|\psi - l\|_{\infty; \bB} + \|\nabla \psi\|_{2; \bB}),
\]
and we get~\eqref{eq:adjust-by-constant} upon letting $\ep \to 0$.
\end{enumerate}
\end{rmk}
\begin{proof}[Proof of Lemma~\ref{lemm:Neumann-Wente}]
We include a proof for the reader's convenience. The idea of extending $a$ by odd reflection across $\partial \bB$, and $b$ by even reflection, is taken from~\cite[page 1498]{DaLioPalmurella2017}. 

Uniqueness is standard. To establish existence and the estimate~\eqref{eq:Neumann-Wente}, assume first that $a \in C^{\infty}_{c}(\bB)$ and that $b \in C^{\infty}(\overline{\bB})$. Integration by parts shows that $a_{x^1}b_{x^2} - a_{x^2}b_{x^1}$ averages to zero on $\bB$, and thus standard theory gives a unique $\psi \in C^{\infty}(\overline{\bB})$ satisfying 
\begin{equation}\label{eq:Wente-Neumann-problem}
\left\{
\begin{array}{ll}
\Delta \psi = a_{x^1}b_{x^2} - a_{x^2}b_{x^1}, & \text{ in }\bB,\\
\partial_{\nu}\psi = 0 , & \text{ on }\partial \bB,\\
\int_{\bB}\psi\, dx = 0.
\end{array}
\right.
\end{equation}
Letting ${I}: \RR^2\setminus\{0\} \to \RR^2\setminus\{0\}$ denote the inversion map, that is,
\[
{I}(x) = |x|^{-2}x,
\]
we introduce 
\[
\overline{\psi}(x) = \left\{
\begin{array}{ll}
\psi(x), & \text{ if }|x| \leq 1,\\
(\psi\circ{I})(x),& \text{ if }|x| > 1,
\end{array}
\right., \quad
\overline{a}(x) = \left\{
\begin{array}{ll}
a(x), & \text{ if }|x| \leq 1,\\
-(a\circ{I})(x),& \text{ if }|x| > 1,
\end{array}
\right.,
\]
and also define $\overline{b}$ in a way similar to $\overline{a}$, but without the minus sign in the case $|x| > 1$. Then $\overline{\psi}$, $\overline{a}$ and $\overline{b}$ are all continuous on $\RR^2$, and have distributional derivatives in $L^2_{\loc}(\RR^2)$. (In fact, $\overline{a}$ is smooth, $\overline{\psi}$ is $C^1$, and $\overline{b}$ is Lipschitz.) Further, by~\cite[Remark 3.1.1]{Helein2002}, or by direct computation, we have 
\begin{equation}\label{eq:PDE-for-psi-bar}
\Delta\overline{\psi} = \overline{a}_{x^1}\overline{b}_{x^2} - \overline{a}_{x^2}\overline{b}_{x^1}, \quad\text{ in the classical sense on }\bB \cup (\RR^2 \setminus \overline{\bB}). 
\end{equation}
Then, using the boundary condition in~\eqref{eq:Wente-Neumann-problem}, it is not hard to verify that the above holds weakly on all of $\RR^2$; that is,
\begin{equation}\label{eq:distributional-solution-on-R2}
\int_{\RR^2} \bangle{\nabla \overline{\psi}, \nabla\zeta} = -\int_{\RR^2} (\overline{a}_{x^1}\overline{b}_{x^2} - \overline{a}_{x^2}\overline{b}_{x^1})\zeta,\quad\text{for all }\zeta \in C^1_{c}(\RR^2).
\end{equation}

Next, by Fubini's theorem, we find $r_0 \in (1, 2)$ such that 
\begin{equation}\label{eq:W12-bound-from-Fubini}
\int_{\partial \bB_{r_0}}|\nabla \overline{\psi}|^2 \leq 3\int_{\bB_{2}\setminus \bB}|\nabla\overline{\psi}|^2 = 3\int_{\bB \setminus \bB_{\frac{1}{2}}} |\nabla \psi|^2,
\end{equation}
\begin{equation}\label{eq:L2-bound-from-Fubini}
\begin{split}
\int_{\partial \bB_{r_0}}|\overline{\psi}|^2 \leq 3\int_{\bB_{2}\setminus \bB}|\overline{\psi}|^2 = 3\int_{\bB\setminus \bB_{\frac{1}{2}}} |x|^{-4} |\psi|^2.
\end{split}
\end{equation}
As $\overline{\psi}$ is smooth near $\partial\bB_{r_0}$, there is a unique smooth harmonic function $h$ on $\bB_{r_0}$ such that
\[
h = \overline{\psi} \quad \text{on }\partial \bB_{r_0}.
\]
With the help of~\eqref{eq:distributional-solution-on-R2}, we can apply the Wente inequality for Dirichlet problems (see~\cite[Lemma A.1]{BrezisCoron1984} or~\cite[Theorem 3.1.2]{Helein2002}) to the difference $\overline{\psi} - h$, thereby getting
\begin{equation}\label{eq:Wente-estimate-h-psi}
\|\overline{\psi} - h\|_{\infty; \bB_{r_0}} \leq C\|\nabla\overline{a}\|_{2; \bB_{r_0}} \|\nabla \overline{b}\|_{2; \bB_{r_0}} \leq 2C\|\nabla a\|_{2; \bB} \|\nabla b\|_{2; \bB}.
\end{equation}
On the other hand, with 
\[
l := \fint_{\partial \bB_{r_0}}\overline{\psi},
\]
we observe that, first of all, by~\eqref{eq:L2-bound-from-Fubini}, the integral condition in~\eqref{eq:Wente-Neumann-problem}, and Poincar\'e's inequality, we have
\begin{equation}\label{eq:bound-on-l}
\begin{split}
|l|^2 \leq \fint_{\partial \bB_{r_0}} |\overline{\psi}|^2 \leq\ & \frac{24}{\pi r_0}\int_{\bB}|\psi|^2 \leq C\int_{\bB}|\nabla \psi|^2.
\end{split}
\end{equation}
Secondly, by the maximum principle followed by the fundamental theorem of Calculus and H\"older's inequality, we have
\begin{equation}\label{eq:max-prin-estimate}
\begin{split}
\|h - l\|_{\infty; \bB_{r_0}} = \|\overline{\psi} - l\|_{\infty; \partial\bB_{r_0}} \leq\ & Cr_0^{\frac{1}{2}}\Big(\int_{\partial \bB_{r_0}}|\nabla \overline{\psi}|^2\Big)^{\frac{1}{2}}\\
\leq\ & C\|\nabla\psi\|_{2; \bB},
\end{split}
\end{equation}
where in the last step we used~\eqref{eq:W12-bound-from-Fubini} and that $r_0 < 2$. Combining~\eqref{eq:max-prin-estimate},~\eqref{eq:bound-on-l}, and~\eqref{eq:Wente-estimate-h-psi}, we get
\begin{equation}\label{eq:Neumann-Wente-infty-bound}
\|\psi\|_{\infty;\bB} \leq C\|\nabla a\|_{2; \bB} \|\nabla b\|_{2;\bB} + C\|\nabla \psi\|_{2; \bB}.
\end{equation}
Testing the differential equation in~\eqref{eq:Wente-Neumann-problem} against $\psi $ and using the above estimate leads to
\[
\begin{split}
\|\nabla\psi\|_{2; \bB}^2 \leq\ & C\|\psi\|_{\infty; \bB}\|\nabla a\|_{2; \bB} \|\nabla b\|_{2;\bB}\\
\leq\ & C\|\nabla a\|_{2; \bB}^2 \|\nabla b\|_{2;\bB}^2 + C\|\nabla \psi\|_{2; \bB}\|\nabla a\|_{2; \bB} \|\nabla b\|_{2;\bB}.
\end{split}
\]
An application of Young's inequality gives
\[
\|\nabla\psi\|_{2; \bB} \leq C\|\nabla a\|_{2; \bB} \|\nabla b\|_{2;\bB},
\]
which is the $L^2$-bound on $\nabla\psi$ asserted in~\eqref{eq:Neumann-Wente}, and the $L^{\infty}$-bound on $\psi$ follows upon substituting the above back into~\eqref{eq:Neumann-Wente-infty-bound}. This finishes the proof assuming that $a \in C^{\infty}_{c}(\bB)$ and $b \in C^{\infty}(\overline{\bB})$. 

To remove these smoothness assumptions, we let $(a_n)$ and $(b_n)$ be sequences in $C^{\infty}_{c}(\bB)$ and $C^{\infty}(\overline{\bB})$, respectively, such that
\begin{equation}\label{eq:a-b-approximated}
\lim_{n \to \infty}\|a_n  - a\|_{1, 2; \bB}  = \lim_{n \to \infty}\|b_n - b\|_{1, 2; \bB} = 0.
\end{equation}
Given $n, m \in \NN$, we denote by $\psi_n$, $\varphi_{n, m}^1$, and $\varphi_{n, m}^{2}$ the unique solution to~\eqref{eq:Wente-Neumann-problem} with the pair $(a, b)$ replaced by $(a_n, b_n)$, $(a_n - a_m, b_n)$, and $(a_m, b_n - b_m)$, respectively. Then we have
\begin{equation}\label{eq:psi-n-psi-m}
\psi_{n} - \psi_{m} = \varphi_{n, m}^{1} + \varphi_{n,m}^{2}.
\end{equation}
Since $a_n- a_m$ and $a_m$ lie in $C^{\infty}_c(\bB)$, what we just proved can be applied to $\varphi_{n, m}^{1}$ and $\varphi_{n, m}^{2}$, which combines with~\eqref{eq:psi-n-psi-m} to give
\[
\begin{split}
\|\psi_n - \psi_m\|_{\infty; \bB} + \|\nabla \psi_n - \nabla \psi_m\|_{2; \bB} \leq\ & C\|\nabla a_n - \nabla a_m\|_{2; \bB} \|\nabla b_n\|_{2;\bB}\\
&+ C\|\nabla a_m\|_{2; \bB}\|\nabla b_n - \nabla b_m\|_{2; \bB}.
\end{split}
\]
From this and~\eqref{eq:a-b-approximated}, we obtain some $\psi \in C^0(\overline{\bB}) \cap W^{1, 2}(\bB)$ such that 
\begin{equation}\label{eq:limit-appears}
\lim_{n \to \infty}\big( \|\psi_n - \psi\|_{\infty; \bB} + \|\nabla \psi_n - \nabla \psi\|_{2; \bB}  \big) = 0
\end{equation}
Since $\psi_n$ integrates to zero on $\bB$ and satisfies the estimate~\eqref{eq:Neumann-Wente}, it follows from~\eqref{eq:limit-appears} and~\eqref{eq:a-b-approximated} that the same is true of $\psi$. Finally, using~\eqref{eq:limit-appears} and~\eqref{eq:a-b-approximated} to pass to the limit in the differential equation satisfied by $\psi_n$, we get~\eqref{eq:Neumann-Wente-PDE-weak-form}. The proof is complete.
\end{proof}

\bibliographystyle{amsplain}
\bibliography{H-Surface-Flow}
\end{document}